%% file: mifnmc4.tex
\documentclass[10pt,reqno]{amsart}

\input{mathbase_paper.tex}

\title{Moduli Interpretations for Noncongruence Modular Curves}
\author{William Yun Chen}

\begin{document}
\maketitle 
\begin{abstract}
We consider the moduli of elliptic curves with $G$-structures, where $G$ is a finite 2-generated group. When $G$ is abelian, a $G$-structure is the same as a classical congruence level structure. There is a natural action of $\SL_2(\ZZ)$ on these level structures. If $\Gamma$ is a stabilizer of this action, then the quotient of the upper half plane by $\Gamma$ parametrizes isomorphism classes of elliptic curves equipped with $G$-structures. When $G$ is sufficiently nonabelian, the stabilizers $\Gamma$ are noncongruence. Using this, we obtain arithmetic models of noncongruence modular curves as moduli spaces of elliptic curves equipped with nonabelian $G$-structures. As applications we describe a link to the Inverse Galois Problem, and show how our moduli interpretations explains the bad primes for the Unbounded Denominators Conjecture, and allows us to translate the conjecture into the language of geometry and Galois theory.

\end{abstract}

\tableofcontents

\section{Introduction}
\subsection{Background and motivation} Historically, congruence subgroups $\Gamma\le\SL_2(\ZZ)$ have been extensively studied, and through them extraordinary connections between geometry, modular forms, and number theory have been found. On the other hand, the current state of understanding of noncongruence subgroups is primitive by comparison. The first systematic investigation into noncongruence subgroups was conducted by Atkin and Swinnerton-Dyer \cite{ASD71} in 1971, who, using extensive computational data, observed that the Fourier coefficients of many noncongruence cusp forms exhibit intricate arithmetic properties analogous to those enjoyed by congruence cusp forms. Their work lit the way for a number of results, including the proof of certain ASD-type congruences and the construction of Galois representations attached to noncongruence cuspforms by Scholl in \cite{Sch85}, the triviality of Hecke operators for noncongruence subgroups by Berger in \cite{Berg94}, the nonuniversality of ASD congruences by Kibelbek in \cite{Kib14}, and the settling of a case of the unbounded denominators conjecture by Li and Long in \cite{LL11}. Their work mapped out the similarities and differences between the congruence and noncongruence worlds.
For more information on their work, see the survey article \cite{Li12}.

\sgap

In this paper\footnote{The research contained in this paper is supported in part by NSF grants DMS-1101368, DMS-1414219, and DMS-1128155. Any opinions, findings and conclusions or recommendations expressed in this paper are those of the author and do not necessarily reflect the views of the National Science Foundation.} we contribute another chapter to this story. Since congruence modular curves parametrize isomorphism classes of elliptic curves equipped with certain level structures (cf. \cite{DR72,KM85}), it is natural to ask:
$$\emph{Do noncongruence modular curves also have a moduli interpretation?}$$
The central result in this paper is that the answer is a resounding yes. We achieve this by considering \emph{Teichm\"{u}ller structures of level $G$} on elliptic curves (or simply $G$-structures), where $G$ is a finite 2-generated group. Roughly speaking, if $E$ is an elliptic curve with origin $O$, then a $G$-structure on $E$ is an unramified $G$-Galois cover of $E^\circ := E - O$, which should be viewed as a $G$-Galois cover of $E$, possibly ramified above $O$. When $G$ is abelian, we recover the classical congruence level structures. If $G$ is sufficiently nonabelian, then the moduli space of elliptic curves equipped with such $G$-structures is a noncongruence modular curve. Using the \emph{congruence subgroup property} for $\Aut(F_2)$, we find that every modular curve has a moduli interpretation in this way.

\sgap

As an application, in \S5 we discuss connections with the \emph{Unbounded Denominators Conjecture} (UBD): The classical theory of modular forms tells us that the Fourier coefficients of congruence modular forms with algebraic Fourier coefficients have bounded denominators - that is, an integral multiple of the form has integral Fourier coefficients. On the other hand, it was observed in \cite{ASD71} that all explicit examples of genuinely noncongruence modular forms with algebraic coefficients have unbounded denominators. UBD asserts that in fact \emph{every} genuinely noncongruence modular form with algebraic coefficients has this property. This is confirmed in \cite{LL11} for the case where the space of cusp forms is spanned by a form with $\QQ$-rational Fourier coefficients. If true, this would imply a conjecture in the theory of vertex operator algebras which is both known and believed by physicists, which says that the graded dimension of any $C_2$-cofinite, rational vertex operator algebra over $\CC$ is a congruence modular function (c.f. \cite{Zhu96,CG99,DLM00}). In \S5, we use our perspective of moduli interpretations to explain the bad primes for UBD, and to reformulate the conjecture in terms of the arithmetic geometry of the Tate curve.

\sgap

Since punctured elliptic curves are hyperbolic, the theory of nonabelian $G$-structures lies in the intersection of Grothendieck's anabelian geometry (c.f. \cite{Pop97}) and the classical theory of elliptic curves as abelian varieties. Our results can be viewed as saying that whereas congruence subgroups capture the structure of elliptic curves as abelian varieties, noncongruence subgroups capture their structure as anabelian/hyperbolic curves (upon removing the origin).

\sgap


Since the inclusion $E^\circ\hookrightarrow E$ induces the abelianization map $F_2\stackrel{\ab}{\longrightarrow}\ZZ^2$ at the level of fundamental groups, the moduli interpretations of noncongruence modular curves gives a direct link between the anabelian and abelian sides of arithmetic geometry. Furthermore, since much of the basic language from the classical theory of modular forms, modular curves, and Galois representations remain intact, many questions originally asked about congruence objects, Galois representations, and how they interact may again be asked in this noncongruence context. 



\sgap

\subsection{Overview}
\subsubsection{Classical Congruence Level Structures}
Recall that a $\Gamma(n)$-structure on an elliptic curve $E$ is a pair $(P,Q)$ of points of order $n$ which generate $E[n]$ (c.f. \cite{KM85} \S3). Such a pair $(P,Q)$ determines an \'{e}tale $(\ZZ/n\ZZ)^2$-torsor $E\rightarrow E/E[n]$. Here, if $G$ is a finite constant group scheme, then a (connected) $G$-torsor is just a Galois cover together with the data of an isomorphism between $G$ and the Galois group. Since $E/E[n]\cong E$, this determines a $(\ZZ/n\ZZ)^2$-torsor on $E$, so we have a bijection
$$\{\Gamma(n)\text{-structures on }E \}/\cong\;\;\;\stackrel{\sim}{\leftarrow\!\rightarrow}\;\;\;\{\text{$(\ZZ/n\ZZ)^2$-torsors $X\rightarrow E$}\}/\cong$$

A $\Gamma_1(n)$-structure on $E$ is a single point $P\in E[n]$, and determines a $\ZZ/n\ZZ$-torsor $E\rightarrow E/\langle P\rangle$. By the Weil pairing, the dual isogeny $E/\langle P\rangle\rightarrow E$ is a $\mu_n$-torsor, and thus we have a bijection
$$\{\Gamma_1(n)\text{-structures on }E\}/\cong\;\;\;\stackrel{\sim}{\leftarrow\!\rightarrow}\;\;\;\{\text{$\mu_n$-torsors $X\rightarrow E$}\}/\cong$$

Under these bijections, we may think of classical congruence level structures as (connected) $G$-torsors, where $G$ is a commutative finite \'{e}tale group scheme. Thus, it is natural to generalize this to consider $G$-torsors on elliptic curves for noncommutative group schemes $G$, though in the noncommutative setting we will restrict ourselves to the case where $G$ is finite constant. We do this via the notion of \emph{Teichm\"{u}ller level structures}, which most directly parallel the classical $\Gamma_1(n),\Gamma(n)$-structures. The resulting moduli spaces are essentially a type of Hurwitz space (c.f. \cite{BR11}, \cite{Rom03}, \cite{ACV03}) parametrizing covers of elliptic curves ramified at most above one point, and are also closely related to \emph{origami curves} as studied by Lochak, Herrlich, Schmith\"{u}sen and others (c.f. \cite{Loc03}, \cite{HS09}, \cite{Her12}).


\subsubsection{Teichm\"{u}ller Level Structures ($G$-structures)}\label{summary_TLS}
Let $G$ be a finite 2-generated group and $E$ an elliptic curve over a scheme $S$, where $|G|$ is invertible on $S$. Let $E^\circ := E - \{O\}$, where $O$ is the origin of its group law, then in \S\ref{section_setup} we will construct \emph{Teichm\"{u}ller structures of level $G$} (or simply \emph{$G$-structures}) on $E/S$. These structures were first considered by Deligne and Mumford in \cite{DM69}, and then later by Pikaart and de Jong \cite{PJ08}, Abramovich-Corti-Vistoli \cite{ACV03}, and Bertin-Romagny (\cite{BR11},\cite{Rom03}) in the case of proper curves. We will show that if $S = \Spec K$ is a field, then $G$-structures on $E/K$ correspond to isomorphism classes of geometrically connected $G$-torsors $X_{\ol{K}}^\circ$ over $E_{\ol{K}}^\circ$ whose field of moduli is contained in $K$ --- that is to say, for any $\sigma\in\Gal(\ol{K}/K)$, $(X_{\ol{K}}^\circ)^\sigma\cong X_{\ol{K}}^\circ$ over $E^\circ$. If $G$ has trivial center, then the field of moduli is a field of definition, and hence $G$-structures in this case correspond to isomorphism classes of $G$-torsors $X^\circ$ over $E^\circ$. If $G$ is abelian, then we recover classical congruence level structures (\ref{prop_abelian_is_congruence}).

\gap

Combinatorially, by the Galois correspondence, a $G$-structure is represented by an \emph{exterior surjection} --- that is, an element of
$$\Hom^\surext(F_2,G) := \Surj(F_2,G)/\Inn(G)$$
Here, $F_2$ is the free group on two generators, which should be viewed as the topological fundamental group of a punctured elliptic curve over $\CC$.


\sgap

The category of elliptic curves equipped with a $G$-structure forms a stack $\mM(G)$ over $\ZZ[1/|G|]$. 
Let $\mM(1)$ be the moduli stack of elliptic curves. Forgetting the level structure gives a morphism
$$p : \mM(G)\longrightarrow \mM(1)_{\ZZ[1/|G|]}$$
which is finite \'{e}tale. Thus $\mM(G)$ is a separated Deligne-Mumford stack and admits a smooth coarse moduli scheme $M(G)$ over $\ZZ[1/|G|]$. An elliptic curve $E_0/\Qbar$ determines a geometric point $x_0\in \mM(1)_{\Qbar}$, and the fiber $p^{-1}(x_0)$ is precisely the set $\Hom^\surext(F_2,G)$. By a result of Oda \cite{Oda97}, the fundamental group of $\mM(1)_{\Qbar}$ is $\widehat{\SL_2(\ZZ)}$, and so we have a monodromy action of $\pi_1(\mM(1)_{\Qbar})\cong\widehat{\SL_2(\ZZ)}$ on $p^{-1}(x_0) = \Hom^\surext(F_2,G)$.

\sgap

Since $\Out(F_2)\cong\GL_2(\ZZ)$, there is a natural action of $\SL_2(\ZZ)\subset\Out(F_2)$ on $p^{-1}(x_0) = \Hom^\surext(F_2,G)$ via its action on $F_2$. Indeed, this is precisely the action of $\pi_1(\mM(1)_{\Qbar})$ on $p^{-1}(x_0)$, and thus the Galois correspondence, combined with the Riemann existence theorem for stacks (c.f. \cite{Noo05}) tells us the following

\begin{thm*}[\ref{cor_G_to_Gamma}] With notation as above,
\begin{itemize}
\item[1.] The connected components of $\mM(G)_{\Qbar}$ are in bijection with the orbits of $\SL_2(\ZZ)$ on the set $\Hom^\surext(F_2,G)$.
\item[2.] For $[\varphi]\in\Hom^\surext(F_2,G)$ the component $Y(\Gamma_{[\varphi]})_{\Qbar}$ of $M(G)_{\Qbar}$ containing $[\varphi]$ is isomorphic over $\CC$ to the modular curve $\hH/\Gamma_{[\varphi]}$, where $\Gamma_{[\varphi]} := \Stab_{\SL_2(\ZZ)}([\varphi])$.
\end{itemize}
\end{thm*}

Roughly speaking, this says that upon fixing a 2-generated group $G$, the moduli space of elliptic curves equipped with a $G$-structure may not be connected. Combinatorially the connected components are in bijection with the orbits of $\SL_2(\ZZ)$ on $\Hom^\surext(F_2,G)$. Geometrically, the components will correspond to various discrete invariants associated to $G$-Galois covers of elliptic curves ramified above $O$ - for example, the genus of the cover (equivalently the ramification index above $O$, or more finely the conjugacy class of the inertia subgroup at $O$). Unfortunately we do not know of a list of invariants that together completely classify the components of $\mM(G)_{\Qbar}$ for any 2-generated group $G$.

\sgap

While this paper owes its existence to the failure of the congruence subgroup property for $\SL_2(\ZZ)$, it is strengthened by Asada's result establishing the \emph{congruence subgroup property for} $\Aut(F_2)$ (c.f. \cite{Asa01}), which immediately implies the following result.

\begin{thm*}[\ref{cor_CSP}] Every modular curve --- that is, $\hH/\Gamma$ where $\Gamma\le\SL_2(\ZZ)$ is finite index, is a quotient of a component $Y(\Gamma_{[\varphi]})_\CC$ of some $M(G)_{\CC}$.
\end{thm*}

The fact that every modular curve is a quotient of some component of $\mM(G)$ implies that every modular curve, congruence or noncongruence has a moduli interpretation parametrizing elliptic curves equipped with equivalence classes of $\Gamma_{[\varphi]}$-structures for some suitable $[\varphi]$. In particular, this allows us to speak of $\Gamma$-structures, where $\Gamma\le\SL_2(\ZZ)$ is any subgroup of finite index. This result is an analog of the main result of Diaz, Donagi, and Harbater in \cite{DDH89}, \emph{Every Curve is a Hurwitz Space}, and also a result of Ellenberg and McReynolds in \cite{EM11}. 

\sgap

Since by Belyi's theorem, every smooth projective curve over a number field can be realized as a modular curve, this implies that in some sense every arithmetic question about curves can be phrased in terms of $G$-structures on elliptic curves.

\sgap

It's worth remarking that given a subgroup $\Gamma\le\SL_2(\ZZ)$, there may be (possibly infinitely) many finite groups $G$ for which $\hH/\Gamma$ appears as a component of $M(G)_\CC$ - ie, $\hH/\Gamma$ may have many moduli interpretations, each of which corresponding to an arithmetic model of $\hH/\Gamma$ over a number field. It turns out these are all compatible, in a certain sense:

\begin{thm*}[\ref{cor_nice}] For any finite index $\Gamma\le\SL_2(\ZZ)$, there exists a number field $K$ such that all moduli-theoretic models of $\hH/\Gamma$ can be defined over $K$, and over $K$ are all isomorphic to $\Spec M_0(\Gamma,K)$, where $M_0(\Gamma,K)$ is the $K$-algebra of weight 0 modular forms for $\Gamma$, holomorphic on $\hH$, possibly meromorphic at the cusps, with Fourier coefficients in $K$.
\end{thm*}

Here, the field $K$ is essentially a small extension of the field of definition of the cusp $i\infty$ of any particular moduli-theoretic model of $\hH/\Gamma$. One can interpret the above result as saying that these moduli interpretations are compatible with the arithmetic properties of noncongruence modular forms, and hence are likely the ``right models'' for studying arithmetic questions about noncongruence objects.

\sgap

As in the classical case, a moduli theoretic model of $\hH/\Gamma$ is a \emph{fine moduli scheme} if and only if $\Gamma$ is torsion free:


\begin{thm*}[\ref{thm_torsion_free_implies_representable}] A modular curve $Y(\Gamma_{[\varphi]})$ over $\ZZ[1/|G|]$ is a \emph{fine moduli scheme}, and hence admits a universal family of elliptic curves with $G$-structure, if and only if $\Gamma_{[\varphi]}$ is torsion-free.
\end{thm*}

These results are collected and carefully stated in \S\ref{ss_basic_properties}. The combinatorial nature of $G$-structures means the geometry of the associated moduli spaces is determined by the structure of finite groups, and is thus readily accessible to computation. In \S\ref{section_examples} we carefully analyze the components of $M(G)$ for a few select groups $G$ and make some observations. In particular, in \S\ref{ss_dihedral_structure}, we use the ``branch cycle lemma'' (\ref{lemma_BCL}) to give a complete description of the structure of $\mM(D_{2k})_\QQ$, where $D_{2k}$ is the dihedral group of order $2k$. The main result of \S\ref{ss_dihedral_structure} is:

\begin{thm*}[\ref{thm_dihedral_geometric_structure}, \ref{thm_dihedral_galois_structure}] Let $k\ge 3$ be an integer, and $\mu_k^\prim$ be the set of primitive $k$th roots of unity. Let $E$ be an elliptic curve over a number field $K$, and $G_K := \Gal(\Qbar/K)$. Let $\Gamma := \Gamma_1(2)$ if $k$ is odd, otherwise $\Gamma := \Gamma(2)$ if $k$ is even. Then there is a $G_K$-equivariant bijection:
$$\{D_{2k}-\text{structures on $E/K$}\}\rightiso\{\Gamma-\text{structures on $E/K$}\}\times\left(\mu_k^\prim\right)^{\pm1}$$
In particular, there is a $G_\QQ$-equivariant bijection:
$$\{\text{Components of $\mM(D_{2k})_{\Qbar}$}\}\rightiso \left(\mu_k^\prim\right)^{\pm1}$$
Each component of $\mM(D_{2k})$ is thus defined over $\QQ(\zeta_k + \zeta_k^{-1})$, and over $\CC$ is isomorphic to $\hH/\Gamma_1(2)$ or $\hH/\Gamma(2)$ according to if $k$ is odd or even.
\end{thm*}
In particular, this says that an elliptic curve $E$ over a number field $K$ admits a $D_{2k}$-structure if and only if it both admits a $\Gamma$ structure and $\zeta_k+\zeta_k^{-1}\in K$.

\sgap

At the end of \S\ref{ss_three_simple_groups}, we describe the components of $\mM(\Sz(8))$, where $\Sz(8)$ is the Suzuki group of order 29120, a finite simple nonabelian group, and is the second smallest simple group which is not known to be a Galois group over $\QQ$. We show that in this case there is a representable component of $\mM(\Sz(8))_\QQ$ which is isomorphic to $Y := \PP^1_\QQ-\text{cusps}$. Thus, there is a universal elliptic curve $\EE$ over $Y$, equipped with an $\Sz(8)$-torsor $\XX\rightarrow\EE$. It seems to be generally believed that there exist $y\in Y(\QQ)$ such that $\EE_y$ over $\QQ$ has positive rank. If such a specialization exists, then the finiteness of rational points on curves of genus $\ge 2$ would imply the existence of infinitely many connected specializations of the $\Sz(8)$-torsor $\XX_y\rightarrow \EE_y$ at some $\QQ$-point of $\EE_y$, which would give many realizations of $\Sz(8)$ as a Galois group over $\QQ$.

\sgap


When $G$ is abelian, the corresponding modular curves are congruence modular curves, and so we say that abelian groups $G$ are \emph{congruence}. On the other hand, as can be seen from the above result about $\mM(D_{2k})$, nonabelian groups $G$ may still yield congruence moduli schemes. However, in all known examples, the only nonabelian groups $G$ which are congruence, are also solvable. Thus, in accordance with the philosophy that noncongruence modular curves should parametrize elliptic curves with nonabelian level structures, we conjecture:

\begin{conj*}[\ref{conj_nonsolvable_implies_noncongruence}] If $G$ is nonsolvable, then for every surjection $\varphi : F_2\twoheadrightarrow G$, $\Gamma_{[\varphi]} := \Stab_{\SL_2(\ZZ)}([\varphi])$ is noncongruence.
\end{conj*}

Towards this, we prove that if $G$ is a group admitting a generating pair $a,b$ such that the orders of $|a|,|b|,|ab|$ are pairwise coprime, then the stabilizer $\Gamma_{[\varphi]}$ of the surjection $\varphi : F_2\twoheadrightarrow G$ mapping the generators of $F_2$ to $a,b$ is noncongruence. In particular we show

\begin{thm*}[\ref{cor_noncongruence_summary}] If $G$ is an extension of $S_n (n\ge 4)$, $A_n (n\ge 5)$, $\PSL_2(\FF_p)\;(p\ge 5)$, or a minimal finite simple nonabelian group, then there is a surjection $\varphi : F_2\twoheadrightarrow G$ such that $\Gamma_{[\varphi]}$ is noncongruence.
\end{thm*}

Here, a minimal finite simple nonabelian group is a finite simple nonabelian group for which every proper subgroup is solvable (c.f. Corollary \ref{cor_MFSNG}).

\sgap

The examples discussed in \S4 are taken from large tables of computational data which can be found in Appendix \ref{appendix_tables}.

\subsubsection{Application to the Arithmetic of Noncongruence Modular Forms}\label{summary_UBD}

A $q$-expansion of a modular form $f$ for a finite index subgroup $\Gamma\le\SL_2(\ZZ)$ is said to have bounded denominators if its Fourier coefficients lie in $\Qbar$ and $c f$ has integral coefficients for some integer $c\ne 0$. While this is true of all congruence modular forms, it was first observed in \cite{ASD71} that it is no longer true if $f$ is not a congruence modular form.

\sgap

For a subgroup $\Gamma\le\SL_2(\ZZ)$, let $M_k(\Gamma,\Qbar)$ be the $\Qbar$-algebra of meromorphic modular forms of weight $k$, holomorphic on $\hH$, with Fourier coefficients in $\Qbar$. Let $M_*(\Gamma,\Qbar) := \bigoplus_{k\in\ZZ} M_k(\Gamma,\Qbar)$. We will say that a modular form $f\in M_*(\Gamma,\Qbar)$ is \emph{primitive} (for $\Gamma$) if it is not a modular form for any strictly larger subgroup $\Gamma'\supsetneq \Gamma$.

\sgap

\begin{conj*}[Unbounded Denominators Conjecture, c.f. \S\ref{conj_UBD}, and \cite{LL11,KL07,Bir94}] A modular form $f\in M_*(\Gamma,\Qbar)$ has bounded denominators at some cusp if and only if $f$ is a congruence modular form.
\end{conj*}


\sgap

By Proposition 5 of \cite{KL07} (also see Proposition \ref{prop_reduce_to_weight_0}), proving UBD for modular functions (modular forms of weight 0) would prove UBD in general, so it suffices to restrict our attention to (meromorphic) modular functions which are holomorphic on $\hH$.

\sgap

It was observed in \cite{ASD71} that there seem to be only finitely many bad primes at which one can have unbounded denominators. The fact that our moduli stacks are finite \'{e}tale over $\mM(1)_{\ZZ[1/N]}$ proves a somewhat more precise result:

\begin{thm*}[\ref{thm_bad_primes}] For a finite index $\Gamma\le\SL_2(\ZZ)$, let $\varphi : F_2\twoheadrightarrow G$ be such that $\Gamma_{[\varphi]}\lhd\Gamma$, then a modular form $f\in M_*(\Gamma,\Qbar)$ must have bounded denominators at any prime not dividing $|G|$.
\end{thm*}

For a number field $K$, let $\ZZ_K$ be its ring of integers. Then we define:
$$B(\Qbar,q) :=  \varinjlim_{n,K}\ZZ_K\ls{q^{1/n}}\otimes_{\ZZ_K}K$$
where the limit ranges over all integers $n\ge 1$ and number fields $K$. We show that if a modular function with Fourier coefficients in $\Qbar$ has bounded denominators, then its $q$-expansion lies in $B(\Qbar,q)$. If the Tate curve $\Tate(q)$ over $B(\Qbar,q)$ admits a $\Gamma$-structure, we would get a map
$$\Spec B(\Qbar,q) \longrightarrow\hH/\Gamma$$
which at the level of affine rings is given by $q$-expansions (c.f. Proposition \ref{prop_level_structures_are_qexps}), and hence the existence of noncongruence $\Gamma$-structures on $\Tate(q)$ over $B(\Qbar,q)$ would disprove UBD. The second main result of \S\ref{section_UBD} is that the converse is also true.

\begin{thm*}[\ref{thm_level_structures_and_UBD}] For any finite index $\Gamma\le\SL_2(\ZZ)$ containing $-I$, all modular forms in $\bigoplus_{k\in 2\ZZ}M_{k}(\Gamma,\Qbar)$ have bounded denominators at some cusp if and only if there exists a $\Gamma$-structure on $\Tate(q)$ over $B(\Qbar,q)$.
\end{thm*}

\sgap

In the last section of the paper (\S\ref{ss_geometric_UBD}) we reinterpret the statement of UBD geometrically. First, if we assume Conjecture \ref{conj_nonsolvable_implies_noncongruence}, then using the above result, a positive answer to the unbounded denominators conjecture would imply:
$$\textit{Every Galois cover of $\Tate(q)/B(\Qbar,q)$ unramified away from the origin is solvable.}$$
Viewed another way, the Tate curve $\Tate(q)/B(\Qbar,q)$ determines a map $\Spec B(\Qbar,q)\longrightarrow\mM(1)_{\Qbar}$, and hence a map of fundamental groups
$$\pi_1(\Spec B(\Qbar,q))\stackrel{\Tate(q)_*}{\longrightarrow}\pi_1(\mM(1)_{\Qbar}) = \widehat{\SL_2(\ZZ)}$$
From the classical theory, or alternatively through an explicit parametrization of the torsion points on $\Tate(q)$ (c.f. Corollary \ref{cor_tate_level_N}), we find that all $\Gamma(N)$-structures are defined on $\Tate(q)$ over $B(\Qbar,q)$. This implies that the image of $\pi_1(\Spec B(\Qbar,q))$ lies inside the congruence kernel $\bigcap_{N\ge 1}\ol{\Gamma(N)}$, which fits inside an exact sequence:
$$1\longrightarrow\bigcap_{N\ge 1}\ol{\Gamma(N)}\longrightarrow\widehat{\SL_2(\ZZ)}\longrightarrow\SL_2(\widehat{\ZZ})\longrightarrow 1$$
In the discussion preceding Theorem \ref{thm_UBDB=UBD}, we prove that the UBD conjecture is equivalent to the image of $\pi_1(\Spec B(\Qbar,q))$ being equal to the congruence kernel. In other words,
\begin{thm*}[\ref{thm_UBDB=UBD}] The UBD conjecture is equivalent to the exactness of the sequence
$$\pi_1(\Spec B(\Qbar,q)) \stackrel{\Tate(q)_*}{\longrightarrow}\widehat{\SL_2(\ZZ)}\longrightarrow\SL_2(\widehat{\ZZ})\longrightarrow 1$$
\end{thm*}

We remark that since the congruence kernel $\bigcap_{N\ge 1}\ol{\Gamma(N)}$ is a free profinite group of countable rank (c.f. \cite{Mel76}), the UBD conjecture would imply the Inverse Galois Problem for the field $\Frac\; B(\Qbar,q)$, but this is actually already known by a result of Harbater (\cite{Har84}).


\subsection{Acknowledgements} The author is grateful to Pierre Deligne, Winnie Li, John Voight, Jordan Ellenberg, Ching-Li Chai, and Christelle Vincent for their generous comments and helpful suggestions while revising this paper. He would also like to thank Hilaf Hasson and Jeff Yelton for many enlightening discussions where much confusion was both generated and dispersed. The paper was partially written during the author's stay at the National Center for Theoretical Sciences (NCTS) in Taiwan, and revised during the author's visit at ICERM, and postdoc at the IAS. He would like to thank NCTS, ICERM, and the IAS for their support and hospitality.


\section{Setup}\label{section_setup}
In this section we give a careful construction of Teichm\"{u}ller level structures on punctured elliptic curves, though we note that the constructions here also work in greater generality.

\sgap
Throughout the paper, the letter $G$ will generally be used to refer to either a finite group of order $N$, or the associated constant group scheme (over some understood base). The one exception is in \S\ref{ss_TLS}, where we may use $G$ to refer to a nonconstant commutative finite \'{e}tale group scheme. A group is \emph{2-generated} if it can be generated by two elements.


\sgap

By default, for a scheme $X$, the notation $\pi_1(X)$ will refer to the \'{e}tale fundamental group. If $X$ is a scheme over $\CC$, its topological fundamental group will be denoted by $\pi_1^\tp(X)$.

\sgap

For a morphism $T\rightarrow S$, and a scheme $X$ over $S$, we will use $X_T$ to denote the pullback $X\times_S T$. For schemes $X,T$ over $S$, we will use $X(T) := \Hom_S(T,X)$, the ``$T$-valued points'' of $X$.

\sgap

For us, a modular curve is by definition the quotient of the upper half plane $\hH$ by a finite index subgroup $\Gamma\subset\SL_2(\ZZ)$. We will reserve the letter $Y$ (resp. $\yY$) to refer to a geometrically connected modular curve (resp. geometrically connected component of a moduli stack).

\subsection{The relative fundamental group}\label{ss_RFG}
Let $f : E\rightarrow S$ be an elliptic curve over a connected scheme $S$ with identity section $e : S\rightarrow E$. Let $E^\circ := E - e(S)$, and let $s$ be a geometric point of $S$. Let $\LL$ be the set of primes invertible on $S$. For a profinite group $\pi$, let $\pi^\LL$ denote its maximal pro-$\LL$ quotient. Let $F_2$ denote the free group of rank 2. We record here the following important fact.

\begin{thm}\label{finite_generation}(\!\cite{SGA1}, \S XIII, Cor. 2.12) Let $k = k^\sep$ be a separably closed field. Let $\LL$ be a set of primes not containing $\ch(k)$ and let $E/k$ be an elliptic curve over $k$, then $\pi_1^\LL(E)\cong(\widehat{\ZZ}^2)^\LL$ and $\pi_1^\LL(E^\circ)\cong(\widehat{F_2})^\LL$. 
\end{thm}

We begin by considering the unpunctured curve $E$. Let $i_s : E_s\hookrightarrow E$ be the inclusion, then by \cite{SGA1}, \S XIII, Prop 4.3 and Exemples 4.4, we have a homotopy exact sequence:
\begin{equation}\label{abHES} 
1\longrightarrow\pi_1^\LL(E_s,e(s))\stackrel{(i_s)_*}{\longrightarrow}\pi_1'(E,e(s))\stackrel{f_*}{\longrightarrow}\pi_1(S,s)\longrightarrow 1
\end{equation}
where $\pi_1'(E) := \pi_1(E)/M$, and $M$ is defined to be the smallest subgroup of $K := \ker\left(\pi_1(E)\rightarrow\pi_1(S)\right)$ such that $K/M$ is pro-$\LL$. Note that $M$ is a characteristic subgroup of $K$, hence normal in $\pi_1(E)$, and the kernel of $f_*$ above is precisely $K^\LL$. The section $e : S\rightarrow E$ induces a splitting of the above sequence $e_* : \pi_1(S,s)\rightarrow \pi_1'(E,e(s))$.  Thus, we get an action of $\pi_1(S,s)$ on $\pi_1^\LL(E_s,e(s))$ by conjugation inside $\pi_1'(E,e(s))$. By the Galois correspondence, this determines a pro-object 
$$\pi_1^\LL(E/S,e,s)$$
of the category of finite \'{e}tale group schemes over $S$ (c.f. \cite{SGA1}, \S V, Prop 5.2), which is called the \emph{relative fundamental group} for $E/S$. More concretely, for every finite characteristic quotient $K_i\;(i\in I)$ of $\pi_1^\LL(E_s,e(s))$, the $\pi_1(S,s)$ action on $\pi_1^\LL(E_s,e(s))$  descends to an action on $K_i$. By the usual Galois correspondence, each such $K_i$ corresponds to a finite \'{e}tale group scheme $K_i(E/S)$ over $S$. Then $\pi_1^\LL(E/S,e,s)$ is the inverse system of the finite \'{e}tale group schemes $\{K_i(E/S)\}_{i\in I}$.

\sgap

\begin{remark} If $S$ is Noetherian, then for our purposes it will be harmless to work instead with the limit $\varprojlim\pi_1^\LL(E/S,e,s) = \varprojlim_K K(E/S)$ of the inverse system described above.

\end{remark}

\begin{lemma}\label{lemma_RFG_is_Tate_module} We have an isomorphism of pro-objects
$$\pi_1^\LL(E/S,e,s)\cong \{E[n]\}_{n}$$
where $n$ ranges over all positive integers which are only divisible by primes in $\LL$.
\end{lemma}

In other words, if $S = \Spec k$ for a field $k$, then $\pi_1(S,s) = G_k := \Gal(k^\sep/k)$ and $\pi_1^\LL(E_s,e(s))$ is isomorphic as a $G_k$-module to the $\LL$-adic Tate module $\prod_{\ell\in\LL}T_\ell(E)$.

\begin{proof} Note that $\pi_1^\LL(E_s,e(s))\cong(\Zhat^2)^\LL$ admits $(\ZZ/n\ZZ)^2$ as a characteristic quotient, and thus the $\pi_1(S,s)$ action on $\pi_1^\LL(E_s,e(s))$ induces an action on $(\ZZ/n\ZZ)^2$. By the construction of $\pi_1^\LL(E/S,e,s)$, it suffices to show that the induced $\pi_1(S,s)$-action on $(\ZZ/n\ZZ)^2$ correspond (via the Galois correspondence) to precisely the group schemes $E[n]$ over $S$, together with the compatibilities between $E[n]$ and $E[m]$ for $n\mid m$. We'll prove the first statement and leave the checking of compatibilities to the reader.

\sgap

The group $\pi_1'(E,e(s))$ acts on the normal subgroup $\pi_1^\LL(E_s,e(s))$ by conjugation, and hence induces an action on the characteristic quotient $(\ZZ/n\ZZ)^2$. In a diagram, we have:

$$\xymatrix{
\pi_1^\LL(E_s,e(s))\ar@{^{(}->}[d]\ar[r] & (\ZZ/n\ZZ)^2\ar[d] \\
\pi_1'(E,e(s))\ar[r]^{\rho} & S_{(\ZZ/n\ZZ)^2}
}$$

where the right vertical arrow is via the left regular representation, and $S_{(\ZZ/n\ZZ)^2}$ is the symmetric group on the set $(\ZZ/n\ZZ)^2$. The arrow $\rho$ corresponds to a finite \'{e}tale morphism $X\rightarrow E$ of degree $n^2$, and thus by Riemann-Hurwitz, the fibers of $X$ over $S$ are genus 1. The composition
$$\pi_1(S,s)\stackrel{e_*}{\longrightarrow}\pi_1'(E,e(s))\stackrel{\rho}{\longrightarrow} S_{(\ZZ/n\ZZ)^2}$$
corresponds to the restriction $X_{e(S)}$ of $X\rightarrow E$ to the identity section $e(S)\subset E$, and since its image is contained in $\Aut((\ZZ/n\ZZ)^2)\subset S_{(\ZZ/n\ZZ)^2}$, $X_{e(S)}$ is a finite \'{e}tale group scheme with geometric fiber $(\ZZ/n\ZZ)^2$. In particular, we may declare its identity section to be the identity section of $X$, thus making $X\rightarrow E$ into an isogeny of elliptic curves with kernel $X_{e(S)}$. Since the geometric fiber of $X_{e(S)}$ is $(\ZZ/n\ZZ)^2$, we must have $X_{e(S)} = X[n]$, and thus the map $X\rightarrow E$ makes $E\cong X/X[n]$. However any elliptic curve is canonically isomorphic to its quotient by its own $n$-torsion, so $E$ is canonically isomorphic to $X$, and so $X_{e(S)}\cong X[n]\cong E[n]$ as desired.

\end{proof}

We now consider the case of punctured elliptic curves. 
\begin{prop}\label{HES} Let $E^\circ := E - e(S)$. With notation as above, suppose either that there is a prime $p$ which is invertible on $S$, or that $E^\circ/S$ admits a section $g : S\rightarrow E^\circ$. Then, the following sequence (choosing appropriate base points) is exact
\begin{equation}\label{eq_HES}\xymatrix{1\ar[r] & \pi_1^\LL(E^\circ_s)\ar[r]^{(i_s)_*} & \pi_1'(E^\circ)\ar[r]^{f_*} & \pi_1(S)\ar[r] & 1}\end{equation}
\end{prop}

\begin{proof} If $E\stackrel{f}{\longrightarrow} S$ admits a section, then this follows from \cite{SGA1}, \S XIII, Prop 4.3 and Exemples 4.4. Otherwise, suppose $p$ is invertible on $S$, so that the subgroup scheme $E[p]\subset E$ is finite \'{e}tale over $S$
. Let $H$ be a non-identity connected component of $E[p]$, so that $H\subset E^\circ$, and let $H'$ be the Galois closure of $H$ as a cover of $S$. Base changing by the natural map $H' \stackrel{p}{\rightarrow} S$, we find that $H\times_S H'$ is completely decomposed over $H'$, and hence $E_{H'}^\circ/H'$ admits a section. Thus, letting $h\in H'$ be a geometric point over $s$, applying \cite{SGA1} \S XIII Prop 4.3 to $E^\circ_{H'}/H'$, and applying \cite{SGA1} \S XIII Prop 4.1 to $E^\circ/S$, we get a commutative diagram
$$\xymatrix{
 & \pi_1^\LL(E^\circ_s)\ar[r]^{(i_s)_*} & \pi_1'(E^\circ)\ar[r]^{f_*} & \pi_1(S)\ar[r] & 1 \\
 1\ar[r] & \pi_1^\LL(E^\circ_h)\ar[r]^{(i_h)_*}\ar[u]^\cong & \pi_1'(E_{H'}^\circ)\ar[r]^{(f_{H'})_*}\ar[u]^{\tilde{p}_*'} & \pi_1(H')\ar[u]^{p_*}\ar[r] & 1
}$$
Since $H'\rightarrow S$ is finite \'{e}tale, $p_*$ and $\tilde{p}_* : \pi_1(E^\circ_{H'})\rightarrow \pi_1(E^\circ)$ are injective with open image of the same index. Thus, they induce an isomorphism on the kernels to $\pi_1(H')$ and $\pi_1(S)$. This implies that $\tilde{p}_*' : \pi_1'(E^\circ_{H'})\rightarrow \pi_1'(E^\circ)$ induces an isomorphism $\ker(f_{H'})_*\rightiso \ker f_*$, and hence $\tilde{p}_*'$ is an injection, which implies that $(i_s)_*$ is injective as well.
\end{proof}

\sgap

\begin{pt}\label{pt_RFG} Suppose now that we have a section $g : S\rightarrow E^\circ$. This induces a section $g_* : \pi_1(S)\rightarrow\pi_1'(E^\circ)$ splitting the exact sequence (\ref{eq_HES}). As in (\ref{abHES}), this data determines a pro-object
$$\pi_1^\LL(E^\circ/S,g,s)$$
in the category of finite \'{e}tale group schemes  - the \emph{relative fundamental group} for $E^\circ/S$.
\end{pt}

\sgap

\begin{prop}\label{prop_basic_results_RFG} \label{prop_inner_unique}  (\cite{SGA1}, Expos\'{e} XIII, 4.5.2, 4.5.3) We have
\begin{itemize}
\item[1.] The formation of $\pi_1^\LL(E^\circ/S,g,s)$ commutes with arbitrary base change. If $h : S'\rightarrow S$ is a connected $S$-scheme, and $s'$ a geometric point of $S'$, then we have a canonical isomorphism
$$h^*\pi_1^\LL(E^\circ/S,g,s)\cong\pi_1^\LL(h^*E^\circ/S',h^*g,s')$$
\item[2.] For any geometric point $\xi$ of $S$, there is a canonical isomorphism
$$\pi_1^\LL(E^\circ/S,g,s)_\xi\cong\pi_1^\LL(E^\circ_\xi,g(\xi))$$
\item[3.] If $s'\in S$ is another geometric point of $S$, then the pro-objects $\pi_1^\LL(E^\circ/S,g,s)$ and $\pi_1^\LL(E^\circ/S,g,s')$ are canonically isomorphic.
\end{itemize}
\end{prop}

\begin{proof} See \cite{SGA1}, Expos\'{e} XIII, 4.5.2, 4.5.3. 
\end{proof}

Thanks to Proposition \ref{prop_basic_results_RFG} (3), we may more simply write $\pi_1^\LL(E^\circ/S,g)$ to denote $\pi_1^\LL(E^\circ/S,g,s)$.

\subsection{Teichm\"{u}ller level structures ($G$-structures)}\label{ss_TLS} In this section we give a careful treatment of Teichm\"{u}ller level structures as first defined by Deligne and Mumford in \cite{DM69} and again studied by Pikaart and de Jong in \cite{PJ08} for proper curves of genus $\ge 2$.

\sgap

As in \ref{HES} and \ref{pt_RFG}, let $f : E\rightarrow S$ be an elliptic curve over a connected scheme $S$, $E^\circ := E - e(S)$, and suppose we have a section $g : S\rightarrow E^\circ$. Let $s\in S$ be a geometric point, and $\LL$ the set of primes invertible on $S$. Then we have the relative fundamental group $\pi_1^\LL(E^\circ/S,g,s)$, which is an inverse system of finite \'{e}tale group schemes over $S$.

\sgap

Let $G$ be a finite \'{e}tale group scheme of order $N$ over scheme $S$, where $N$ is invertible on $S$. Let $\G := G_s$ be its geometric fiber, viewed as an abstract group equipped with a $\pi_1(S,s)$-action. We will focus on the case where $G$ is a constant group scheme.

\sgap

Note that our assumptions imply that if $E$ is an elliptic curve over a separably closed field, then by Theorem \ref{finite_generation}, we have canonical isomorphisms
$$\Hom(\pi_1^\LL(E^\circ),\G) \cong \Hom(\pi_1(E^\circ),\G) \cong \Hom(F_2,\G)$$
A Teichm\"{u}ller level structure will be represented by a surjective homomorphism of this kind. To exclude trivial cases from now on we will assume that $\G$ is a nontrivial 2-generated group.

\sgap

For any finite 2-generated group $\G$ of order $N$ invertible on $S$ as above, let $H_{\G}$ be the intersection of the kernels of \emph{all} homomorphisms $\pi_1^\LL(E_s,g(s))\rightarrow\G$. Then $H_\G$ is an open normal subgroup invariant under $\pi_1(S,s)$, and hence the quotient $K_\G := \pi_1^\LL(E_s,g(s))/H_\G$ inherits the action of $\pi_1(S,s)$, and corresponds to a finite \'{e}tale group scheme $K_\G(E^\circ/S)$. Let $\gG$ be a set of representatives of isomorphism classes of all finite 2-generated groups of order invertible on $S$, then the subgroups $\{H_\G\}_{\G\in\gG}$ are cofinal in the lattice of open subgroups of $\pi_1^\LL(E_s,g(s))$, and hence we have:
$$\pi_1^\LL(E^\circ/S,g) = \{K_\G(E^\circ/S)\}_{\G\in\gG}$$
By the definition of morphisms of pro-objects, we have
$$\Hom_S(\pi_1^\LL(E^\circ/S,g),G) = \varinjlim_{\G\in \gG}\Hom_S(K_\G(E^\circ/S),G)$$

By construction, any homomorphism $\pi_1^\LL(E^\circ/S,g)\rightarrow G$ must factor uniquely through $K_\G(E^\circ/S)$, and hence for any $\G\in\gG$ we have a canonical bijection $\Hom_S(\pi_1^\LL(E^\circ/S,g),G)\cong\Hom_S(K_\G(E^\circ/S),G)$, and hence (by Proposition \ref{prop_basic_results_RFG}(1)) an isomorphism
\begin{equation}\label{eq_characteristic_quotient} \hHom_S(\pi_1^\LL(E^\circ/S,g),G) \cong \hHom_S(K_\G(E^\circ/S),G)\end{equation}
of sheaves of sets on the \'{e}tale site $(\Sch/S)_\et$. 
We denote the subsheaf corresponding to surjective homomorphisms by
$$\hHom^\sur_S(\pi_1^\LL(E^\circ/S,g),G)$$
The isomorphism of (\ref{eq_characteristic_quotient}) implies:




\begin{prop}\label{prop_FLC} The sheaves $\hHom_S(\pi_1^\LL(E^\circ/S,g),G)$ and $\hHom_S^\sur(\pi_1^\LL(E^\circ/S,g),G)$ are finite locally constant on $(\Sch/S)_{\et}$, hence representable by schemes finite \'{e}tale over $S$.
\end{prop}

\begin{proof} This follows from the fact that the Hom sheaf between two finite \'{e}tale group schemes is finite, and becomes constant over a common trivializing \'{e}tale covering. 
\end{proof}

\sgap


There is a natural action of $G$ on $\hHom_S^\sur(\pi_1^\LL(E^\circ/S,g),G)$ by inner automorphisms. Define:
$$\hHom^\surext(\pi_1(E^\circ/S,g),G) := \hHom^\sur_S(\pi_1^\LL(E^\circ/S,g),G)/\Inn(G)$$
Here we are justified in dropping the $\LL$ from the notation since $|G|$ is divisible only by primes in $\LL$. For two different sections $g,h : S\rightarrow E^\circ$, the relative fundamental groups $\pi_1^\LL(E^\circ/S,g),\pi_1^\LL(E^\circ/S,h)$ may \emph{not} be isomorphic. We wish to investigate how/if
$$\hH(g) := \hHom^\surext(\pi_1(E^\circ/S,g),G)$$
depends on the choice of the section $g : S\rightarrow E^\circ$. Clearly its geometric fiber is
$$\Hom^\surext(\pi_1^\LL(E^\circ_s,g(s)),\G) := \Hom^\sur(\pi_1^\LL(E^\circ_s,g(s)),\G)/\Inn(\G)$$
Let $T\rightarrow S$ be a Galois finite \'{e}tale covering which trivializes both $K_\G(E^\circ/S)$ and $G$, then $\hH(g)_T$ is completely decomposed over $T$, and its sections are in bijection with $\Hom^\sur(\pi_1^\LL(E^\circ_s,g(s)),\G)/\Inn(\G)$. Since $\hH(g)$ is a sheaf, the global sections $\hH(g)(S)$ of $\hH(g)$ are identified with the sections $\hH(g)_T(T)$ which are invariant under $\Gal(T/S)$, or equivalently invariant under $\pi_1(S,s)$
. Thus, we have a bijection
\begin{equation}\label{eq_sections_of_H(g)}
\hH(g)(S) = \{[\varphi]\in\Hom^\surext(\pi_1^\LL(E^\circ_s,g(s)),\G) : \text{ for all $\sigma\in\pi_1(S,s)$, we have } [\varphi\circ\sigma] = [\varphi]\}
\end{equation}
where above we view $\sigma$ as an automorphism of $\pi_1^\LL(E^\circ_s,g(s))$. Thus, every section in $\hH(g)(S)$ is represented by the $\Inn(\G)$-class of a surjective homomorphism $\varphi : \pi_1^\LL(E^\circ_s,g(s))\rightarrow\G$ such that

\begin{equation}\label{eq_compatibility}
\forall\sigma\in\pi_1(S,s),\;\exists x \in\G,\;\forall\gamma\in\pi_1^\LL(E^\circ_s,g(s)) :\qquad \varphi(\,^\sigma\gamma) = x\varphi(\gamma) x^{-1}
\end{equation}

We will use this to show:
\begin{prop} Let $g,h : S\rightarrow E^\circ$ be two sections. Then we have a \emph{canonical} isomorphism
$$\hH(g) := \hHom^\surext(\pi_1(E^\circ/S,g),G)\rightiso \hHom^\surext(\pi_1(E^\circ/S,h),G) =: \hH(h)$$
\end{prop}

\begin{proof} For a scheme $X$, let $\FEt_X$ be the category of schemes finite \'{e}tale over $X$. Then let $F_{g(s)},F_{h(s)} : \FEt_{E^\circ}\rightarrow\Sets$ be the fiber functors associated to the points $h(s),g(s)\in E^\circ$. Let $F_{s,g(s)},F_{s,h(s)} : \FEt_{E^\circ_s}\rightarrow\Sets$ be the fiber functors associated to $h(s),g(s)\in E^\circ_s$, and let $F_s : \FEt_S\rightarrow\Sets$ be the fiber functor for $s\in S$. Then we have the following com/mutative diagram:


\begin{equation}\label{eq_FEt_diagram}
\raisebox{-0.5\height}{\includegraphics{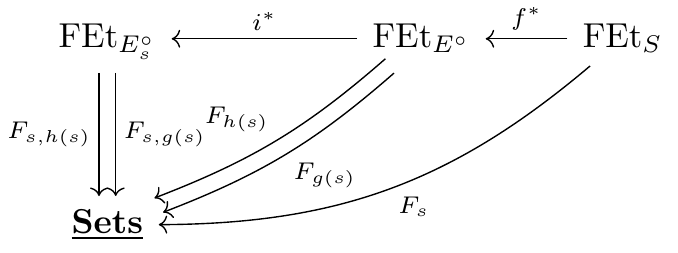}}
\end{equation}

Where ``commutative'' means that the composition of any two paths with the same endpoints which don't involve both $h(s)$ and $g(s)$ are canonically isomorphic. Choose an isomorphism $\alpha_s : F_{s,h(s)}\rightiso F_{s,g(s)}$ \;\footnote{this $\alpha_s$ should be thought of as a ``path'' $h(s)\leadsto g(s)$ in $E^\circ_s$}. Then this induces an isomorphism of groups
\begin{eqnarray*}
\pi_1^\LL(E^\circ_s,h(s)) = \Aut(F_{s,h(s)})^\LL & \stackrel{(\alpha_s)_*}{\longrightarrow} & \Aut(F_{s,g(s)})^\LL = \pi_1^\LL(E^\circ_s,g(s)) \\
\gamma' & \mapsto & \alpha_s\circ\gamma'\circ\alpha_s^{-1}
\end{eqnarray*}
Let $[\varphi]\in\hH(g)$ be a section, then it is represented by a surjection $\varphi : \pi_1^\LL(E^\circ_s,g(s))\rightarrow\G$ satisfying (\ref{eq_compatibility}) above. We'd like to show that $\psi := \varphi\circ(\alpha_s)_*$ also satisfies (\ref{eq_compatibility}), and hence represents a section of $\hH(h) := \hHom^\surext(\pi_1(E^\circ/S,h),G)$. We recall that the action of $\sigma\in\pi_1(S,s)$ on $\gamma\in\pi_1^\LL(E^\circ_s,g(s))$ is via conjugation $\,^\sigma\gamma = g_*(\sigma)i_*(\gamma)g_*(\sigma)^{-1}$ in $\pi_1'(E^\circ,g(s))$, and similarly for $h$. For simplicity of notation we will identify $\pi_1^\LL(E^\circ_s,g(s))$ with its image in $\pi_1'(E^\circ,g(s))$ via $i_*$, and similarly for $h$.

\sgap

Thus, for any $\sigma\in\pi_1(S,s)$, let $\delta = \delta_\sigma := i_*\alpha_s^{-1} \circ g_*(\sigma)^{-1}\circ i_*\alpha_s\circ h_*(\sigma)\in\Aut(F_{h(s)}) = \pi_1(E^\circ,h(s))$, we have (for any $\gamma\in\pi_1^\LL(E_s^\circ,h(s))$)
$$\psi(\,^\sigma\gamma') = \varphi\circ(\alpha_s)_*(\,^\sigma\gamma') = \varphi(\alpha_sh_*(\sigma)\gamma'h_*(\sigma)^{-1}\alpha_s^{-1}) = \varphi(g_*(\sigma)\alpha_s\delta\gamma'\delta^{-1}\alpha_s^{-1} g_*(\sigma)^{-1})$$
Since $\varphi$ represents an element of $\hH(g)$, it satisfies (\ref{eq_compatibility}), so there exists $x\in\gG$ such that
$$\cdots = x\varphi(\alpha_s\delta\gamma'\delta^{-1}\alpha_s^{-1})x^{-1} = x\left( \psi(\delta\gamma'\delta^{-1})\right)x^{-1} = x\psi(\delta)\big(\psi(\gamma')\big)\psi(\delta)^{-1}x^{-1}$$
which shows that $\psi := \varphi\circ(\alpha_s)_*$ also satisfies (\ref{eq_compatibility}). However, the last equality above only makes sense if $\delta\in\pi_1^\LL(E^\circ_s,h(s))$. We explain why this must always be the case. By definition,
$$\delta = \delta_s := i_*\alpha_s^{-1}\circ g_*(\sigma)^{-1}\circ i_*\alpha_s \circ h_*(\sigma)\in\Aut(F_{h(s)}) = \pi_1(E^\circ,h(s))$$
By the diagram (\ref{eq_FEt_diagram}), both $F_{s,h(s)}\circ i^*\circ f^*$ and $F_{s,g(s)}\circ i^*\circ f^*$ are canonically isomorphic to $F_s$. Thus, $\alpha_s : F_{s,h(s)}\rightiso F_{s,g(s)}$ induces an automorphism $f_*i_*\alpha_s\in\Aut(F_s) = \pi_1(S,s)$. Using the fact that $f\circ i : E^\circ_s\rightarrow S$ factors through the geometric point $s\in S$, it is straightforward to check that this induced automorphism $f_*i_*\alpha_s = \id_{F_s}$. Then, we have:
$$f_*\delta = f_*(i_*\alpha_s^{-1}\circ g_*(\sigma)^{-1}\circ i_*\alpha_s\circ h_*(\sigma)) = \id_{F_s}\circ f_*g_*(\sigma)^{-1}\circ \id_{F_s}\circ f_*h_*(\sigma) = \sigma^{-1}\circ\sigma = \id_{F_s}$$
and thus by the exactness of (\ref{eq_HES}), we find that $\delta\in \pi_1^\LL(E^\circ_s,h(s))$ as desired.

\sgap

The above shows that any choice of path $\alpha_s : F_{s,h(s)}\rightiso F_{s,g(s)}$ we obtain a bijection of global sections $\hH(h)(S)\rightiso\hH(g)(S)$. Since any two choices of ``paths'' $F_{s,h(s)}\rightiso F_{s,g(s)}$ differ by an inner automorphism, it is clear from the description of $\hH(g)(S),\hH(h)(S)$ in (\ref{eq_sections_of_H(g)}) that the bijection induced by $\alpha_s$ does not depend on the choice of path. Lastly, by compatibility with base change (c.f. \ref{prop_basic_results_RFG}(1)), the same argument above shows that $\hH(g)(T)$ and $\hH(h)(T)$ are also canonically isomorphic for any $T\rightarrow S$ finite \'{e}tale. Since $\hH(g),\hH(h)$ are representable, by the Yoneda lemma applied to $\FEt_S$, this shows that $\hH(g)$ and $\hH(h)$ are canonically isomorphic as desired.

\end{proof}

Thus, we may more simply write:
$$\hHom^\surext(\pi_1(E^\circ/S),G) := \hHom_S^\sur(\pi_1^\LL(E^\circ/S,g),G)/\Inn(G)$$
We call its sections ``surjective exterior homomorphisms''.





\sgap

Since sections to smooth morphisms always exist \'{e}tale-locally (c.f. \cite{Stacks} tag 055U), in the absence of a section $g$, we may define $\hHom^\surext(\pi_1(E^\circ/S),G)$ by sheafification:

\begin{defn}[$G$-Structures] Let $E/S$ be an elliptic curve with identity section $e$, and $E^\circ := E - e(S)$. Let $G$ be a nontrivial finite \'{e}tale group scheme of order invertible on $S$, whose geometric fiber $\G := G_s$ is a 2-generated group. By stability under base change (Prop \ref{prop_basic_results_RFG}(1)), the following rule defines a presheaf $\fF$ on $(\Sch/S)_\et$:
$$(T\stackrel{\;}{\rightarrow} S)\mapsto\left\{\begin{array}{ll}
\hHom^\surext(\pi_1(E^\circ/S),G)(T) & \text{if $E_T^\circ/T$ admits a section $g$ over $T$}, \\
\emptyset & \text{otherwise.}
\end{array}\right.$$
We will also use $\hHom^\surext(\pi_1(E^\circ/S),G)$ to denote the sheafification of this presheaf $\fF$. We define a Teichm\"{u}ller structure of level $G$ (or simply $G$-structure) on $E/S$ to be a global section of $\hHom^\surext(\pi_1(E^\circ/S),G)$.
\end{defn}

\sgap

Note that whenever $T\rightarrow S$ is such that $E^\circ_T/T$ admits a section, then $\fF|_T$ is just the finite locally constant sheaf $\hHom^\surext(\pi_1^\LL(E^\circ_T/T),G)$, and so $\fF$ is a separated presheaf. Since sections of smooth morphisms exist \'{e}tale-locally, we have

\begin{prop}\label{prop_FLC_2} The sheaf $\hHom^\surext(\pi_1(E^\circ/S),G)$ of $G$-structures on $E^\circ/S$ is finite locally constant, representable by a scheme finite \'{e}tale over $S$. Let $\tilde{s}\in E_s^\circ$ be a geometric point above $s\in S$, then its geometric fiber is the set
$$\hHom^\surext(\pi_1(E^\circ/S),G)_s = \Hom^\surext(\pi_1^\LL(E^\circ_s,\tilde{s}),\G) = \Hom^\surext(\widehat{F_2},\G)$$
\end{prop}
\begin{proof} The statement follows from the discussion above and Prop \ref{prop_basic_results_RFG}(2).
\end{proof}


\begin{pt} The above constructions are also valid with $E^\circ$ replaced by any complement of a normal crossings divisor relative to $S$ inside a proper smooth morphism $X\rightarrow S$ with geometrically connected fibers (\cite{SGA1} \S XIII.4.4). I.e., $X/S$ need not be genus 1, nor even a curve.
\end{pt}

We have the following description of $G$-structures:



\begin{prop}\label{prop_Teichmuller_moduli_2} In the above situation, for a homomorphism $\varphi : \pi_1^\LL(E^\circ_s)\rightarrow \G$, let $[\varphi]$ be its class mod $\Inn(\G)$. Let $\tilde{s}\in E^\circ_s$ be a geometric point above $s$. 
Then
\begin{itemize}
\item[(1)] The following sequence is exact
$$\xymatrix{1\ar[r] & \pi_1^\LL(E^\circ_s,\tilde{s})\ar[r] & \pi_1'(E^\circ,\tilde{s})\ar[r] & \pi_1(S,s)\ar[r] & 1}$$
Hence we obtain an outer representation
$$\bar{\rho} : \pi_1(S,s)\longrightarrow\Out(\pi_1^\LL(E^\circ_s,\tilde{s}))$$
\item[(2)] Let $\tilde{s}\in E^\circ_s$ be a geometric point lying above $s$. The monodromy action of $\pi_1(S,s)$ on the fiber
$$\hHom^\surext(\pi_1(E^\circ/S),G)_s = \Hom^\surext(\pi_1^\LL(E^\circ_s,\tilde{s}),\G)$$
is given by 
$$\,^\sigma[\varphi] = [\sigma_G\circ\varphi]\circ\ol{\rho}(\sigma)^{-1}$$
where $\sigma_G$ is the automorphism of $\G$ induced by the monodromy action of $\pi_1(S,s)$ on $\G := G_s$. Hence, we have a bijection
$$\{\text{$G$-structures on $E/S$}\}\rightiso\{[\varphi] \in\Hom^\surext(\pi_1^\LL(E^\circ_s,\tilde{s}),\G) : \forall\sigma\in\pi_1(S,s),\; \,^\sigma[\varphi] = [\varphi]\}$$
In particular, every $G$-structure can be represented by a surjection $\varphi : F_2\twoheadrightarrow \G$.

\sgap

\item[(3)] Suppose $G$ is a constant group scheme, and suppose either that $E^\circ/S$ admits a section, or that $S = \Spec R$ and $f_*\Omega_{E/S}$ is trivial. Then by possibly making a different choice of base point $\tilde{s}$, the exact sequence of (1) is split, from which we get a representation:
$$\rho : \pi_1(S,s)\longrightarrow\Aut(\pi_1^\LL(E^\circ_s,\tilde{s}))$$
lifting $\bar{\rho}$. In this case there is a natural map $\beta$
$$\{\text{geometrically connected} \text{ $G$-torsors}\footnote{Our $G$-torsors are relative to the \'{e}tale topology, and so a $G$-torsor $X$ over $Y$ is the same as a finite \'{e}tale Galois cover equipped with an isomorphism $G\rightiso\Gal(X/Y)$. Sometimes this is referred to as a $G$-Galois cover.} \text{ on $E^\circ/S$}\}/\cong\quad\stackrel{\beta}{\longrightarrow}\quad\{\text{$G$-structures on $E/S$}\}$$
which is a \emph{bijection} if the center $Z(G)$ of $G$ is trivial.



\end{itemize}
\end{prop}

Before we give the proof, we make the following remarks:

\begin{pt} In (3), the obstructions to the map $\beta$ being injective or surjective both lie in certain cohomology groups with coefficients in $Z(G)$, which explains the bijectivity in the case $Z(G) = 1$. For the sake of explicitness, in our proof we have translated the cohomology which appears into elementary group theory.

\end{pt}

\begin{pt} If $S = \Spec k$ for some field $k$, we also have that $G$-structures on $E^\circ/S$ correspond precisely to the $G$-torsors on $E_{\bar{k}}^\circ$ whose isomorphism class is fixed by $\Gal(\bar{k}/k)$. In other words, these are precisely the isomorphism classes of connected $G$-torsors over $E_{\bar{k}}^\circ$ whose \emph{field of moduli} (as a $G$-torsor) is contained in $k$. The statement of Proposition \ref{prop_Teichmuller_moduli_2}(3) corresponds to the classical fact that if $Z(G) = 1$, then every $G$-torsor of a curve over an algebraically closed field is defined over its field of moduli. On the other hand, without any restrictions on $G$, the field of moduli $k$ is always a field of definition if one only cares about the cover and not the $G$-action --- ie., if one does not require that the $G$-action be defined over $k$. For more on such subtleties, see \cite{Has13}.



\end{pt}


\begin{pt}\label{pt_coarser_structures} One may also consider variations of $G$-structures. For example, one may consider $\Aut(G)$-classes instead of $\Inn(G)$ classes. These would parametrize Galois covers with Galois group isomorphic to $G$ (ie, $G$-torsors where one forgets the action of $G$). The corresponding moduli spaces that we define in \S\ref{ss_definition_of_M(G)} would then yield \emph{origami curves} as studied in e.g. \cite{Loc03}, \cite{HS09}, and \cite{Her12}. In general, for any subgroup $H\le\Aut(G)$ containing $\Inn(G)$, one may consider equivalence classes of $G$-structures mod $H$. If $G = (\ZZ/n\ZZ)^2$, then $H\le\GL_2(\ZZ/n\ZZ)$, and we recover the ``structures of level $H$'' as defined in \cite{DR72} \S IV.3.

\end{pt}

\begin{pt}Technically, the relative fundamental group requires $S$ to be connected. Since the connected components of locally noetherian schemes are open, we may define it there as a product over the connected components. From there, noetherian approximation arguments allow us to extend our constructions above to arbitrary schemes in a way that is compatible with base change.
\end{pt}

\begin{proof} Since $G$ is nontrivial, $\LL$ is nonempty, so the exact sequence of (1) follows from Prop \ref{HES}. 

\sgap

The statements of (2) follow from standard Galois-categorical yoga and the argument surrounding (\ref{eq_sections_of_H(g)}) as long as $E^\circ$ admits a section $g$. Otherwise, we may find a finite \'{e}tale pointed Galois cover $(T,t)\rightarrow (S,s)$ such that $E^\circ/S$ admits a section over $T$ (for example, set $T = E[p]$, where $p\in\LL$). Then, by definition, $G$-structures on $E/S$ are precisely the $G$-structures on $E_T^\circ/T$ which are invariant under $\Gal(T/S)$. The result then follows from noting that if $[\varphi] : \pi_1^\LL(E^\circ_s,\tilde{s})\rightarrow G$ is invariant under both $\pi_1(T,t)$ and $\Gal(T/S)$, then it is also invariant under $\pi_1(S,s)$.

\sgap

For (3), first, if $E^\circ/S$ admits a section, then the splitting of the sequence is the one induced by the splitting. Otherwise, the exact sequence is split by a \emph{tangential morphism} corresponding to the identity section $e : S\rightarrow E$ (c.f. \S 2.1, \cite{Mat97}).

\sgap

There is a bijection (c.f. \cite{SGA1}, \S XI.5) 
between isomorphism classes of connected $G$-torsors on $E^\circ$ and surjective classes of the group cohomology
$$H^1(\pi_1(E^\circ,\tilde{s}),G) = \Hom^\ext(\pi_1(E^\circ,\tilde{s}),G) = \Hom^\ext(\pi_1'(E^\circ,\tilde{s}),G)$$
The splitting of (1) makes $\pi_1'(E^\circ,\tilde{s})$ into the semidirect product
$$\pi_1'(E^\circ,\tilde{s}) \cong \pi_1^\LL(E^\circ,\tilde{s})\rtimes_\rho \pi_1(S,s)$$
By Prop 27 of \cite{BbkGT}, there is a natural bijection
\begin{multline}
\Hom(\pi_1'(E^\circ,\tilde{s}),G) \rightiso \{(\varphi,\psi)\in\Hom(\pi_1^\LL(E^\circ,\tilde{s}),G)\times\Hom(\pi_1(S,s),G) : \nonumber \\
 \forall \gamma\in\pi_1^\LL(E^\circ,\tilde{s})\text{ and }\sigma\in\pi_1(S,s),\text{ we have }\varphi(\rho(\sigma)\gamma) = \psi(\sigma)\varphi(\gamma)\psi(\sigma)^{-1}\}
\end{multline}
Under the bijection, a pair $(\varphi,\psi)$ uniquely determines the homomorphism
$$\varphi\rtimes\psi : \pi_1'(E^\circ,\tilde{s}) \rightarrow G\quad\text{ given by }\quad(\varphi\rtimes\psi)(\gamma,\sigma) := \varphi(\gamma)\psi(\sigma)$$
Considering the quotient of the above sets mod $\Inn(G)$ and projecting onto the first coordinate ``$[\varphi]$'' gives a map
\begin{align}\label{sdp_compatibility}
\{(\varphi,\psi) : \forall\gamma,\sigma,\quad\varphi(\,^{\rho(\sigma)}\gamma) = \psi(\sigma)\varphi(\gamma)\psi(\sigma)^{-1}\} & \stackrel{\tilde{\beta}}{\longrightarrow}\;\Hom^\ext(\pi_1^\LL(E^\circ,\tilde{s}),G) \tag{$\ast$}\\
(\varphi,\psi) & \;\mapsto \;\; [\varphi]  \nonumber
\end{align}
whose image is a surjective homomorphism precisely if the $G$-torsor was geometrically connected. The compatibility condition (\ref{sdp_compatibility}) on $(\varphi,\psi)$ ensures that $\varphi$ corresponds to a $G$-structure (as in (2)). The map $\beta$ indicated in (3) is just the restriction of $\tilde{\beta}$ to the pairs $(\varphi,\psi)$ corresponding to geometrically connected $G$-torsors.

\sgap

Now suppose $Z(G) = 1$, and suppose we have two geometrically connected $G$-torsors corresponding to pairs $(\varphi,\psi)$ and $(\varphi',\psi')$ such that $[\varphi_1] = [\varphi_1']$. Since $Z(G) = 1$ there is a \emph{unique} element $g\in G$ such that $\varphi = g\varphi' g^{-1}$, so since we're working modulo $\Inn(G)$ we may assume $\varphi = \varphi'$. To show that $\beta$ is injective, it suffices to show that $\psi$ is uniquely determined by $\varphi$. Indeed, suppose we have $\psi,\psi'$ with 
$$\varphi(\,^{\rho(\sigma)}\gamma) = \psi(\sigma)\varphi(\gamma)\psi(\sigma)^{-1} = \psi'(\sigma)\varphi(\gamma)\psi'(\sigma)^{-1}\qquad\forall(\gamma,\sigma)\in\pi_1^\LL(E^\circ_s,\tilde{s})\times\pi_1(S,s)$$
then we find that $\psi(\sigma)^{-1}\psi'(\sigma)\varphi(\gamma)\psi'(\sigma)^{-1}\psi(\sigma) = \varphi(\gamma)$ for all $\sigma,\gamma$. Fixing $\sigma$ and letting $\gamma$ range over all of $\pi_1^\LL(E^\circ_s,\tilde{s})$, we get, thanks to the surjectivity of $\varphi$, that $\psi(\sigma)^{-1}\psi'(\sigma) \in Z(G) = 1$. Since this holds for all $\sigma$, we conclude that $\psi = \psi'$, and hence $\beta$ is injective.

\sgap

To show that $\beta$ is surjective, let $\varphi$ be a representative of a $G$-structure $[\varphi]$, so we know that for every $\sigma\in\pi_1(S,s)$, there exists a uniquely determined $g_\sigma\in G$ such that $\varphi\circ\rho(\sigma) = g_\sigma\varphi g_\sigma^{-1}$. For any $\sigma,\tau\in\pi_1(S,s)$ since $\rho$ is a homomorphism, we have:
$$\varphi = \varphi\circ\rho(\sigma)\circ\rho(\tau)\circ\rho(\sigma\tau)^{-1} = g_\sigma g_\tau g_{\sigma\tau}^{-1}\varphi g_{\sigma\tau}g_\tau^{-1} g_\sigma^{-1}$$
Since $\varphi$ is surjective, this shows $g_\sigma g_\tau g_{\sigma\tau}^{-1}\in Z(G) = 1$, so the association $\sigma\mapsto g_\sigma$ defines a homomorphism $\psi : \pi_1(S,s)\rightarrow G$, which is easily checked to satisfy the compatibility relation (\ref{sdp_compatibility}) with $\varphi$.

\end{proof}

Next we show that when $G$ is abelian, $G$-structures correspond to classical congruence level structures. First, recall that:

\begin{defn} Let $E/S$ be an elliptic curve, and $n\in\ZZ$ be invertible on $S$.
\begin{itemize}
\item A $\Gamma(n)$-structure on $E/S$ is an isomorphism:
$$(\ZZ/n\ZZ)^2_S\rightiso E[n]$$
\item A $\Gamma_1(n)$-structure on $E/S$ is an injective group scheme homomorphism:
$$(\ZZ/n\ZZ)_S\hookrightarrow E[n]$$
\item We define a $\Gamma_1(n)^\vee$-structure on $E/S$ to be a injective group scheme homomorphism:
$$\mu_{n,S}\hookrightarrow E[n]$$
\end{itemize}
\end{defn}

\begin{prop}\label{prop_abelian_is_congruence} With notation as above, there are bijections between
\begin{itemize}
\item[(1)] The set of $(\ZZ/n\ZZ)^2$-structures on $E/S$ and the set of $\Gamma(n)$-structures on $E/S$
\item[(2)] The set of $\ZZ/n\ZZ$-structures on $E/S$ and the set of  $\Gamma_1(n)^\vee$-structures on $E/S$
\item[(3)] The set of $\mu_{n,S}$-structures on $E/S$ and the set of $\Gamma_1(n)$-structures on $E/S$
\end{itemize}
\end{prop}
\begin{proof} First note that the inclusion $E^\circ\subset E$ induces a $\pi_1(S,s)$-equivariant map $\pi_1(E^\circ_s,\tilde{s})\rightarrow\pi_1(E_s,\tilde{s})$ which is abelianization at the level of groups. A $(\ZZ/n\ZZ)^2_S$-structure is given by a $\pi_1(S,s)$-equivariant surjection $\varphi : \pi_1(E^\circ_s,\tilde{s})\longrightarrow (\ZZ/n\ZZ)^2_S$. Since the image is abelian, $\varphi$ factors uniquely as:
$$\pi_1(E^\circ_s,\tilde{s})\stackrel{\ab}{\longrightarrow}\pi_1(E_s,\tilde{s})\longrightarrow (\ZZ/n\ZZ)^2_S$$
Thus any $(\ZZ/n\ZZ)^2_S$-structure uniquely determines a $\pi_1(S,s)$-equivariant surjection $\pi_1(E_s,\tilde{s})\longrightarrow(\ZZ/n\ZZ)^2_S$, whence a surjection of group schemes $\pi_1^\LL(E/S,e,s)\rightarrow (\ZZ/n\ZZ)^2_S$. By Lemma \ref{lemma_RFG_is_Tate_module}, this is a surjection $\varprojlim E[n]_S\rightarrow(\ZZ/n\ZZ)^2_S$, which must factor uniquely through its characteristic $n$-torsion quotient $E[n]$, giving us a surjection $E[n]\rightarrow(\ZZ/n\ZZ)^2_S$, which by comparing degrees must be an isomorphism. This proves (1).

\sgap

For (2) and (3), the same argument shows that a $\ZZ/n\ZZ$-structure is equivalent to a surjection $E[n]\rightarrow(\ZZ/n\ZZ)_S$, which upon taking Cartier duals gives an injection $\mu_{n,S}\hookrightarrow E[n]_S$. Similarly, a $\mu_{n,S}$-structure is equivalent to a surjection $E[n]\rightarrow\mu_{n,S}$, and taking duals yields an injection $(\ZZ/n\ZZ)_S\hookrightarrow E[n]$.

\end{proof}

\section{Stacks of elliptic curves with $G$-structures}\label{section_stacks_TLS}
From now on $G$ will always refer to a finite 2-generated group or its corresponding constant group scheme. In this section we will study the stacks of elliptic curves equipped with $G$-structures. Let $N := |G|$, and let $\LL$ be the set of primes dividing $N$. 

\subsection{The definition of $\mM(G)$}\label{ss_definition_of_M(G)}


\begin{thm}\label{thm_M11_is_DM} The moduli stack $\mM(1)$ of elliptic curves is a connected Noetherian smooth separated Deligne-Mumford stack over $\Spec\ZZ$.
\end{thm}
\begin{proof} See, e.g., Theorem 2.5 in \S III and Proposition 2.2 in \S IV of \cite{DR72}. Here, by Deligne-Mumford stack we mean an ``algebraic stack'' in the sense of Definition 4.6 of \cite{DM69}.
\end{proof} 







\begin{prop}\label{prop_TLS_form_a_sheaf} The functor $\pP_G : \mM(1)_{\ZZ[1/N]}\rightarrow\Sets$ defined by sending
$$(E/S)\mapsto \{\text{global sections of }\hHom^\surext(\pi_1(E^\circ/S),G)\}$$
is a sheaf on $\mM(1)_{\ZZ[1/N]}$.
\end{prop}
\begin{proof} This follows from the construction of $\hHom^\surext(\pi_1(E^\circ/S),G)$ in  \S\ref{ss_TLS}.
\end{proof}

\begin{remark*} The functor $\pP_G$ defines a finite \'{e}tale, relatively representable moduli problem for elliptic curves in the sense of \cite{KM85}, \S4.
\end{remark*}

\begin{defn}\label{defn_TMS} We define the fibered category $\mM(G)$ over $\Sch/\ZZ[1/N]$ as follows. Its objects over $S\in\Sch/\ZZ[1/N]$ are pairs $(E/S,\alpha)$, where $\alpha\in\pP_G(E/S)$, and a morphism between $(E/S,\alpha)$ and $(E'/S',\alpha')$ is a cartesian diagram
$$\xymatrix{
(E')^\circ\ar[d]\ar[r]^\phi & E^\circ\ar[d] \\
S'\ar[r] & S
}$$
such that $\phi^*\alpha = \alpha'$.
\end{defn}

By Proposition \ref{prop_TLS_form_a_sheaf}, $\mM(G)$ is a stack in sets over $\mM(1)_{\ZZ[1/N]}$, and hence a stack in groupoids over $\ZZ[1/N]$. Forgetting level structures defines a natural map
\begin{eqnarray*}
\mM(G) & \longrightarrow & \mM(1)_{\ZZ[1/N]} \\
(E/S,\alpha) & \mapsto & (E/S)
\end{eqnarray*}
The pullback of this map by any morphism $S\stackrel{E/S}{\longrightarrow}\mM(1)_{\ZZ[1/N]}$ is precisely the scheme $\hHom^\surext(\pi_1(E^\circ/S),G)$, which is finite \'{e}tale over $S$. This implies:

\begin{prop}\label{prop_GM_finite_\'{e}tale_over_M11} $\mM(G)$ is a Deligne-Mumford stack, finite \'{e}tale over $\mM(1)_{\ZZ[1/N]}$, smooth and separated over $\ZZ[1/N]$.
\end{prop}

We call the stack $\mM(G)$ over $\ZZ[1/N]$ the stack of elliptic curves equipped $G$-structures.

\sgap

\textbf{Notation.} By Proposition \ref{prop_abelian_is_congruence}, $\Gamma(n)$ (resp. $\Gamma_1(n)$)-structures on elliptic curves over $\ZZ[1/n]$ are the same as $(\ZZ/n\ZZ)^2$ (resp. $\mu_n$)-structures. Thus, we will use
$$\mM(n) := \mM((\ZZ/n\ZZ)^2),\qquad \mM_1(n) := \mM(\mu_n)$$
to denote the classical moduli stacks of elliptic curves with (naive) $\Gamma(n),\Gamma_1(n)$-structures over $\ZZ[1/n]$. Let $\zeta_n := e^{2\pi i/n}$, then note that $\mM(n)_{\ZZ[1/n,\zeta_n]}$ has $\phi(n)$ geometrically connected components, each a model of $\hH/\Gamma(n)$ over $\ZZ[1/n,\zeta_n]$. For some $\ZZ[1/n,\zeta_N]$-scheme $S$, let $\yY(n)_S$ denote a geometrically connected component of $\mM(n)_S$. Over $\mM(1)_{\ZZ[1/n]}$, $\mM(n)$ is finite \'{e}tale Galois with Galois group $\GL_2(\ZZ/n\ZZ)$.

\sgap




\begin{remark}\label{remark_alternate_constructions} There are alternative ways of constructing the stacks $\mM(G)$. The first is via the Hurwitz stacks of \cite{BR11} \S6 - Let $\hH(G)$ be the stack over $\ZZ[1/N]$ whose objects are proper smooth marked curves $X$ equipped with a faithful action of $G$, such that $G$ acts freely outside the marking divisor, and such that $E := X/G$ is a pointed curve of genus 1. The automorphisms of $X$ in $\hH(G)$ are the $G$-equivariant automorphisms of $X$. By restricting the cover $X\rightarrow E$ to $E^\circ$, we obtain a $G$-structure on $E$, and indeed this gives a functor $\hH(G)\rightarrow\mM(G)$ which is an \'{e}tale gerbe, and hence gives a bijection on objects over algebraically closed fields.

\sgap

However, since the automorphisms of $X$ viewed in $\mM(G)$ are the automorphisms of $E = X/G$ which admit a lift to an automorphism of $X$, we find that $\Aut_{\mM(G)}(X) = \Aut_{\hH(G)}(X)/Z(G)$, where $Z(G)$ is the center of $G$, which can be identified with the subgroup of $\Aut_{\hH(G)}(X)$ consisting of automorphisms inducing the identity on $E$. In the language of \cite{Rom03} \S5, the map $\hH(G)\rightarrow\mM(G)$ induces an isomorphism $\hH(G)\fs Z(G)\cong \mM(G)$, where $\hH(G)\fs Z(G)$ is the \emph{rigidification} of $\hH(G)$ along $Z(G)$ (c.f. \cite{ACV03} \S5, where they use superscripts ``$\hH(G)^{Z(G)}$'' to denote rigidification).

\sgap

A second alternative is via the language of ``twisted stable maps'' (c.f. \cite{ACV03}, \cite{AV02}), where we start with the stack ``$\bB^\text{tei}_{1,1}(G)$'', which we define to be the rigidification of their stack $\bB^\text{bal}_{1,1}(G)$ (c.f. \cite{ACV03} \S2.2) by $Z(G)$. In this setup, $\bB^\text{bal}_{1,1}(G)$ is analogous to $\hH(G)$, but now the objects are essentially given by morphisms $p : \eE\rightarrow \bB G$, where $\bB G$ is the classifying stack of $G$, and $\eE$ is stable pointed curve of genus 1 with stacky structure at the marking and nodes. Thus, the morphism $p$ corresponds to an \'{e}tale $G$-torsor on $\eE$, which can be viewed as a ramified $G$-cover of the coarse scheme $E$ of $\eE$. From this perspective, $\mM(G)$ can be identified with the open substack of $\bB^\text{tei}_{1,1}(G)$ parametrizing maps $\eE\rightarrow\bB G$ from smooth curves.
\end{remark}

\subsection{The Galois category of $\mM(1)$} \label{ss_Galois_category}
The fact that $\mM(G)$ is finite \'{e}tale over $\mM(1)_{\ZZ[1/N]}$ means that we may study them from the viewpoint of Grothendieck's Galois theory. 

\begin{defn}(Galois Category of $\mM(1)$, see \cite{Noo04} Definition 4.1) Let $S$ be a scheme. We define $\cC_{\mM(1)_S}$ to be the 1-category associated to the 2-category of finite \'{e}tale stacks over $\mM(1)_S$. In other words,
\begin{itemize}
\item $\Ob(\cC_{\mM(1)_S})$ are pairs $(\mM,f)$, where $\mM$ is an algebraic stack over $S$ and $f : \mM\rightarrow\mM(1)_S$ is finite \'{e}tale.
\item $\Hom_{\cC_{\mM(1)_S}}((\mM,f),(\nN,g))$ are morphisms $a : \mM\rightarrow\nN$ such that $f = g\circ a$.
\end{itemize}

Let $k = \bar{k}$ be an algebraically closed field. Pick a geometric point $x_0 : \Spec k\rightarrow\mM(1)_S$ corresponding to an elliptic curve $E_0/k$. The fiber functor $F_{x_0} : \cC_{\mM(1)_S}\rightarrow\Sets$ is defined as follows
$$F_{x_0}(\mM,f) := \Hom_{\Spec k}(\Spec k,\mM_{x_0})$$
where $\mM_{x_0} := \Spec k\times_{x_0,\mM(1),f}\mM$.
\end{defn}

\begin{example}\label{ex_fiber}
If $\mM = \mM(G)_S$ and $f$ the natural ``forget-level-structure'' map, then
$$F_{x_0}(\mM) = \hHom^\surext(\pi_1(E_k^\circ/k),G) = \Hom^\surext(\widehat{F_2}^\LL,G) = \Hom^\surext(F_2,G)$$
where the last equality comes from the fact that $N = |G|$ is only divisible by primes in $\LL$.
\end{example}

\begin{thm}\label{thm_Galois_category_of_stacks} The pair $(\cC_{\mM(1)_S},F_{x_0})$ is a Galois category with fiber functor $F_{x_0}$. In particular, the automorphism group $\Aut(F_{x_0})$ is a profinite group, called the fundamental group $\pi_1(\mM(1)_S)$ of $\mM(1)_S$, and we have an equivalence of categories (\emph{the Galois correspondence})
\begin{eqnarray*}
\cC_{\mM(1)_S} & \rightiso & \text{Finite $\pi_1(\mM(1)_S)$-sets} \\
\mM & \mapsto & F_{x_0}(\mM,f)
\end{eqnarray*}
\end{thm}
\begin{proof} This is theorem 4.2 of \cite{Noo04}. For an introduction to Galois categories, see \cite{Mu67}, \cite{Sza09} \S V, or \cite{SGA1}, \S V.4. 
\end{proof}



The connection with modular curves begins with the following fundamental result of Oda \cite{Oda97}.

\begin{thm}[Oda, \cite{Oda97}]\label{thm_fundamental_group_M11} The fundamental group of $\mM(1)_{\Qbar}$ is $\widehat{\SL_2(\ZZ)}$.
\end{thm}

So far we know that our stacks $\mM(G)$ are objects in the Galois category of algebraic stacks finite \'{e}tale over $\mM(1)$. Theorem \ref{thm_fundamental_group_M11} tells us that over $\Qbar$, the fundamental group of the base stack $\mM(1)_{\Qbar}$ is $\widehat{\SL_2(\ZZ)}$. In Example \ref{ex_fiber}, we gave an explicit description of their geometric fibers as $\Hom^\surext(F_2,G)$. In order to pin down the geometry of $\mM(G)$, it remains to describe the monodromy action of $\pi_1(\mM(1)_{\Qbar})$ on the geometric fibers.




\sgap

\begin{pt} Let $\eE^\circ$ be the universal elliptic curve over $\mM(1)_{\Qbar}$ with the identity section removed, and let $\eE_{x_0}^\circ$ be a geometric fiber above some geometric point $x_0\in\mM(1)_{\Qbar}$ corresponding to an elliptic curve $E_0/\Qbar$. Then $\eE_{x_0} = E_0$, and there is an exact sequence
$$1\rightarrow\pi_1(E_0^\circ)\rightarrow\pi_1(\eE^\circ)\rightarrow\pi_1(\mM(1)_{\Qbar})\rightarrow 1,$$
and hence an outer representation
\begin{equation}\label{eq_stacky_outer_representation}
\widehat{\SL_2(\ZZ)}\cong\pi_1(\mM(1)_{\Qbar})\stackrel{\rho_{\Qbar}}{\longrightarrow}\Out(\pi_1(E_0^\circ)) \cong \Out(\widehat{F_2})
\end{equation}

\sgap

Let $\mM(G)_{\Qbar}\stackrel{p}{\longrightarrow}\mM(1)_{\Qbar}$ be the natural forgetful map. In Example \ref{ex_fiber} we showed that
$$p^{-1}(x_0) = \Hom^\surext(\pi_1^\LL(E_0^\circ),G) = \Hom^\surext(\widehat{F_2},G)$$
It follows from our construction of $G$-structures that the action of $\pi_1(\mM(1)_{\Qbar})$ on these fibers is given by the outer representation $\rho_{\Qbar}$ of (\ref{eq_stacky_outer_representation}).

\sgap

On the other hand, there is an exact sequence
$$1\rightarrow\Inn(F_2)\rightarrow\Aut(F_2)\rightarrow\GL_2(\ZZ)\rightarrow 1$$
identifying $\GL_2(\ZZ)$ with the group of outer automorphisms $\Out(F_2)$. Since any outer automorphism of $F_2$ induces an outer automorphism of $\widehat{F_2}$, we get a natural injection $\SL_2(\ZZ)\stackrel{\rho_{\text{top}}}{\longrightarrow}\Out(F_2)\longrightarrow\Out(\widehat{F_2})$. Automorphism groups of finitely generated profinite groups are themselves profinite (c.f. \S 4.4 \cite{RZ10}), and hence $\rho_\tp$ induces a map
\begin{equation}\label{eq_naive_outer_representation}
\widehat{\SL_2(\ZZ)}\stackrel{\widehat{\rho_{\text{top}}}}{\longrightarrow}\Out(\widehat{F_2})
\end{equation}
\end{pt}



\begin{thm}\label{thm_outer_representation} The representations of (\ref{eq_stacky_outer_representation}) and (\ref{eq_naive_outer_representation}) are isomorphic, ie
$$\rho_{\Qbar} \cong \widehat{\rho_{\text{top}}}$$
In particular, the action of $\mM(1)_{\Qbar}$ on the geometric fiber
$$p^{-1}(x_0) = \Hom^\surext(\widehat{F_2},G)$$
agrees with the natural outer action of $\SL_2(\ZZ)$ on $F_2$ as elements of $\GL_2(\ZZ)\cong\Out(F_2)$.
\end{thm}
\begin{proof} This is stated in \cite{Mat00} p.376. We give a proof of this in Appendix \S\ref{proof_thm_outer_representation}.
\end{proof}


Via the Galois correspondence, Theorem \ref{thm_outer_representation} tells us
\begin{cor}\label{cor_G_to_Gamma} As usual let $G$ be a finite 2-generated group.
\begin{itemize}
\item[(1)] The connected components of the stack $\mM(G)_{\Qbar}$ are in bijection with the orbits of $\SL_2(\ZZ)$ acting on $\Hom^\surext(F_2,G)$ via outer automorphisms of $F_2$.
\item[(2)] For an exterior surjection $[\varphi] : F_2\twoheadrightarrow G\mod\Inn(G)$, the component of $\mM(G)_{\Qbar}$ containing $[\varphi]$ corresponds via the Galois correspondence to the conjugacy class of the stabilizer $\Gamma_{[\varphi]} := \Stab_{\SL_2(\ZZ)}([\varphi])\le\SL_2(\ZZ)$.
\end{itemize}
\end{cor}
Technically the Galois correspondence gives us open subgroups of $\widehat{\SL_2(\ZZ)}$, so we should have said $\ol{\Gamma_{[\varphi]}}$ instead of $\Gamma_{[\varphi]}$, where the bar denotes (topological) closure. However, since $\SL_2(\ZZ)$ is finitely generated, taking closures gives a bijection between the finite index subgroups of $\SL_2(\ZZ)$ and the open subgroups of $\widehat{\SL_2(\ZZ)}$, so we will continue to use this abuse of notation.

We also have the following useful construction:

\begin{prop}\label{prop_descent_of_groups} Let $G$ be group of order $N$. If $f : G\rightarrow G'$ is a surjection, and $\pi$ is a group acting on $\widehat{F_2}$, then we have a \emph{surjection} of $\pi$-sets
$$\Hom^\surext(\widehat{F_2},G)\rightarrow\Hom^\surext(\widehat{F_2},G')$$
given by sending $\varphi : \widehat{F_2}\twoheadrightarrow G$ to $f\circ\varphi$. In particular, we get a surjective finite \'{e}tale morphism
$$\mM(G)\stackrel{f_*}{\longrightarrow} \mM(G')$$
and for every $\varphi : \widehat{F_2}\twoheadrightarrow G$, $\Gamma_{[\varphi]}\subseteq\Gamma_{[f\circ\varphi]}$.
\end{prop}

\begin{proof} The surjectivity follows from Gasch\"{u}tz' lemma (Proposition 2.5.4 of \cite{RZ10}), and the rest follows from the Galois correspondence, setting $\pi = \pi_1(\mM(1)_{\ZZ[1/N]})$.
\end{proof}

\begin{remark} If $H := \ker(G\rightarrow G')$, then the map $\mM(G)\rightarrow\mM(G')$ is geometrically given by sending a $G$-torsor $X\rightarrow E$ to $X/H\rightarrow E$, which is a $G'$-torsor.
\end{remark}

\subsection{Coarse moduli schemes and modular curves}\label{ss_cms}

In this section we first recall the definition of \emph{coarse moduli scheme} associated to an algebraic stack and present some basic properties we'll need. The main reference for this section will be Katz-Mazur's book \cite{KM85}. We will end this section by showing that the coarse moduli schemes of connected finite \'{e}tale covers of $\mM(1)_\CC$ are modular curves.

\sgap

\begin{defn}\label{def_CMS} Let $\mM$ be a stack over $\Sch/S$, then a coarse moduli scheme for $\mM$ is a scheme $M/S$ equipped with a morphism $c : \mM\rightarrow M$ such that
\begin{itemize}
\item[(1)] For any $S$-scheme $X$, any morphism $\mM\rightarrow X$ factors uniquely as
$$\mM\stackrel{c}{\longrightarrow} M\longrightarrow X$$
\item[(2)] For any algebraically closed field $k$, $c$ induces a bijection
$$c(k) : \mM(k)/\!\cong\;\;\rightiso M(k)$$
where $\mM(k)/\!\cong$ is the set of isomorphism classes of objects of $\mM$ over $k$.
\end{itemize}
\end{defn}

\textbf{Notation:} We will always denote a stack by script letters $\mM,\xX,\yY,\ldots$, and their associated coarse moduli spaces by $M,X,Y,\ldots$. In particular, $M(G)$ is the coarse moduli scheme of $\mM(G)$, and $M(n)$ is the coarse moduli scheme of $\mM(n)$, whose connected components over $\ZZ[1/n,\zeta_n]$ are smooth models of $\hH/\Gamma(n)$ over $\ZZ[1/n,\zeta_n]$. We will use $\yY(n)_S$ to refer to a geometrically connected component of $\mM(n)_S$ over some $\ZZ[1/n,\zeta_n]$-scheme $S$, and $Y(n)_S$ its coarse moduli scheme.

\begin{pt} First, we observe that any stack $\mM$ finite over a separated Deligne-Mumford stack is itself separated, so by the Keel-Mori theorem \cite{KM97}, coarse moduli schemes always exist for such stacks. Furthermore, by (1) above, if $\mM$ is representable, then it is represented by its coarse moduli scheme $M$, in which case $M$ is also a \emph{fine moduli scheme} for $\mM$.
\end{pt}

\begin{pt}\label{pt_uniformizability} Given a stack $\mM$ finite \'{e}tale over $\mM(1)_S$, if there is a prime $p$ invertible on $S$, then we may construct its coarse moduli scheme as follows. We know that $\mM(p^2)$ is representable and finite \'{e}tale over $\mM(1)_S$ with Galois group $\GL_2(\ZZ/p^2\ZZ)$. Thus, $\mM\times_{\mM(1)_S}\mM(p^2)_S$ is also Galois over $\mM$ with group $\GL_2(\ZZ/p^2\ZZ)$, and so as stacks, we have:
$$\mM = [(\mM\times_{\mM(1)_S}\mM(p^2)_S)/\GL_2(\ZZ/p^2\ZZ)]$$
Since $\mM(p^2)$ is representable, so is $\mM\times_{\mM(1)_S}\mM(p^2)_S$, and moreover it is affine over $S$, so we may define $M$ to be the quotient (in $\Sch/S$)
$$M := (\mM\times_{\mM(1)_S} M(p^2)_S)/\GL_2(\ZZ/p^2\ZZ)$$
The universal property of quotients implies that $M$ is a coarse moduli scheme for $\mM$.
\end{pt}

\begin{prop}\label{prop_basic_results_cms} Let $\mM$ be any stack finite \'{e}tale over $\mM(1)_{\ZZ[1/N]}$ with coarse moduli scheme $M$. Then
\begin{itemize}
\item[(1)] $M$ is smooth over $\ZZ[1/N]$.
\item[(2)] If $S$ is either a regular Noetherian scheme over $\ZZ[1/N]$ or $6$ divides $N$, then $M_S := M\times_{\ZZ[1/N]} S$ is the coarse moduli scheme of $\mM_S$.
\item[(3)] $M$ is finite over the $j$-line $\Spec \ZZ[1/N][j] = M(1)_{\ZZ[1/N]}$.
\item[(4)] The ``coarse map'' $c : \mM\rightarrow M$ induces a homeomorphism on underlying topological space $|c| : |\mM|\rightarrow |M|$.
\end{itemize}
\end{prop}

\begin{proof} For (1) and (2), see the appendix of Katz-Mazur \cite{KM85}, p508-510 and also 8.1.6. For (3), see \cite{KM85} Proposition 8.2.2. Lastly, (4) follows from the definition of coarse moduli schemes.
\end{proof}

\begin{pt} \textit{Modular Curves and GAGA.}\label{subsection_cms_are_modular_curves}
Given a scheme $X$ locally of finite type over $\CC$, there is an associated complex analytic space $X^\an$ whose underlying set is $X(\CC)$ (c.f. \S XII of \cite{SGA1}). The association is functorial and sends finite \'{e}tale morphisms to finite covering maps. For some object $A$, let $\cC_A$ denote the Galois category of $A$ (that is, the category of ``finite covering spaces'' of $A$). Then for a scheme $X$ locally of finite type over $\CC$, the classical Riemann existence theorem gives us an equivalence of categories
$$\cC_X\cong \cC_{X^\an}$$


\sgap

Similarly, let $\xX$ be an algebraic stack smooth and locally of finite type over $\CC$, then we may associate to it a smooth analytic stack $\xX^\an$ in the following manner (c.f. \cite{BN05} and \cite{Noo05}). Let $\xX\cong [X/R]$ be a presentation 
 for $\xX$, then we define $\xX^\an$ to be simply the quotient $[X^\an/R^\an]$ in the category of analytic stacks. 

\sgap

For $\xX = \mM(1)_\CC$, we may take $X = Y(3)_\CC := \hH/\Gamma(3)$, and $R$ to be the relation generated by the group action of $\SL_2(\ZZ/3\ZZ)$ on $X = Y(3)_\CC$. The associated analytic stack $\mM(1)_\CC^\an$ is by definition the stack $[Y(3)_\CC^\an/\SL_2(\ZZ/3\ZZ)]$. It follows directly from the definitions that
$$\mM(1)_\CC^\an = [Y(3)_\CC^\an/\SL_2(\ZZ/3\ZZ)]\cong [\hH/\SL_2(\ZZ)]$$
as analytic stacks.

\end{pt}

\begin{thm}\label{thm_stacky_RET} (Riemann Existence Theorem for Stacks, \cite{Noo05}) Let $\xX$ be an algebraic stack that is locally of finite type over $\CC$, and let $\xX^\an$ be the associated analytic stack. The analytification functor $\yY\mapsto\yY^\an$ defines an equivalence between the category of representable finite \'{e}tale maps $\yY\rightarrow\xX$ and the category of finite covering analytic stacks of $\xX^\an$.
\end{thm}
\begin{proof} See Theorem 20.4 in \cite{Noo05}. 
\end{proof}


Thus, the fundamental group of the Galois category of finite covers of the analytic stack $[\hH/\SL_2(\ZZ)]$ is also $\widehat{\SL_2(\ZZ)}$. By comparing the monodromy actions we see that for any finite index $\Gamma\le\SL_2(\ZZ)$, the analytic stack corresponding to the open subgroup $\ol{\Gamma}\le\widehat{\SL_2(\ZZ)}$ is just $[\hH/\Gamma]$, and its coarse space is just the Riemann surface $\hH/\Gamma$ and is characterized by the natural morphism $[\hH/\Gamma]\rightarrow\hH/\Gamma$ which induces a bijection on points and is universal amongst all morphisms from $[\hH/\Gamma]$ to analytic spaces.

\begin{prop}\label{prop_cms_are_modular_curves} Let $\mM$ be a connected algebraic stack finite \'{e}tale over $\mM(1)_\CC$ corresponding to some finite index subgroup $\Gamma\le\SL_2(\ZZ)$. Then $\mM^\an \cong [\hH/\Gamma]$, and the analytification $M^\an$ of the coarse moduli scheme $M$ of $\mM$ is the modular curve $\hH/\Gamma$.

\sgap

In particular, if $[\varphi] : F_2\twoheadrightarrow G$ represents a geometric point of $\mM(G)_{\CC}$ (c.f. Example \ref{ex_fiber}), then the (analytification of the) connected component of its coarse scheme $M(G)_{\CC}$ containing $[\varphi]$ is precisely the modular curve $\hH/\Gamma_{[\varphi]}$, where $\Gamma_{[\varphi]} := \Stab_{\SL_2(\ZZ)}([\varphi])$.
\end{prop}
\begin{proof} We first prove that $M^\an \cong\hH/\Gamma$. By Theorem \ref{thm_stacky_RET}, and the discussion above, $\mM^\an \cong [\hH/\Gamma]$. Let $M$ be the coarse moduli space for $\mM$, and $\mM\stackrel{c}{\rightarrow}M$ the coarse map, then we have a commutative diagram
$$\xymatrix{
\mM\ar@{~>}[r]^\an\ar[d]^c & \mM^\an\ar[r]^\cong\ar[d]^{c^\an} & [\hH/\Gamma]\ar@{-->}[d]\\
M\ar@{~>}[r]^\an & M^\an & \hH/\Gamma\ar@{-->}[l]
}$$
where the squiggly arrows are really functors, all arrows induce bijections on (equivalence classes) of $\CC$-points, and the dashed arrows are holomorphic maps uniquely determined by $c^\an$ and the universal property of coarse moduli spaces. The first result then follows from the fact that bijective holomorphic maps are isomorphisms.

\sgap

The second statement is a direct consequence of Corollary \ref{cor_G_to_Gamma}.
\end{proof}

\subsection{Modular curves and the congruence subgroup property}\label{ss_CSP}
A congruence subgroup of $\SL_2(\ZZ)$ is a subgroup which contains a principal congruence subgroup $\Gamma(n)$, $n\ge 1$, where
$$\Gamma(n) := \ker\big(\SL_2(\ZZ)\rightarrow\SL_2(\ZZ/n\ZZ)\big)$$
On the other hand, we may view $\SL_2(\ZZ)$ as a subgroup of $\GL_2(\ZZ) = \Aut(\ZZ^2)$. For any surjection $\psi : \ZZ^2\twoheadrightarrow H$ onto a finite group $H$ with kernel $K_\psi$, we may consider the ``congruence subgroup''
$$\begin{array}{rcl}
\Gamma[K_\psi] & := & \{\gamma\in\SL_2(\ZZ) : \gamma(K_\psi) = K_\psi\text{ and $\gamma$ acts trivially on $H$}\} \\
 & = & \{\gamma\in\SL_2(\ZZ) : \psi\circ\gamma = \psi\}
\end{array}$$

Since any surjection from $\ZZ^2$ onto a finite group of order $n$ must factor through the characteristic quotient $\ZZ^2\twoheadrightarrow(\ZZ/n\ZZ)^2$, we find that the systems $\{\Gamma(n)\}_{N\ge 1}$ and $\{\Gamma[K_\psi]\}_{\psi}$ induce equivalent (profinite) topologies on $\SL_2(\ZZ)$ (c.f. \cite{RZ10} \S3.1). The fact that $\SL_2(\ZZ)$ does \emph{not} have the \emph{congruence subgroup property} then implies the existence of finite index subgroups of $\SL_2(\ZZ)$ which do not contain any subgroup of the form $\Gamma[K_\psi]$.

\sgap

Similarly, we may consider surjections $\varphi : F_2\twoheadrightarrow G$ with kernel $K_\varphi$, where $G$ is a finite group, and look at the ``congruence subgroups'' $\Gamma[K_\varphi]\le\Aut(F_2)$
$$\begin{array}{rcl}
\Gamma[K_\varphi] & := & \{\gamma\in\Aut(F_2) : \gamma(K_\varphi) = K_\varphi\text{ and $\gamma$ acts trivially on $G$}\} \\
 & = & \{\gamma\in\Aut(F_2) : \varphi\circ\gamma = \varphi\} 
\end{array}$$
which are precisely the preimages of the stabilizers $\Gamma_{[\varphi]}$ in $\Aut(F_2)$ (c.f. Cor \ref{cor_G_to_Gamma} and Prop \ref{prop_cms_are_modular_curves}). In this case we find that the situation is the opposite:

\begin{thm}[Congruence Subgroup Property for $\Aut(F_2)$]\label{thm_CSP} Let $G$ be a finite group and $\varphi : F_2\rightarrow G$ a surjective homomorphism with kernel $K_\varphi$. Let
$$\Gamma[K_\varphi] := \{\gamma\in\Aut(F_2) : \varphi\circ\gamma = \varphi\}$$
Then every finite index subgroup of $\Aut(F_2)$ contains a subgroup of the form $\Gamma[K_\varphi]$.
\end{thm}
\begin{proof} This was originally proved by Asada \cite{Asa01}. See \cite{BER11} for an exposition of the congruence subgroup property.
\end{proof}

\sgap

\begin{cor}\label{cor_CSP} Let $\Gamma$ be a finite index subgroup of $\SL_2(\ZZ)\subset \Out(F_2)$. Then there exists a finite group $G$ and an exterior surjection $[\varphi] : F_2\twoheadrightarrow G$ such that $\Gamma\supset\Gamma_{[\varphi]} := \Stab_{\SL_2(\ZZ)}([\varphi])$.

\sgap

In particular, for every modular curve $\hH/\Gamma$ corresponding to a finite index subgroup $\Gamma\le\SL_2(\ZZ)$, there exists a finite group $G$ and exterior surjection $[\varphi] : F_2\twoheadrightarrow G$ such that $\hH/\Gamma$ is a quotient of $\hH/\Gamma_{[\varphi]}$. I.e., every modular curve has a moduli interpretation.

\end{cor}
\begin{proof} 
The fact that $\hH/\Gamma$ can be covered by some $\hH/\Gamma_{[\varphi]}$ follows from the Galois correspondence, the above discussion, and Prop \ref{prop_cms_are_modular_curves}. To show that $\hH/\Gamma$ is a quotient, it suffices to construct a Galois closure of $\hH/\Gamma_{[\varphi]}$. Let $\SL_2(\ZZ)\cdot[\varphi] = \{[\varphi_1],\ldots,[\varphi_n]\}$ be the $\SL_2(\ZZ)$-orbit of $[\varphi]$. Consider the map $\prod_{i=1}^n\varphi_i : F_2\longrightarrow G^n$. This map may not be surjective, so let $H$ be its image. It follows from the Galois correspondence that the component of $\mM(H)$ containing $\prod_i\varphi_i$ is the Galois closure of $\hH/\Gamma_{[\varphi]}$ (c.f. Prop \ref{prop_descent_of_groups}).
\end{proof}

\begin{remark} In \cite{EM11}, Ellenberg and McReynolds further prove that \emph{every subgroup of $\Gamma(2)$ containing $\pm I$ is a Veech group}. If we define a $(G)$-structure as a $G$-structure modulo $\Aut(G)$ (c.f. \ref{pt_coarser_structures}), then the result of \cite{EM11} says that every subgroup $\Gamma$ of $\Gamma(2)$ containing $\pm I$ is a stabilizer of a $(G)$-structure, and thus every modular curve over $\hH/\Gamma(2)$ has a moduli interpretation parametrizing $(G)$-structures, and is a component of a quotient of some $M(G)_{\Qbar}$.

\sgap

The group-theoretic proof of the congruence subgroup property for $\Aut(F_2)$ given in \cite{BER11} provides an explicit procedure which given a finite index $\Gamma\le\SL_2(\ZZ)$, produces a normal subgroup $K\lhd F_2$ with $\Gamma[K]\le \Gamma$. However, the subgroup $K$ produced this way almost always has \emph{very} high index in $F_2$.

\end{remark}

\subsection{Automorphisms and representability}\label{ss_representability}

In this section we address the issue of representability. For a finite index subgroup $\Gamma\le\SL_2(\ZZ)$, the existence of torsion in $\Gamma$ (equivalently, fixed points of the action of $\Gamma$ on $\hH$) presents an obstruction to constructing a universal family of elliptic curves over $\hH/\Gamma$. In this section we will show that this is the only obstruction.

\begin{lemma}\label{lemma_closure_of_TF_is_TF} If $\Gamma\le\SL_2(\ZZ)$ is torsion-free, then its closure $\ol{\Gamma}$ is also torsion-free.
\end{lemma}
\begin{proof} Theorem 3.2 of \cite{Kul91} shows that every finite index torsion-free subgroup of $\SL_2(\ZZ)$ is actually free and finitely generated. The closure of finitely generated subgroup inside a profinite completion is itself a profinite completion, so the result follows from the fact that profinite free groups are torsion-free (c.f. \cite{RZ10}, Cor 7.7.6).
\end{proof}

The key result we need is:

\begin{lemma}[Rigidity]\label{lemma_rigidity} $\;$
\begin{itemize}
\item[(1)] Let $X/S$ be an abelian scheme over a connected scheme $S$, and let $\sigma\in\Aut_S(X)$. If there is a point $s\in S$ such that $\sigma|_{X_s} = \id_{X_s}$, then $\sigma = \id$.
\item[(2)] Let $S$ be a scheme, and $\mM$ a stack affine over $\mM(1)_S$, then $\mM$ is representable (by a scheme) if and only if for any $S$-scheme $T$ and any object $x\in\mM(T)$, $\Aut_{\mM(T)}(x) = \{1\}$.
\item[(3)] Let $G$ be a finite 2-generated group of order $N$ and $S$ be a scheme on which $N$ is invertible. Let $\mM$ be a component of $\mM(G)_S$ finite \'{e}tale over $\mM(1)_S$, then $\mM$ is representable (by a scheme) if and only if its geometric points have trivial automorphism groups.
\end{itemize}
\end{lemma}
\begin{proof} Item (1) follows from Corollary 6.2 of \cite{GIT}, and item (2) is Theorem 4.7.0 of \cite{KM85}. Item (3) is what we'll need and follows from (1) and (2).
\end{proof}



\begin{thm}\label{thm_torsion_free_implies_representable} Let $G$ be a finite group of order $N$, and let $S$ be a scheme over $\ZZ[1/N]$. Let $x\in\mM(1)_S$ be a geometric point. Let $\mM_S$ be a connected component of $\mM(G)_S$, with forgetful map $p : \mM_S\rightarrow\mM(1)_S$. Then $\mM_S$ is representable by a scheme if and only if every $[\varphi]\in p^{-1}(x)\subseteq\Hom^\surext(F_2,G)$, we have that $\Gamma_{[\varphi]} := \Stab_{\SL_2(\ZZ)}([\varphi])$ is torsion-free.
\end{thm}

\begin{remark} Note that $\mM_S$ may not be geometrically connected (it may decompose into multiple components over an extension of $S$). 
\end{remark}

\begin{proof} We may assume $S = \Spec R$, where $R$ is an integral domain of characteristic 0. If $\mM_R$ is representable, then let $K$ be the fraction field of $R$, so $\mM_{\ol{K}}$ must also be representable, and hence by Prop \ref{prop_cms_are_modular_curves}, the components of $\mM_{\ol{K}}^\an$ are representable analytic stacks $[\hH/\Gamma_{[\varphi]}]$ for $[\varphi]\in p^{-1}(x)$. However, a quotient $[\hH/\Gamma_{[\varphi]}]$ of analytic stacks can only be representable if $\Gamma_{[\varphi]}$ acts without fixed points, and hence as a subgroup of $\SL_2(\ZZ)$, this means that it must be torsion-free.

\sgap

For the other direction, we begin by proving the statement for $\mM_{\Qbar}$, and then use Deuring's lifting theorem and the rigidity lemma to deal with the positive characteristic case.

\sgap

\textbf{Claim:} $\mM_{\Qbar}$ is representable iff each $\Gamma_{[\varphi]}$ is torsion-free.

\sgap

Indeed, by considering each connected component separately, we may assume $\mM_{\Qbar}$ is connected. Then $\mM_{\Qbar}$ corresponds to $\ol{\Gamma_{[\varphi]}}$ for some $[\varphi]\in p^{-1}(x)$ as above, which by Lemma \ref{lemma_closure_of_TF_is_TF} is also torsion-free. Further, we have $\pi_1(\mM_{\Qbar}) = \ol{\Gamma_{[\varphi]}}$. Since $\mM_{\Qbar}$ is uniformizable (c.f. \ref{pt_uniformizability}), the automorphism groups of its geometric points (which must have order dividing 6, and are called \emph{hidden fundamental groups} in \cite{Noo04}) inject into $\pi_1(\mM)\cong\ol{\Gamma_{[\varphi]}}$ (\cite{Noo04} Thm 6.2). Since $\ol{\Gamma_{[\varphi]}}$ is torsion-free, they must be trivial. The claim follows from the rigidity lemma \ref{lemma_rigidity}.



\sgap

To complete the proof, by Lemma \ref{lemma_rigidity}(3), it suffices to show that if $E_0/k$ is an elliptic curve over an algebraically closed field $k$, and $[\varphi_0] : \pi_1(E_0^\circ/k)\twoheadrightarrow G$ an exterior surjection in $p^{-1}(x)$, then there does not exist an automorphism $\alpha_0\in\Aut_k(E_0)$ such that $[\varphi_0\circ(\alpha_0)_*] = [\varphi_0]$. If $\ch(k) = 0$, then is settled by the triviality of the hidden fundamental groups in the above claim, so we may assume $\ch(k) = p > 0$.

\sgap

Let $\alpha_0\in\Aut_k(E_0)$ be such that $[\varphi_0\circ(\alpha_0)_*] = [\varphi_0]$. By Deuring's lifting theorem (c.f. \cite{Lang87} \S13.5), there exists an elliptic curve $E$ over a characteristic 0 complete discrete valuation ring $A$ with fraction field $K$, maximal ideal $\mf{p}$ and residue field $A/\mf{p}\cong k$, such that $E_k := E\times_A k \cong E_0$, and an automorphism $\alpha\in\Aut_A(E)$ whose restriction to $E_k$ is $\alpha_0$. Note that since $k$ is algebraically closed, $\pi_1(\Spec A)$ is trivial, and thus $[\varphi_0]$ extends to give a $G$-structure $[\varphi]$ on $E$.


\sgap

The specialization homomorphism (\cite{SGA1} Cor 3.9) then gives us an isomorphism of prime-to-$p$ fundamental groups $\pi_1^{(p)}(E_{\ol{K}})\rightiso \pi_1^{(p)}(E_0)$ and hence an exterior surjection
$$[\varphi_{\ol{K}} : \pi_1^{(p)}(E_{\ol{K}})\rightiso \pi_1^{(p)}(E_0)\stackrel{\varphi_0}{\longrightarrow} G]$$
which is fixed by $\alpha$ (note that $p\nmid N$). Since the morphisms
$$\Spec\ol{K}\stackrel{[\varphi_{\ol{K}}]}{\longrightarrow}\mM(G)_R,\quad\text{and}\quad\Spec k\stackrel{[\varphi_0]}{\longrightarrow}\mM(G)_R$$
determined by the $G$-structures $[\varphi_{\ol{K}}]$ and $[\varphi_0]$ both factor through $\Spec A\stackrel{[\varphi]}{\longrightarrow}\mM(G)_R$, they must both land in the same connected component, namely $\mM_R$. Thus, since $[\varphi_{\ol{K}}]\in\mM_{\Qbar}$, and since we already know $\mM_{\Qbar}$ is representable, we must have $\alpha = \id$, and hence $\alpha_0 = \id$.
\end{proof}

\begin{cor}\label{cor_representable_cofinal} The collection of stacks
$$\{\mM(G)_{\QQ} : \text{$G$ is finite, 2-generated, and $\mM(G)_{\QQ}$ is representable}\}$$
form a cofinal system in the category of stacks finite \'{e}tale over $\mM(1)_{\QQ}$
\end{cor}
\begin{proof} It suffices to show that every $\mM(G)_{\QQ}$ is covered by a representable stack of the form $\mM(G')_{\QQ}$. Let $p\nmid |G|$ be a prime with $p\ge 5$, then any $\ZZ/p\ZZ$-structure $[\varphi] : F_2\twoheadrightarrow \ZZ/p\ZZ$ has stabilizer conjugate to $\Gamma_1(p)$, and hence is torsion-free, so by the above theorem, $\mM(\ZZ/p\ZZ)_{\QQ}$ is representable.

\sgap

Taking $G' := G\times\ZZ/p\ZZ$, we find that $G'$ is still 2-generated and admits both $G$ and $\ZZ/p\ZZ$ as quotients. Thus, by \ref{prop_descent_of_groups}, $\mM(G')_\QQ$ is representable and covers $\mM(G)_\QQ$.
\end{proof}

\subsection{Basic properties of $\mM(G)$}\label{ss_basic_properties} In this section we will collect the basic properties of the stacks $\mM(G)$ developed in the preceding sections, and briefly discuss the action of $\Gal(\Qbar/\QQ)$ on the connected components of $\mM(G)_{\Qbar}$. As usual, $G$ will be a 2-generated finite group of order $N$. 

\sgap

Firstly, the following is just Proposition \ref{prop_GM_finite_\'{e}tale_over_M11}.
\begin{thm}\label{thm_basic_1} The stack $\mM(G)$ is a smooth separated Deligne-Mumford stack over $\ZZ[1/N]$, finite \'{e}tale over $\mM(1)_{\ZZ[1/N]}$.
\end{thm}

By Proposition \ref{prop_basic_results_cms} we have:
\begin{thm}\label{thm_basic_2} The coarse moduli scheme $M(G)$ of $\mM(G)$ is a Noetherian normal affine scheme, smooth in relative dimension 1 over $\ZZ[1/N]$, and finite over the $j$-line
$$\Spec \ZZ[1/N][j] = M(1)_{\ZZ[1/N]}$$
For any $S\in\Sch/\ZZ[1/N]$, $M(G)_S := M(G)\times_{\ZZ[1/N]}S$ is the coarse moduli scheme of $\mM(G)_S$. 
\end{thm}


Forgetting level structures defines a finite \'{e}tale morphism
$$p : \mM(G)\longrightarrow\mM(1)_{\ZZ[1/N]}$$
Let $x_0 : \Spec\Qbar\rightarrow \mM_{1,1}$ correspond to an elliptic curve $E_0/\Qbar$. Then example \ref{ex_fiber} showed that the fiber $p^{-1}(x_0)$ can be identified with
$$p^{-1}(x_0) = \hHom^\surext(\underbrace{\pi_1(E^\circ_0/\Qbar)}_{\cong\widehat{F_2}},G) = \Hom^\surext(F_2,G)$$
By Theorem \ref{thm_outer_representation}, the action of $\pi_1(\mM(1)_{\Qbar})\cong\widehat{\SL_2(\ZZ)}$ on $p^{-1}(x_0)$ is given by the natural action of $\SL_2(\ZZ)$ acting on the group $F_2$ via outer automorphisms of determinant 1. For any surjective homomorphism $\varphi : F_2\rightarrow G$, let $[\varphi] $ denote the corresponding exterior homomorphism (ie, class mod $\Inn(G)$), which we may think of as an element of $p^{-1}(x_0)$. Let $\Gamma_{[\varphi]} := \Stab_{\SL_2(\ZZ)}([\varphi])$ be its stabilizer in $\SL_2(\ZZ)$. Then by the Galois correspondence (c.f. \ref{cor_G_to_Gamma} and \ref{prop_descent_of_groups}), we have

\begin{thm}\label{thm_orbits_components} The connected components of $\mM(G)_{\Qbar}$ are in bijection with the orbits of $\SL_2(\ZZ)$ on the set $\Hom^\surext(F_2,G)$. The component of $\mM(G)_{\Qbar}$ containing $[\varphi]$ corresponds to the subgroup $\ol{\Gamma_{[\varphi]}}\subset\widehat{\SL_2(\ZZ)}$. If $G'$ is another group with surjection $f : G\twoheadrightarrow G'$, then $f$ induces a surjective map $\Hom^\surext(F_2,G)\rightarrow\Hom^\surext(F_2,G')$, whence a surjective map
$$\mM(G)\stackrel{f_*}{\longrightarrow} \mM(G')$$
sending $[\varphi] \mapsto [f\circ\varphi]$ on geometric fibers. In particular, $\Gamma_{[\varphi]}\le\Gamma_{[f\circ\varphi]}$.
\end{thm}




The connected components of $M(G)_{\Qbar}$ may not be defined over $\QQ$. For example, if $G = (\ZZ/n\ZZ)^2$, then for any surjection $\psi : F_2\twoheadrightarrow(\ZZ/n\ZZ)^2$, the stabilizer $\Gamma_{[\psi]}$ is $\Gamma(n)$, and there are $\phi(n)$ orbits of $\SL_2(\ZZ)$ on $\Hom^\surext(F_2,(\ZZ/n\ZZ)^2)$ corresponding to the $\phi(n)$ possible determinants of matrices over $\ZZ/n\ZZ$. The Galois equivariance of the Weil pairing implies that $M((\ZZ/n\ZZ)^2)_\QQ = M(n)_\QQ$ is connected, and its geometric components are defined over $\QQ(\zeta_n)$, where $\zeta_n := e^{2\pi i/N}$.

\sgap

\textbf{Notation (Important).} From now on for some $G$-structure $[\varphi]$ and $\ZZ[1/N]$-scheme $S$, let $M(G)_S([\varphi])$ (resp. $\mM(G)_S([\varphi])$) be the connected component of $M(G)_S$ (resp. $\mM(G)_S$) containing $[\varphi]$. If the component is furthermore \emph{geometrically connected}, then we will refer to it as $Y([\varphi])_S$ (resp. $\yY([\varphi])_S$).

\sgap

Given the base change diagram

$$\xymatrix{
\mM(G)_{\Qbar}\ar[d]\ar[r] & \mM(G)_\QQ\ar[d] \\
\Spec\Qbar\ar[r] & \Spec\QQ
}$$

the natural action of $G_\QQ := \Gal(\Qbar/\QQ)$ on the bottom induces an automorphism of $\mM(G)_{\Qbar}$ over $\mM(G)_\QQ$, and whence an action on the set of connected components $\pi_0(\mM(G)_{\Qbar})$ of $\mM(G)_{\Qbar}$. Note that this is not in general an automorphism of the cover $\mM(G)_{\Qbar}/\mM(1)_{\Qbar}$.

\sgap

Note that $Y([\varphi])_{\Qbar}$ is defined over the fixed field $K_{[\varphi]}$ of $\Stab_{\Gal(\Qbar/\QQ)} (Y([\varphi]))$, and by Proposition \ref{prop_cms_are_modular_curves}, $Y([\varphi])_\CC \cong \hH/\Gamma_{[\varphi]}$. Thus $K_{[\varphi]}$ is a number field and is the minimal extension of $\QQ$ with the property that there exists a connected component of $M(G)_{K_{[\varphi]}}$ whose inverse image in $M(G)_{\Qbar}$ is $Y([\varphi])_{\Qbar}$. Thus, there is a model of $Y([\varphi])_{\Qbar}$ defined over $K_{[\varphi]}$ which extends to a smooth model of $\hH/\Gamma_{[\varphi]}$ over the ring of integers $\ZZ_{K_{[\varphi]}}[1/N]$. We have the following:

\begin{thm} For any $[\varphi]\in p^{-1}(x_0)$, the connected component $Y([\varphi])_{\Qbar}\subset M(G)_{\Qbar}$ is a $\Qbar$-model of $\hH/\Gamma_{[\varphi]}$. Further, there is a number field $K_{[\varphi]}$, determined by $[\varphi]$, and a smooth model $Y([\varphi])_{\ZZ_{K_{[\varphi]}}[1/N]}$ of $Y([\varphi])_{\Qbar}$ over $\ZZ_{K_{[\varphi]}}[1/N]$.
\end{thm}

Next, by Theorem \ref{thm_torsion_free_implies_representable}, we have

\begin{thm} Let $S$ be a scheme over $\ZZ[1/N]$. Let $\mM_S$ be a connected component of $\mM(G)_S$, then $\mM_S$ is representable by a scheme if and only if every
$$[\varphi]\in \big(p^{-1}(x_0)\cap\mM_S\big)\subset\Hom^\surext(F_2,G),$$
the stabilizer $\Gamma_{[\varphi]}$ is torsion-free.
\end{thm}

We recall that the congruence subgroup property (c.f. Corollary \ref{cor_CSP}) implies that every modular curve has a moduli interpretation:

\begin{thm} For every finite index $\Gamma\le\SL_2(\ZZ)$, there exists a 2-generated group $G$ and a $G$-structure $[\varphi]$ such that $\Gamma_{[\varphi]}$ is a normal subgroup of $\Gamma$ --- that is, $\hH/\Gamma$ is a quotient of $Y([\varphi])_{\CC}\subset M(G)_\CC$.
\end{thm}

We now conclude this section with a discussion of the action of $G_\QQ$ on $\pi_0(\mM(G)_{\Qbar})$.

\sgap

Let $E$ be an elliptic curve over $\QQ$, then $p^{-1}(E)$ is a scheme over $\QQ$, and hence the action of $G_\QQ$ on $\mM(G)_{\Qbar}$ restricts to an action on $p^{-1}(E_{\Qbar})$, which one can verify is identical to the monodromy action of $\pi_1(\Spec\QQ) = G_\QQ$ on the geometric fiber $p^{-1}(E_{\Qbar})$ of $p^{-1}(E)$. We will now give a combinatorial description of the $G_\QQ$ action on $\pi_0(\mM(G)_{\Qbar})$.

\sgap

Consider the exact sequence (c.f. \cite{Zoo01}, Corollary 6.6)
$$1\rightarrow\underbrace{\pi_1(\mM(1)_{\Qbar})}_{\widehat{\SL_2(\ZZ)}}\rightarrow\pi_1(\mM(1)_\QQ)\rightarrow G_\QQ\rightarrow 1$$
The elliptic curve $E/\QQ$ induces a section $\pi_1(\mM(1)_\QQ)\leftarrow G_\QQ$, and hence we may view $\pi_1(\mM(1)_\QQ)$ as being generated by the subgroups $\widehat{\SL_2(\ZZ)}$ and $G_\QQ$. Via the Galois correspondence, we may identify $\mM(G)_\QQ$ with the set $p^{-1}(E_{\Qbar})$ equipped with the action of $\pi_1(\mM(1)_\QQ)$ (and hence also an action of $\widehat{\SL_2(\ZZ)}$ and $G_\QQ$). Then, $\mM(G)_{\Qbar}$ corresponds to the same set $p^{-1}(E_{\Qbar})$ equipped with the action of $\widehat{\SL_2(\ZZ)}$ via pullback
. Since $\widehat{\SL_2(\ZZ)}$ sits as a normal subgroup of $\pi_1(\mM(1)_\QQ)$, the action of $G_\QQ$ on $p^{-1}(E_{\Qbar})$ induces a well-defined action on the $\widehat{\SL_2(\ZZ)}$-orbits of $p^{-1}(E_{\Qbar})$. This is precisely the action of $G_\QQ$ on $\pi_0(\mM(G)_{\Qbar})$. From this description, we also find that if a connected component $\yY([\varphi])_{\Qbar}\subset\mM(G)_{\Qbar}$ corresponds to a finite index subgroup $\ol{\Gamma}\le\widehat{\SL_2(\ZZ)}$, then for any $\sigma\in G_\QQ$, the action of $\sigma$ on $\pi_0(\mM(G)_{\Qbar})$ will send $\yY([\varphi])_{\Qbar}$ to the component $\yY(^\sigma[\varphi])_{\Qbar}$, corresponding to the subgroup $\sigma\ol{\Gamma}\sigma^{-1}\le\widehat{\SL_2(\ZZ)}$. Note that since $\sigma\notin\widehat{\SL_2(\ZZ)}$, conjugation by $\sigma$ only acts as an automorphism of $\widehat{\SL_2(\ZZ)}$, not necessarily an inner automorphism. Since automorphisms preserve the index of a subgroup, we find that $\ol{\Gamma}$ and $\sigma\ol{\Gamma}\sigma^{-1}$ have the same index in $\widehat{\SL_2(\ZZ)}$ and hence $\yY([\varphi])_{\Qbar}$ and $\yY(^\sigma[\varphi])_{\Qbar}$ have the same degree over $\mM(1)_{\Qbar}$. This proves the first statement of the next result.

\sgap

Before presenting the result, we recall that a point $x$ on a modular curve $\hH/\Gamma$ is called an \emph{elliptic point} of order 2 (resp. 3) if the natural map $\hH/\Gamma\rightarrow M(1)_\CC$ (with coordinate $j$) sends $x$ to $j = 1728$ (resp. $j = 0$) and is unramified at $x$.

\begin{prop}\label{prop_galois_action_preserves_signature} Let $\yY([\varphi])_{\Qbar}$ be a connected component of $\mM(G)_{\Qbar}$, then for any $\sigma\in G_\QQ$, the action of $\sigma$ on $\pi_0(\mM(G)_{\Qbar})$ sends $\yY([\varphi])_{\Qbar}$ to $\yY(^\sigma[\varphi])$, where both components must have the same degree, their coarse moduli spaces must have the same number of cusps with the same cusp widths, and have the same number of elliptic points of orders 2 and 3.
\end{prop}

\begin{proof} Their coarse moduli spaces must have the same number of cusps because $\sigma$ induces an isomorphism of the nonsingular curves $Y([\varphi])_{\Qbar}$ and $Y(^\sigma[\varphi])_{\Qbar}$.

\sgap

Let $p : \ol{Y([\varphi])_{\Qbar}}\rightarrow\ol{M(1)_{\Qbar}}$ be the corresponding map of compactifications of coarse moduli schemes. Then since the cusp $i\infty$ of $\ol{M(1)_{\Qbar}}$ is $\QQ$-rational, the fiber $p^{-1}(i\infty)$ is an effective divisor on $\ol{Y([\varphi])_{\Qbar}}$ fixed by $G_\QQ$. Let
$$p^{-1}(i\infty) = \sum n_i(P_i)\qquad n_i\ge 0,\; P_i\in\ol{Y([\varphi])_{\Qbar}},\; p(P_i) = i\infty$$
then the cusp widths are precisely the multiplicities $n_i$. Since the divisor is fixed by $G_\QQ$, it follows that the action of $G_\QQ$ on $p^{-1}(i\infty)$ preserves cusp widths. By considering the fiber above $j = 0,1728\in M(1)_{\Qbar}$, the same argument  shows that $Y([\varphi])_{\Qbar},Y(^\sigma[\varphi])_{\Qbar}$ have the same number of elliptic points of orders 2 and 3.
\end{proof}

\section{Examples and remarks}\label{section_examples}
The geometry of $\mM(G)$ is determined by the structure of finite groups, and hence is readily accessible to computation. In this section we give some examples and discuss some phenomena. For more examples see Appendix \ref{appendix_tables}.

\sgap

By Proposition \ref{prop_abelian_is_congruence}, abelian $G$-structures are equivalent to classical congruence level structures. In particular, the associated modular curves are congruence modular curves. Thus, in this section we will focus our attention on nonabelian groups $G$.

\subsection{Some small nonabelian groups $G$}\label{section_first_examples}

We begin by describing $\mM(G)_{\Qbar}$ for the first four nonabelian groups $G$, together with the abelian group $(\ZZ/3\ZZ)^2$ for comparison, in Table \ref{table_first_four_nonabelian}.

\begin{table}[h]\footnotesize
$$\begin{array}{llllrlllllllll}
\text{Label} & \text{Size} & G & m & d & c_4 & c_6 & c_{-1} & \text{Cusp Widths} & \text{Genus} & \text{c/nc} & \text{c/f} & e & g\\
\hline
\Gamma(3) & 9 & (\ZZ/3\ZZ)^2 & 2 & 24 & 0 & 0 & 0 & 3^4 & 0 & \text{cong} & \text{fine} & 1 & 1 \\

\Gamma(D_6) & 6 & D_6 = S_3 & 1 & 3 & 1 & 0 & 1 & 1^12^1 & 0 & \text{cong} & \text{crse} & 3 & 3 \\
\Gamma(D_8) & 8 & D_8 & 1 & 6 & 0 & 0 & 1 & 2^3 & 0 & \text{cong} & \text{crse} & 2 & 3 \\
\Gamma(Q_8) & 8 & Q_8 & 1 & 6 & 0 & 0 & 1 & 2^3 & 0 & \text{cong} & \text{crse} & 2 & 3 \\
\Gamma(D_{10}) & 10 & D_{10} & 2 & 3 & 1 & 0 & 1 & 1^12^1 & 0 & \text{cong} & \text{crse} & 5 & 5 \\
\end{array}$$
\caption{Components of $\mM(G)_{\Qbar}$ for $(\ZZ/3\ZZ)^2$ and the first four nonabelian groups.}\label{table_first_four_nonabelian}
\end{table}


Here, each line describes the stabilizer of a representative of an $\SL_2(\ZZ)$-orbit on $\Hom^\surext(F_2,G)$, which in turn corresponds to a connected component of $\mM(G)_{\Qbar}$.

\sgap

Going left to right, the label is what we will call this stabilizer, which is  well-defined up to conjugacy. Thus, the second line for example describes the modular curve $\hH/\Gamma(D_6) \cong M(D_6)_\CC$. The field ``$m$'' is for ``multiplicity'', and refers to the number of isomorphic components of that type in $\mM(G)_{\Qbar}$ - that is to say, the components are isomorphic as covers of $\mM(1)_{\Qbar}$, or equivalently, that the corresponding subgroups of $\SL_2(\ZZ)$ are conjugate. 

\sgap

Next, $d := [\SL_2(\ZZ) : \Gamma]$, and $c_4$ (resp. $c_6$) refer to the number of conjugacy classes of elements of order 4 (resp. 6) in the stabilizer $\Gamma$, or equivalently the number of elliptic points of the corresponding modular curve (c.f. Prop \ref{prop_galois_action_preserves_signature}). The field ``$c_{-1}$'' is 1 if $-1\in\Gamma$, and is otherwise 0. The format of cusp widths is $(\text{width})^\text{\# cusps of that width}$. Genus refers to the genus of $\hH/\Gamma$, ``\text{c/nc}'' refers to if $\Gamma$ is congruence or noncongruence, and ``c/f'' asks if the moduli scheme is coarse or fine - ie, whether the stabilizer $\Gamma$ is torsion-free.

\sgap

For any particular line in the table, the data $(d,c_4,c_6,c_{-1},\text{Cusp Widths})$ is called the \emph{signature} of the corresponding subgroup of $\SL_2(\ZZ)$ or corresponding component of $\mM(G)_{\Qbar}$. By Proposition \ref{prop_galois_action_preserves_signature}, we find that the signature is an invariant of the $G_\QQ$ action on the components of $\mM(G)_{\Qbar}$.

\sgap

To explain the last two columns, we will need a result of Higman:

\begin{lemma}[Higman]\label{lemma_higman} Let $x,y\in F_2$ be generators, and let $\Aut^+(F_2) < \Aut(F_2)$ be the subgroup of automorphisms of determinant 1. Then for any $\gamma\in\Aut^+(F_2)$, the commutator $[x,y]$ is conjugate to $[\gamma(x),\gamma(y)]$.
\end{lemma}
\begin{proof} The subgroup $\Aut^+(F_2)\le\Aut(F_2)$ of determinant 1 is generated by $\gamma_S,\gamma_T$, where $\gamma_S$ sends $(x,y)\mapsto (y,x^{-1})$, and $\gamma_T$ sends $(x,y)$ to $(x,xy)$. One may then explicitly check that $\gamma_S,\gamma_T$ fix the conjugacy class of $[x,y]$. This also appears in \cite{Neu56} in a slightly different form.
\end{proof}

\begin{defn}[Nielsen Invariant] \label{def_nielsen_invariant} For a surjection $\varphi : F_2\twoheadrightarrow G$, the conjugacy class of $\varphi([x,y])$ is called the \emph{Nielsen invariant} of the $\SL_2(\ZZ)$-orbit of $[\varphi]$.
\end{defn}

A surjection $\varphi : F_2\twoheadrightarrow G$ corresponds to a $G$-galois cover of an elliptic curve $E$ ramified only above  $O$. The inertia subgroup at $O$ is generated by $\varphi([x,y])$, and its order is the ramification index of any point lying above $O$, which we record as ``$e$'' in the Table. Similarly, $g$ is the genus of the cover, and is determined by $e$ and $|G|$. By Higman's lemma, $e$, $g$, and the Nielsen invariant are indeed invariants of the $\SL_2(\ZZ)$ orbit of $[\varphi]$. As we will see later in Lemma \ref{lemma_BCL}, $e,g$ are also invariants of the $G_\QQ$-action on the components of $\mM(G)_{\Qbar}$.

\sgap

\textbf{Warning.} There is no reason for ``$e$'' to remain constant amongst isomorphic components of $\mM(G)$. However, it turns out that in almost all computed examples this is true. The smallest group $G$ such that $\mM(G)$ contains multiple isomorphic components with different $e$'s is the group of order 216 and  index 87 in GAP's ``Small Groups Library''. The components of $\mM(G)$ can be seen in Appendix \ref{section_NC_1-255}, where the ``$e$'' field shows ``$2^16^2$'' and should be read ``one component has $e = 2$, and two components have $e = 6$''. The next smallest such group has order 384 and is not shown in the tables.

\sgap

The components above are all only coarse moduli spaces, since in each case $c_{-1} = 1$, which means that the $G$-structures are fixed by the automorphism $[-1]$. Furthermore, it turns out that $\Gamma(D_8)$ and $\Gamma(Q_8)$ are equal to the principal congruence subgroup $\Gamma(2)$. Similarly, $\Gamma(D_6)$ and $\Gamma(D_{10})$ are conjugate to $\Gamma_1(2) = \Gamma_0(2)$. Next, we see that $\mM(D_6)_{\Qbar},\mM(D_8)_{\Qbar},\mM(Q_8)_{\Qbar}$ are all connected, and hence their (single) components are defined over $\QQ$. By the Weil pairing we know that $\mM(3)_{\Qbar}$ has two components, each defined over $\QQ(\zeta_3)$. The next section will discuss the components of $\mM(D_{2k})_{\Qbar}$. Lastly, we note that in all these examples above, despite the groups $G$ being nonabelian, the stabilizers are congruence. One might attribute this to the fact that these groups are not ``nonabelian enough''. In particular, they are metabelian. Indeed, in all known examples, if $G$ is metabelian, then every component of $\mM(G)$ is congruence.

\subsection{Dihedral groups --- The structure of $\mM(D_{2k})$}\label{ss_dihedral_structure} In this section we determine the Galois-module structure of $\Hom^\surext(\pi_1(E^\circ_{\Qbar}),D_{2k})$ and the structure of the stacks $\mM(D_{2k})_\QQ$.

\sgap


The dihedral group $D_{2k}$ of order $2k$ can be presented as the semidirect product $\ZZ/k\ZZ\rtimes\mu_2$ with $\mu_2 = \{\pm1\}$ acting on $\ZZ/k\ZZ$ by inversion. Thus, we may write the elements of $D_{2k}$ in the form $(n,\pm1)$ (or just $(n,\pm)$) with $n\in\ZZ/k\ZZ$. Multiplication is given by
$$(n,\pm)\cdot (m,-) = (n \pm m,\mp)\qquad\text{and}\qquad (n,\pm)\cdot(m,+) = (n\pm m,\pm)$$
and inversion is given by
$$(n,\pm)^{-1} = (\mp n,\pm)$$
Explicitly, conjugation is given by the formulas:
\begin{eqnarray*}
(r,+)(n,\pm)(r,+)^{-1} & = & (n+r\mp r,\;\pm) \\
(r,-)(n,\pm)(r,-)^{-1} & = & (-n+r\mp r,\;\pm)
\end{eqnarray*}
We have the abelianization maps
$$\begin{array}{lll}
\ab : D_{2k}\twoheadrightarrow D_{2k}^\ab = \mu_2 & (n,\pm)\mapsto \; \pm 1 & \text{if $k$ is odd} \\
\ab : D_{2k}\twoheadrightarrow D_{2k}^\ab = \mu_2\times\ZZ/2\ZZ & (n,\pm)\mapsto (\pm1,\; n\text{ mod } 2) & \text{if $k$ is even}
\end{array}$$

We will call the composition $p_i\circ[\varphi]$ the abelianization of $[\varphi]$, denoted $\ab([\varphi])$. Note that by Proposition \ref{prop_descent_of_groups}, the abelianization maps induce finite \'{e}tale surjections:
$$\begin{array}{ll}
\ab : \mM(D_{2k})\longrightarrow\mM_1(2) & \text{if $k$ is odd} \\
\ab : \mM(D_{2k})\longrightarrow\mM(2) & \text{if $k$ is even}
\end{array}$$
of stacks over $\ZZ[1/2k]$. In particular, a $D_{2k}$-structure on an elliptic curve $E/S$ induces a $D_{2k}^\ab$-structure on $E/S$, but the reverse may not be true.

\sgap

We will identify $\Hom^\surext(F_2,D_{2k})$ with the set of conjugacy classes of ordered pairs of generators for $D_{2k}$, and similarly for their abelianizations. As usual we will use $[(n,\cdot),(m,\cdot)]$ to denote the conjugacy class of a generating pair $((n,\cdot),(m,\cdot))$. If $k$ is odd, then the elements of $\Hom^\surext(F_2,D_{2k})$ have three possible abelianizations: $(+,-),(-,+)$, and $(-,-)$. Any pair $((n,\cdot),(m,\cdot))$ with such an abelianization will generate $D_{2k}$ if and only if the subgroup it generates contains the cyclic normal subgroup $\ZZ/k\ZZ\lhd D_{2k}$, thus, for example, (generating) pairs of the form $((n,+),(m,-))$ must have $n\in(\ZZ/k\ZZ)^\times\subset\ZZ/k\ZZ$. This explains the column labelled ``Restrictions'' in the table below.

\sgap

One checks using the formulas above that for a pair of the form $((n,+),(m,-))$, the value $\pm n\in(\ZZ/k\ZZ)^\times/\pm 1$ is an invariant of the conjugacy class $[(n,+),(m,-)]$.

\begin{defn} Let $\zeta_k := e^{2\pi i/k}$, and $\mu_k^\prim$ the set of primitive $k$th roots of unity.
$$\Inv([(n,+),(m,-)]) := \zeta_k^{\pm n}\in(\mu_k^\prim)^{\pm1}$$
which we will simply call the \emph{invariant} of $[\varphi] = [(n,+),(m,-)]$. Similarly, we define
$$\Inv([(n,-),(m,+)]) := \zeta_k^{\pm m}\qquad \Inv([(n,-),(m,-)]) := \zeta_k^{\pm(n-m)}$$
\end{defn}

\sgap

We summarize the general situation for $k$ odd in the table below.
$$\begin{array}{llrl}
\text{Abelianization} & \text{Representative} & \text{Restrictions}\quad & \text{$\Inv([(n,\cdot),(m,\cdot)])$} \\
(+,-) & ((n,+),(m,-)) & n\in(\ZZ/k\ZZ)^\times & \zeta_k^{\pm n} \\


(-,+) & ((n,-),(m,+)) & m\in(\ZZ/k\ZZ)^\times & \zeta_k^{\pm m} \\


(-,-) & ((n,-),(m,-)) & n-m\in(\ZZ/k\ZZ)^\times & \zeta_k^{\pm (n-m)}


\end{array}$$

In the above, we view $(\ZZ/k\ZZ)^\times$ as a subset of $\ZZ/k\ZZ\lhd D_{2k}$. If $k$ is even, then we have six possible abelianizations, which are completely described by the sign of the generators $(n,\cdot),(m,\cdot)$ as well as the class of a ``$-$'' generator modulo 2.

$$\begin{array}{llrrcl}
\text{Abelianization} & \text{Representative} & \text{Restrictions} & & \;\; & \text{$\Inv([(n,\cdot),(m,\cdot)])$} \\
(+,-,0) & ((n,+),(m,-)) & n\in(\ZZ/k\ZZ)^\times, & m\equiv 0\mod 2 & &\zeta_k^{\pm n} \\
(+,-,1) & ((n,+),(m,-)) & n\in(\ZZ/k\ZZ)^\times, & m\equiv 1\mod 2 & &\zeta_k^{\pm n} \\
(-,+,0) & ((n,-),(m,+)) & m\in(\ZZ/k\ZZ)^\times, & n\equiv 0\mod 2 & &\zeta_k^{\pm m} \\
(-,+,1) & ((n,-),(m,+)) & m\in(\ZZ/k\ZZ)^\times, & n\equiv 1\mod 2 & &\zeta_k^{\pm m} \\
(-,-,0) & ((n,-),(m,-)) & n-m\in(\ZZ/k\ZZ)^\times, & m\equiv 0\mod 2 & & \zeta_k^{\pm (n-m)} \\
(-,-,1) & ((n,-),(m,-)) & n-m\in(\ZZ/k\ZZ)^\times, & m\equiv 1\mod 2 & & \zeta_k^{\pm (n-m)}
\end{array}$$
These lists classify all elements of $\Hom^\surext(F_2,D_{2k})$. In particular, there is a bijection
$$\begin{array}{rcl}
\Hom^\surext(F_2,D_{2k}) & \rightiso & \Hom^\surext(F_2,D_{2k}^\ab)\times\left(\mu_k^\prim\right)^{\pm1} \\
{}[\varphi] & \mapsto & (\ab([\varphi]),\Inv([\varphi]))
\end{array}$$
Note that $\Inv([\varphi])^2$ is precisely ($\zeta_k$ to the power of) the Nielsen invariant of $[\varphi]$. 

\sgap

Let $E = \spmatrix{0}{1}{-1}{0}$, and $T = \spmatrix{1}{1}{0}{1}$, then $E,T$ generate $\SL_2(\ZZ)$. To compute the action of $\SL_2(\ZZ)$ on $\Hom^\surext(F_2,D_{2k})$, we will consider the lifts of $E,T$ to $\Aut(F_2)$ given by
$$\gamma_E : \left\{\begin{array}{rcl}
x & \mapsto & y^{-1} \\
y & \mapsto & x
\end{array}\right. \qquad\text{and}\qquad \gamma_T : \left\{\begin{array}{rcl}
x & \mapsto & x \\
y & \mapsto & xy
\end{array}\right.$$
The action of $\gamma_E,\gamma_T$ on $[\varphi]$ is as follows, where $\alpha^{\pm 1} = \Inv([\varphi])\in(\mu_k^\prim)^{\pm1}$.

$$\gamma_E : \left\{\begin{array}{rcl}
((+,-,0), \alpha^{\pm1}) & \mapsto & ((-,+,0),\alpha^{\pm1}) \\
((+,-,1), \alpha^{\pm1}) & \mapsto & ((-,+,1),\alpha^{\pm1}) \\
((-,+,0), \alpha^{\pm1}) & \mapsto & ((+,-,0),\alpha^{\pm1}) \\
((-,+,1), \alpha^{\pm1}) & \mapsto & ((+,-,1),\alpha^{\pm1}) \\
((-,-,0),\alpha^{\pm1}) & \mapsto & ((-,-,0),\alpha^{\pm1}) \\
((-,-,1),\alpha^{\pm1}) & \mapsto & ((-,-,1),\alpha^{\pm1})
\end{array}\right.\qquad
\gamma_T : \left\{\begin{array}{rcl}
((+,-,0), \alpha^{\pm1}) & \mapsto & ((+,-,1),\alpha^{\pm1}) \\
((+,-,1), \alpha^{\pm1}) & \mapsto & ((+,-,0),\alpha^{\pm1}) \\
((-,+,0), \alpha^{\pm1}) & \mapsto & ((-,-,1),\alpha^{\pm1}) \\
((-,+,1), \alpha^{\pm1}) & \mapsto & ((-,-,0),\alpha^{\pm1}) \\
((-,-,0),\alpha^{\pm1}) & \mapsto & ((-,+,1),\alpha^{\pm1}) \\
((-,-,1),\alpha^{\pm1}) & \mapsto & ((-,+,0),\alpha^{\pm1})
\end{array}\right.$$
Here if $k$ is odd then we simply ignore the 0 or 1 in the abelianization. From this, we see that $\SL_2(\ZZ)$ acts transitively on the set of $[\varphi]$ with a fixed invariant. This moreover shows that $\Inv([\varphi])$ is also invariant under $\SL_2(\ZZ)$. Thus, there are $\frac{\phi(k)}{2}$ distinct $\SL_2(\ZZ)$-orbits corresponding to the $\phi(k)$ possible choices for $\alpha\in\mu_k^\prim$, each of size 3 or 6 according to whether $k$ is odd or even. This shows that if $k$ is odd, then $\Gamma_{[\varphi]}$ is conjugate to $\Gamma_1(2) = \Gamma_0(2)$. If $k$ is even, then $\Gamma_{[\varphi]} = \Gamma(2)$. This proves part of the following proposition.

\sgap

\begin{thm}\label{thm_dihedral_geometric_structure}[Structure of $\mM(D_{2k})$] For any integer $k\ge 3$, there is a $G_\QQ$-equivariant bijection between the components of $\mM(D_{2k})_{\Qbar}$ and $\left(\mu_k^\prim\right)^{\pm1}$. In particular there are $\frac{\phi(k)}{2}$ distinct components, each defined over $\QQ(\zeta_k+\zeta_k^{-1})$.
\begin{itemize}
\item[1.] If $k$ is odd, then for any surjection $\varphi : F_2\twoheadrightarrow D_{2k}$, $\Gamma_{[\varphi]}$ is conjugate to $\Gamma_1(2)$, and thus every component of $\mM(D_{2k})_{\Qbar}$ is isomorphic to $\yY_1(2)_{\Qbar}=\mM_1(2)_{\Qbar}$.
\item[2.] If $k$ is even, then $\Gamma_{[\varphi]} = \Gamma(2)$, and thus every component of $\mM(D_{2k})_{\Qbar}$ is isomorphic to $\yY(2)_{\Qbar} = \mM(2)_{\Qbar}$.
\end{itemize}
\end{thm}

We will need to understand the $G_\QQ$ action on the inertia subgroup $\langle [x,y]\rangle\le F_2$. This was first articulated as the ``branch cycle lemma'' of Fried (c.f. \cite{Vol96}).

\begin{lemma}[``Branch Cycle Lemma'']\label{lemma_BCL} Let $E/K$ be an elliptic curve over a number field $K$. From the homotopy exact sequence, we obtain a representation (c.f. \ref{prop_Teichmuller_moduli_2}(1)):
$$G_\QQ\longrightarrow\Aut(\pi_1(E^\circ_{\Qbar})) \cong \Aut(\widehat{F_2})$$
There exist topological generators $x,y\in\pi_1(E^\circ_{\Qbar})$, such that for any $\sigma\in G_K$, we have $\sigma([x,y]) = [x,y]^{\chi(\sigma)}$, where $[x,y] := xyx^{-1}y^{-1}$ and $\chi : G_K\longrightarrow\widehat{\ZZ}^\times$ is the cyclotomic character.
\end{lemma}
\begin{proof} This statement is also a special case of the easy direction of \cite{Nak94}, Theorem 2.1.1.
\end{proof}

Note that this implies that the data ``$e$'' associated to any component of $\mM(G)_{\Qbar}$ is also invariant under the $G_\QQ$ action on connected components. 

\sgap

Proposition \ref{thm_dihedral_geometric_structure} will follow from the following result, setting $K = \QQ$.

\begin{thm}\label{thm_dihedral_galois_structure} Let $K$ be a number field and $E/K$ an elliptic curve, and $x,y$ topological generators for $\pi_1(E^\circ_{\Qbar})$ as in the Branch Cycle Lemma. Then the bijection
$$\begin{array}{rcl}
\Hom^\surext(\pi_1(E^\circ_{\Qbar}),D_{2k}) & \rightiso & \Hom^\surext(\pi_1(E^\circ_{\Qbar}),D_{2k}^\ab)\times\left(\mu_k^\prim\right)^{\pm1} \\
{}[\varphi] & \mapsto & (\ab([\varphi]),\Inv([\varphi]))
\end{array}$$
is an isomorphism if $G_K$-modules, where the $G_K$ action on the first direct factor is determined by the classical mod 2 representation $G_K\rightarrow\GL_2(\ZZ/2\ZZ)$.
\end{thm}
\begin{proof} By \ref{prop_descent_of_groups}, the first coordinate of the bijection above is $G_K$-equivariant. It remains to show that the $G_K$ action on $\Inv([\varphi])$ is via the cyclotomic character $\chi$. In the discussion above we found that for $[\varphi]\in\Hom^\surext(\pi_1(E^\circ_{\Qbar}),D_{2k})$, we have $\Inv([\varphi])^2 = \zeta_k^{\pm\varphi([x,y])}$, where $\pm\varphi([x,y])$ is viewed inside $\ZZ/k\ZZ\le D_{2k}$ and is precisely the Nielsen invariant of $[\varphi]$. 

\sgap

Fix a $[\varphi]\in\Hom^\surext(\pi_1(E^\circ_{\Qbar},D_{2k})$. By definition, $\sigma.[\varphi] = [\varphi\circ\sigma^{-1}]$. For ease of notation, let $I([\varphi]) := \log_{\mu_k}\Inv([\varphi])$, so that $I([\varphi])$ takes values in $\pm(\ZZ/k\ZZ)^\times$ instead of $(\mu_k^\prim)^{\pm1}$. Thus we want to show that $I([\varphi\circ\sigma^{-1}]) = \chi(\sigma)\cdot I([\varphi])$. Using the fact that $2I([\varphi]) = \pm\varphi([x,y])$, this essentially follows from the Branch Cycle Lemma, but there is a difficulty when $k\equiv 0\mod 4$ which we will circumvent as follows:

\sgap

Let $[\tilde{\varphi}]\in\Hom^\surext(\pi_1(E^\circ_{\Qbar}),D_{4k})$ be a lift of $[\varphi]$ via the surjection $f : D_{4k}\rightarrow D_{2k}$ given by $(n\mod 2k,\pm)\mapsto (n\mod k,\pm)$ (This is possible by Gasch\"{u}tz' lemma, c.f. \ref{prop_descent_of_groups}). Then by the Branch Cycle Lemma, we have:
$$2\cdot I([\tilde{\varphi}\circ\sigma^{-1}]) = \pm2\cdot\tilde{\varphi}(\sigma^{-1}([x,y])) = \pm2\cdot\chi(\sigma)\tilde{\varphi}([x,y]) = 2\chi(\sigma)\cdot I([\tilde{\varphi}])$$
which takes place inside $\pm(\ZZ/2k\ZZ)$. This implies that either
$$I([\tilde{\varphi}\circ\sigma^{-1}]) = \chi(\sigma)\cdot I([\tilde{\varphi}])\qquad\text{or}\qquad I([\tilde{\varphi}\circ\sigma^{-1}]) = \chi(\sigma)\cdot I([\tilde{\varphi}]) + k$$
This ambiguity however disappears mod $k$, so passing to $D_{2k}$ via $f$, we find that
$$I([\varphi\circ\sigma^{-1}]) = \chi(\sigma)\cdot I([\varphi])$$
as desired.

\end{proof}

\begin{remark}\label{remark_same_Gamma_different_stack} Let $k\ge 3$ be odd. As noted in the proof, a $D_{2k}$-structure on an elliptic curve $E/K$ induces a $\Gamma_1(2)$-structure on $E$, and furthermore the components of $\mM(D_{2k})_{\Qbar}$ are isomorphic to $\mM_1(2)_{\Qbar}$. Since the components are determined up to $\Qbar$-isomorphism by the mapping class group action of $\SL_2(\ZZ)$ on $\pi_1(E^\circ_{\Qbar})$, this means that the mapping-class-group action cannot distinguish $D_{2k}$-structures from $\Gamma_1(2)$-structures --- in other words, it cannot see the nonabelian side of $D_{2k}$. Moduli-theoretically, this means that every 2-isogeny $E'\rightarrow E$ over $\Qbar$ can be extended (in $\frac{\phi(k)}{2}$ different ways) to a $D_{2k}$-torsor $X\rightarrow E'\rightarrow E$ over $E$. On the other hand, we see from Theorem \ref{thm_dihedral_galois_structure} that the Galois action \emph{does} see the nonabelian part of $D_{2k}$, which in this case is completely described by the Galois action on the inertia subgroup. In other words, while one might say a $D_{2k}$-structure is a ``congruence'' level structure in that the components of $\mM(D_{2k})_{\Qbar}$ are congruence, it is not in general a ``classical congruence'' level structure due to the difference in the Galois action. 

\end{remark}

\subsection{The simple groups $A_5$, $\PSL_2(\FF_7)$ and $\Sz(8)$}\label{ss_three_simple_groups}

\begin{defn}\label{def_purity} Let $G$ be a finite 2-generated group, then we will say that $G$ is \emph{purely noncongruence} (resp. \emph{purely congruence}) if for every surjection
$$\varphi : F_2\twoheadrightarrow G$$
the stabilizer $\Gamma_{[\varphi]}$ is noncongruence (resp. congruence). We say that $G$ is \emph{noncongruence} if $G$ is not purely congruence.
\end{defn}
In all examples of 2-generated groups $G$ for which we have computed the set $\Hom^\surext(F_2,G)$ and the associated $\SL_2(\ZZ)$-stabilizers, $G$ has been either purely congruence or purely noncongruence. At the time of writing, this list includes all 2-generated groups of order $\le 255$ (c.f. \S\ref{section_NC_1-255}), as well as all simple groups of order $\le 29120$ (c.f. \S\ref{section_23_FSGs}). However we do not know if this should always be the case.

\sgap

We saw in the previous section that all $D_{2k}$-structures are congruence, but generally are not ``classical congruence'' due to the difference in the $G_\QQ$-action. Indeed, this is no exception -- of all 2-generated groups of order $\le 255$, computations show that every metabelian group is congruence. However, by the congruence subgroup property for $\Aut(F_2)$ (c.f. \ref{thm_CSP}, \ref{cor_CSP}), we know that for any finite index $\Gamma\le\SL_2(\ZZ)$, there exists a group $G$ and a surjection $\varphi : F_2\twoheadrightarrow G$ such that $\Gamma_{[\varphi]}\lhd \Gamma$. Thus, there definitely exist noncongruence $G$, and by \ref{prop_abelian_is_congruence}, such groups will necessarily be nonabelian. Thus, it is natural to expect that if a group $G$ is ``sufficiently nonabelian'' in some sense, then it should be noncongruence. Indeed, computations show that all solvable groups of order $\le 255$ and solvable length $\ge 4$ are purely noncongruence. In particular, one would expect that nonabelian finite simple groups should be purely noncongruence (note that every finite simple group is 2-generated). We list the data for the smallest two nonabelian simple groups in Table \ref{table_first_two_nonabelian_simple} (see Appendix \ref{section_23_FSGs} for the smallest 23 nonabelian finite simple groups).





\begin{table}[h]\footnotesize
$$\begin{array}{llllrlllllllll}
\text{Label} & \text{Size} & G & m & \text{d} & c_4 & c_6 & c_{-1} & \text{cusp widths} & \text{genus} & \text{c/nc} & \text{c/f} & e & g\\
\hline
\Gamma(A_5)_1 & 60 & A_5 & 1 & 18 & 0 & 0 & 1 & 2^13^25^2 & 0 & \text{ncng} & \text{crse} & 3 & 21 \\
\Gamma(A_5)_2 & 60 & A_5 & 2 & 10 & 0 & 1 & 1 & 2^13^15^1 & 0 & \text{ncng} & \text{crse} & 5 & 25 \\
\\
\Gamma(2,7)_1 & 168 & \PSL_2(\FF_7) & 2 & 7 & 1 & 1 & 1 & 3^14^1 & 0 & \text{ncng} & \text{crse} & 7 & 73 \\
\Gamma(2,7)_2 & 168 & \PSL_2(\FF_7) & 1 & 32 & 0 & 1 & 0 & 2^13^14^17^1& 0 & \text{ncng} & \text{crse} & 4 & 64 \\
\Gamma(2,7)_3 & 168 & \PSL_2(\FF_7) & 1 & 32 & 0 & 1 & 0 & 2^13^14^17^1& 0 & \text{ncng} & \text{crse} & 4 & 64 \\
\Gamma(2,7)_4 & 168 & \PSL_2(\FF_7) & 1 & 36 & 0 & 0 & 0 & 1^13^24^17^1 & 0 & \text{ncng} & \text{fine} & 3 & 57
\end{array}$$
\caption{Components of $\mM(A_5)_{\Qbar}$ and $\mM(\PSL_2(\FF_7))_{\Qbar}$}\label{table_first_two_nonabelian_simple}
\end{table}

\sgap

Here, we find that there are three components of $\mM(A_5)_{\Qbar}$. One has stabilizer $\Gamma(A_5)_1$, corresponding to a degree $18$ cover of the $j$-line $\AA^1_j$, and two have conjugate stabilizers $\Gamma(A_5)_2$, corresponding to two degree 10 covers of $\AA^1_j$. The action of the absolute Galois group $G_\QQ$ on the components of $\mM(A_5)_{\Qbar}$ must preserve the signature of the components (c.f. Proposition \ref{prop_galois_action_preserves_signature}), as well as the field ``$e$'' (c.f. Lemma \ref{lemma_BCL}), and hence also ``$g$'' . Thus, we deduce that the component $\yY(\Gamma(A_5)_1)\subset \mM(A_5)_{\Qbar}$ corresponding to $\Gamma(A_5)_1$ is defined over $\QQ$ and $Y(\Gamma(A_5)_1)$ is a model of $\hH/\Gamma(A_5)_1$. Each component $\yY(\Gamma(A_5)_2)$ is defined over at most a quadratic extension of $\QQ$. In this case, since $e = 5$, we know that the Nielsen invariant of an $\Gamma(A_5)_2$-structure $[\varphi] : F_2\twoheadrightarrow A_5$ must have order 5. The group $A_5$ has two conjugacy classes of 5-cycles, represented by $(12345)$ and $(12345)^2$. The Branch Cycle Lemma (\ref{lemma_BCL}) then tells us that for any $\sigma\in G_\QQ$, $\varphi(\sigma([x,y])) = \varphi([x,y])^{\chi(\sigma)}$. This implies that each component $\yY(\Gamma(A_5)_2)$ is defined over the field $\QQ(\zeta_5+\zeta_5^{-1}) = \QQ(\sqrt{5})$.

\sgap

Similarly, for $\PSL_2(\FF_7)$, there are two conjugacy classes of elements of order 7. Let $T := \spmatrix{1}{1}{0}{1}$, then the two conjugacy classes are $\{T,T^2,T^4\}$ and $\{T^3,T^5,T^6\}$. As above, we find that $\yY(\Gamma(2,7)_1)$ is defined over $\QQ(\sqrt{-7})$, the unique quadratic subfield of $\QQ(\zeta_7)$. Since $m = 1$ for $\Gamma(2,7)_4$, we find that $\yY(\Gamma(2,7)_4)$ is defined over $\QQ$. Unfortunately, since $\PSL_2(\FF_7)$ has only one conjugacy class of order $e = 4$, this method says nothing about the field of definition of $\yY(\Gamma(2,7)_2)$ or $\yY(\Gamma(2,7)_3)$.

\sgap

Note that for $G = \PSL_2(\FF_7)$, the groups $\Gamma(2,7)_2,\Gamma(2,7)_3$ have the same signature (and the same $e,g$), but are not conjugate as subgroups of $\SL_2(\ZZ)$, which explains why they are presented on different lines in the table.

\sgap



As expected, all components of $\mM(A_5)$ and $\mM(\PSL_2(\FF_7))$ are noncongruence, and further we note that the component $Y(\Gamma(2,7)_3)\subset M(\PSL_2(\FF_7))_{\Qbar}$ is a genus 0 fine moduli scheme, with multiplicity $m = 1$ and a unique signature, so it is defined over $\QQ$. Furthermore, it has a unique cusp of width 1, which must be $\QQ$-rational, and so $Y(\Gamma(2,7)_3)_\QQ \cong \PP^1_\QQ - \{\text{cusps}\}$. Thus, there exists a universal family of elliptic curves $\EE(2,7)_3$ over $Y(\Gamma(2,7)_3)_\QQ$ and a universal $\PSL_2(\FF_7)$-Galois cover $\mathbb{X}(2,7)_3$ over $\EE(2,7)_3$ with fibers of genus 57.

\sgap





We conclude this section with data for the largest group $G$ for which we've computed the components of $\mM(G)$ - the Suzuki Group $G = \Sz(8)$, in Table \ref{table_Suzuki_8}.

\sgap

\begin{table}[h]\footnotesize
$$\begin{array}{llllrlllllllll}
\text{Label} & \text{Size} & G & \text{m} & \text{d} & c_4 & c_6 & c_{-1} & \text{cusp widths} & \text{genus} & \text{c/nc} & \text{c/f} & e & g\\
\hline
\Gamma(\Sz(8))_1 & 29120 &  \Sz(8)  & 3 & 84 & 0 & 0 & 0 &  1^{1}4^{3}5^{3}7^{2}  & 0 & \text{ncng} & \text{fine} & 7 & 12481\\
\Gamma(\Sz(8))_2 & 29120 &  \Sz(8)  & 3 & 468 & 0 & 0 & 0 &  1^{3}4^{13}5^{7}7^{15}13^{3}  & 0 & \text{ncng} & \text{fine} & 13 & 13441\\
\Gamma(\Sz(8))_3 & 29120 &  \Sz(8)  & 3 & 588 & 0 & 0 & 0 &  1^{3}4^{13}5^{12}7^{20}13^{3}  & 0 & \text{ncng} & \text{fine} & 7 & 12481\\
\Gamma(\Sz(8))_4 & 29120 &  \Sz(8)  & 1 & 660 & 0 & 0 & 0 &  1^{3}4^{9}5^{21}7^{21}13^{3}  & 0 & \text{ncng} & \text{fine} & 5 & 11649\\
\Gamma(\Sz(8))_5 & 29120 &  \Sz(8)  & 3 & 624 & 0 & 0 & 0 &  2^{4}4^{14}5^{13}7^{15}13^{6}  & 1 & \text{ncng} & \text{fine} & 13 & 13441\\
\Gamma(\Sz(8))_6 & 29120 &  \Sz(8)  & 3 & 624 & 0 & 0 & 0 &  2^{4}4^{14}5^{13}7^{15}13^{6}  & 1 & \text{ncng} & \text{fine} & 13 & 13441\\
\Gamma(\Sz(8))_7 & 29120 &  \Sz(8)  & 3 & 624 & 0 & 0 & 0 &  2^{4}4^{14}5^{13}7^{15}13^{6}  & 1 & \text{ncng} & \text{fine} & 13 & 13441\\
\Gamma(\Sz(8))_8 & 29120 &  \Sz(8)  & 3 & 234 & 4 & 0 & 1 &  2^{3}5^{9}7^{15}13^{6}  & 3 & \text{ncng} & \text{crse} & 13 & 13441\\
\Gamma(\Sz(8))_9 & 29120 &  \Sz(8)  & 3 & 1008 & 0 & 0 & 0 &  2^{4}4^{16}5^{21}7^{30}13^{9}  & 3 & \text{ncng} & \text{fine} & 7 & 12481\\
\Gamma(\Sz(8))_{10} & 29120 &  \Sz(8)  & 3 & 1008 & 0 & 0 & 0 &  2^{4}4^{16}5^{21}7^{30}13^{9}  & 3 & \text{ncng} & \text{fine} & 7 & 12481\\
\Gamma(\Sz(8))_{11} & 29120 &  \Sz(8)  & 3 & 1008 & 0 & 0 & 0 &  2^{4}4^{16}5^{21}7^{30}13^{9}  & 3 & \text{ncng} & \text{fine} & 7 & 12481\\
\Gamma(\Sz(8))_{12} & 29120 &  \Sz(8)  & 1 & 192 & 0 & 0 & 1 &  5^{6}7^{12}13^{6}  & 5 & \text{ncng} & \text{crse} & 2 & 7281\\
\Gamma(\Sz(8))_{13} & 29120 &  \Sz(8)  & 1 & 192 & 0 & 0 & 1 &  5^{6}7^{12}13^{6}  & 5 & \text{ncng} & \text{crse} & 2 & 7281\\
\Gamma(\Sz(8))_{14} & 29120 &  \Sz(8)  & 3 & 1200 & 0 & 0 & 0 &  2^{4}4^{10}5^{24}7^{45}13^{9}  & 5 & \text{ncng} & \text{fine} & 5 & 11649\\
\Gamma(\Sz(8))_{15} & 29120 &  \Sz(8)  & 1 & 1536 & 0 & 0 & 0 &  4^{24}5^{36}7^{48}13^{12}  & 5 & \text{ncng} & \text{fine} & 4 & 10921\\
\Gamma(\Sz(8))_{16} & 29120 &  \Sz(8)  & 1 & 1536 & 0 & 0 & 0 &  4^{24}5^{36}7^{48}13^{12}  & 5 & \text{ncng} & \text{fine} & 4 & 10921\\
\Gamma(\Sz(8))_{17} & 29120 &  \Sz(8)  & 3 & 462 & 8 & 0 & 1 &  2^{3}5^{13}7^{28}13^{15}  & 8 & \text{ncng} & \text{crse} & 7 & 12481\\
\Gamma(\Sz(8))_{18} & 29120 &  \Sz(8)  & 1 & 690 & 12 & 0 & 1 &  2^{3}5^{12}7^{39}13^{27} & 15 & \text{ncng} & \text{crse} & 5 & 11649\\
\end{array}$$
\caption{Components of $\mM(\Sz(8))_{\Qbar}$}\label{table_Suzuki_8}
\end{table}

The Suzuki group $\Sz(8)$ is the second smallest nonabelian simple group for which the Inverse Galois Problem for $\QQ$ is not known (according to a short note on David Zywina's website). As expected, all components are noncongruence. In particular, the component $Y(\Gamma(\Sz(8))_4)$ has multiplicity one and unique signature, so it is defined over $\QQ$, and is a fine moduli scheme of genus 0. Further, it has precisely three cusps of width 1, which must correspond to a point on $Y(\Gamma(\Sz(8))_4)$ defined at most over a cubic extension of $\QQ$. Riemann-Roch implies that $Y(\Gamma(\Sz(8))_4)_\QQ$ is isomorphic to $\PP^1_\QQ - \{\text{cusps}\}$. Thus, we have an universal elliptic curve $\EE$ over $Y(\Gamma(\Sz(8))_4)_\QQ$ with nonconstant $j$-invariant, equipped with a universal $\Sz(8)$-torsor $\mathbb{X}^\circ/\EE^\circ$. It seems to be generally believed that any such family of elliptic curves should admit infinitely many specializations to elliptic curves over $\QQ$ with positive rank. If this were true, or if we could find such a specialization, then I claim this would solve the Inverse Galois Problem for $\Sz(8)$. Indeed, let $y\in Y(\Gamma(\Sz(8))_4)(\QQ)$ be such that $\EE_y^\circ$ has positive rank. Then the universal $\Sz(8)$-torsor specializes to a connected $\Sz(8)$-torsor $f : \mathbb{X}^\circ_y\rightarrow\EE_y^\circ$. For every $\QQ$-point $P\in\EE_y^\circ(\QQ)$, if $f^{-1}(P)$ is disconnected, then it is a disjoint union of copies of $\Spec K$, where $K$ is a field Galois over $\QQ$ with Galois group isomorphic to some subgroup $H\le\Sz(8)$, and each $\Spec K$ maps to a $\QQ$-rational point of the intermediate cover $\mathbb{X}^\circ_y/H$, whose image in $\EE_y^\circ$ is $P$. Let $u : \mathbb{X}_y/H\rightarrow\EE_y$ be the map of proper curves corresponding to $\mathbb{X}^\circ_y/H\rightarrow \EE_y^\circ$. 

\sgap

\textbf{Claim:} The map $u : \mathbb{X}_y/H\rightarrow\EE_y$ is ramified.

\sgap

We will need the following lemmas:

\begin{lemma}\label{lemma_commutator_subgroup_F2} Let $F_2$ be the free group on generators $x,y$. Then its commutator subgroup $[F_2,F_2]$ is generated by the set $A := \{z[x,y]z^{-1} : z\in F_2\}$.
\end{lemma}
\begin{proof} Certainly $\langle A\rangle\subseteq[F_2,F_2]$. On the other hand, note that $\langle A\rangle$ is the smallest normal subgroup containing $[x,y]$, and it's clear that any quotient of $F_2$ which sends $[x,y]$ to the identity must be abelian, and thus $\langle A\rangle\supseteq [F_2,F_2]$.
\end{proof}
\begin{lemma}\label{lemma_perfect_generated_by_conjugates} If $G$ is a perfect group generated by $g,h$. Then it is also generated by the conjugates of $[g,h]$.
\end{lemma}
\begin{proof} Consider the surjection $\phi : F_2\twoheadrightarrow G$ given by sending the generators $x,y$ of $F_2$ to $g,h$. By Lemma \ref{lemma_commutator_subgroup_F2}, the commutator subgroup $[F_2,F_2]$ is generated by conjugates of $[x,y]$. On the other hand, since $G$ is perfect, we have $\phi([F_2,F_2]) = G$, so $G$ is generated by conjugates of $\phi([x,y]) = [g,h]$.
\end{proof}

\begin{proof}(of claim) Note that $\mathbb{X}^\circ_y\rightarrow\EE^\circ_y$ determines a monodromy representation
$$\varphi : \pi_1(\EE^\circ_{y,\Qbar})\cong\Fhat\twoheadrightarrow\Sz(8)$$
Hence, the fiber $u^{-1}(P)$ of $\mathbb{X}^\circ_y/H$ over $P$ corresponds to the coset space $\Sz(8)/H$, whose monodromy action is given via left multiplication via $\varphi$. Let $x,y$ be generators of $\Fhat$, then $\Sz(8)$ is generated by the conjugates of $\varphi([x,y])$, which implies that $H$ cannot contain every conjugate of $\varphi([x,y])$. Thus, there is a conjugate $z[x,y]z^{-1}$ such that $\varphi(z[x,y]z^{-1})\notin H$, and hence $\varphi(z[x,y]z^{-1})$ acts nontrivially on $\Sz(8)/H$. The claim follows by noting that the order of $\varphi(z[x,y]z^{-1})$ acting as a permutation of $\Sz(8)/H$ is $ > 1$, and is precisely the ramification index of a point in the fiber $u^{-1}(O)$.
\end{proof}

Finally, by the Riemann-Hurwitz formula, any ramified cover of an elliptic curve must have genus $\ge 2$. By Falting's theorem, such curves can have only finitely many $\QQ$-rational points. Thus, since there are only finitely many intermediate covers $\mathbb{X}_y^\circ/H$, together the images of their $\QQ$-points account for only finitely many $\QQ$-points of $\EE_y$, and hence there must exist infinitely many $\QQ$-points on $\EE_y$ whose fiber in $\mathbb{X}^\circ_y$ is connected. This would solve the Inverse Galois Problem for $\Sz(8)$.

\sgap

While it is known that there exist at most finitely many congruence subgroups of any genus, there exist in fact infinitely many noncongruence subgroups of any genus (c.f. \cite{Jon79}). In a way, the existence of genus 0 components of $\mM(\Sz(8))_{\Qbar}$ can be seen as manifestation of this fact. By the Riemann-Hurwitz formula, the genera of components of the stacks $\mM(G)_{\Qbar}$ is determined by the size of $G$, the number of components of $\mM(G)_{\Qbar}$, and their numbers of cusps. The problem of counting the components of $\mM(G)_{\Qbar}$ is essentially the problem of \emph{Nielsen equivalence} in group theory. Very little is known about this in our case (c.f. \cite{Pak00}) beyond the existence of the Nielsen invariant. 

\subsection{Noncongruence criteria for finite groups $G$} From the discussion above, it's natural to conjecture that every finite nonabelian simple group should be \emph{purely noncongruence} (c.f. \ref{def_purity}). First we make an observation. Given an exact sequence of groups
$$1\rightarrow N\rightarrow G'\stackrel{f}{\rightarrow} G\rightarrow 1,$$
we say that $G'$ is an extension of $G$ by $N$, and $G$ is a quotient of $G'$. Then, by Prop \ref{prop_descent_of_groups}, we have a surjection
\begin{eqnarray*}
\Hom^\surext(F_2,G') & \stackrel{f_*}{\longrightarrow} & \Hom^\surext(F_2,G) \\
{}[\varphi'] & \mapsto & [f\circ\varphi']
\end{eqnarray*}
and hence $\Gamma_{[\varphi']}\subset\Gamma_{[f\circ\varphi']}$. This implies that the property of being (purely) noncongruence is stable under taking extensions, whereas the property of being (purely) congruence is stable under taking quotients (note that this gives another proof that every 2-generated abelian group $G$ is purely congruence). It seems reasonable to somewhat generalize the conjecture that finite nonabelian simple groups are purely noncongruence to the following:

\begin{conj}\label{conj_nonsolvable_implies_noncongruence} Every nonsolvable finite 2-generated group $G$ is purely noncongruence.
\end{conj}

In this section we will make some partial progress towards the conjecture. For a congruence subgroup $\Gamma\le\SL_2(\ZZ)$, recall that the \emph{congruence level} of $\Gamma$ is the least integer $n$ such that $\Gamma\supset\Gamma(n)$. We will need a generalization of this notion of level.

\sgap

\begin{defn}\label{defn_geometric_level} For any finite index subgroup $\Gamma\subset\SL_2(\ZZ)$, we define its (geometric) level $l_\Gamma$ to be the least common multiple of its cusp widths.
 \end{defn}
 
From now on by default ``level'' will refer to geometric level.
 
\begin{prop}\label{prop_level} Let $\varphi : F_2\twoheadrightarrow G$ be surjective. For any finite index $\Gamma\le\SL_2(\ZZ)$, let $\pm\Gamma$ be the subgroup generated by $\Gamma$ and $-I$. The following quantities are equal
\begin{itemize}
\item[1.] The geometric level $l_{\Gamma_{[\varphi]}}$
\item[2.] The order of $\spmatrix{1}{1}{0}{1}$ as a permutation acting on the coset space $\SL_2(\ZZ)/\pm\Gamma_{[\varphi]}$.
\item[3.] The order of $\spmatrix{1}{1}{0}{1}$ as a permutation acting on $(\SL_2(\ZZ)\cdot[\varphi])/\pm I$.
\end{itemize}
In particular, the cusp widths of $\Gamma_{[\varphi]}$ are the cycle lengths in the disjoint cycle decomposition of the the permutation representation of $\spmatrix{1}{1}{0}{1}$ acting as in (2) or (3).

\sgap

Furthermore, the sum of the cusp widths is the index $[\SL_2(\ZZ) : \pm\Gamma_{[\varphi]}]$, and the geometric level is a conjugacy-invariant. 
\end{prop}

Note that the geometric level only cares about the image of $\Gamma$ in $\PSL_2(\ZZ)$. If $-I\in\Gamma$, then the geometric level is a generalization of the congruence level. If $-I\notin\Gamma$, then the situation is described a result of Kiming-Sch\"{u}tt-Verrill:

\begin{thm}[Wohlfart, Kiming-Sch\"{u}tt-Verrill] Let $\Gamma\subset\SL_2(\ZZ)$ be finite index of geometric level $l$. Then $\pm\Gamma$ is congruence if and only if $\pm\Gamma\supset\Gamma(l)$, and in this case the congruence level of $\pm\Gamma$ is $l$.

\sgap

In general, $\Gamma$ is congruence if and only if it contains $\Gamma(2l)$. In this case, the congruence level of $\Gamma$ is either $l$ or $2l$.
\end{thm}
\begin{proof} The first statement is a result of Wolhfart \cite{Woh63}. The second is Proposition 3 in \cite{KSV11}.
\end{proof}

\begin{remark}\label{remark_pcong_does_not_imply_cong} It is shown in \cite{KSV11} that even if $\pm\Gamma$ is congruence, and hence by Wohlfart contains $\Gamma(l)$, this doesn't imply that $\Gamma$ also contains $\Gamma(l)$ - in fact, $\Gamma$ might even be noncongruence! We will not need the refinement of \cite{KSV11}, but it is good to be aware of this subtlety.
\end{remark}

\sgap

We will adapt an idea of Schmithusen (c.f. \cite{Sch12}). For a subgroup $\Gamma\subset\SL_2(\ZZ)$ of level $l$, let $d := [\SL_2(\ZZ):\pm\Gamma]$, and consider the commutative diagram

\begin{equation}\label{cd_schmithusen}\xymatrix{
1\ar[r] & \Gamma(l)\ar[r] & \SL_2(\ZZ)\ar[r]^{p_l} & \SL_2(\ZZ/l\ZZ)\ar[r] & 1 \\
1\ar[r] & \Gamma(l)\cap\pm\Gamma\ar@{-}[u]_f\ar[r] & \pm\Gamma\ar@{-}[u]_d\ar[r]^{p_l} & p_l(\pm\Gamma)\ar@{-}[u]_e\ar[r] & 1
}\end{equation}

where the rows are exact, the vertical lines are inclusions, and $f,d,e$ are the respective indices. In \cite{Sch12}, $f$ is also called the \emph{congruence deficiency} of $\Gamma$. When $e = 1$ and $d > 1$, $\Gamma$ is said to be \emph{totally noncongruence}. From this we get
\begin{prop} In the diagram above, $d = e\cdot f$, so $\pm\Gamma$ is congruence if and only if $f = 1$ or equivalently $e = d$.
\end{prop}

Thus, to show that $\Gamma$ is noncongruence, it suffices to show that $e < d$.
\gap

In the following we will prove that all $S_n (n\ge 4)$, $A_n (n\ge 5)$, and $\PSL_2(\FF_p) (p\ge 5)$ are noncongruence. In fact the proofs all have essentially the same structure, and so we only include the proof for $S_n$, which serves as a good example to illustrate the basic idea and some of the issues that may arise.

\gap

We will make use of the following three automorphisms of $F_2$ representing the matrices $E = \spmatrix{0}{1}{-1}{0}, T = \spmatrix{1}{1}{0}{1}$, and $-I = \spmatrix{-1}{0}{0}{-1}\in\SL_2(\ZZ)$
$$\gamma_E : \left\{\begin{array}{rcl}
x & \mapsto & y^{-1} \\
y & \mapsto & x
\end{array}\right.\qquad\text{and}\qquad
\gamma_T : \left\{\begin{array}{rcl}
x & \mapsto & x \\
y & \mapsto & xy
\end{array}\right.\qquad\text{and}\qquad
\gamma_{-I} : \left\{\begin{array}{rcl}
x & \mapsto & x^{-1} \\
y & \mapsto & y^{-1}
\end{array}\right.$$

\begin{thm}\label{thm_S_n_is_NC} Let $S_n$ be the symmetric group on $n$ elements, where $n\ge 4$. Let $\varphi : F_2\rightarrow S_n$ be the surjection given by
$$\varphi_1 : \left\{\begin{array}{rcl}
x & \mapsto & (12) \\
y & \mapsto & (123\cdots n)
\end{array}\right.$$
Then $\Gamma_{[\varphi]}$ is noncongruence.
\end{thm}

\begin{proof} 
Consider the two surjections

$$\varphi_1 : \left\{\begin{array}{rcl}
x & \mapsto & (12) \\
y & \mapsto & (123\cdots n) \\
\end{array}\right.\qquad
\varphi_2 = \varphi_1\circ\gamma_T : \left\{\begin{array}{rcl}
x & \mapsto & (12) \\
y & \mapsto & (23\cdots n) 
\end{array}\right.$$

Let $\Gamma_1 := \Gamma_{[\varphi_1]}$, and $\Gamma_2 := \Gamma_{[\varphi_2]}$. Clearly $\Gamma_2 = \gamma_T^{-1}\Gamma_1\gamma_T$, so it suffices to show that one of $\Gamma_1,\Gamma_2$ is noncongruence.

\sgap

Let $p_1,\ldots,p_r$ be the odd primes dividing $n$. Write $l := l(\Gamma_1) = l(\pm\Gamma_1) = 2^k AM$ where $A$ is divisible only by $p_1,\ldots,p_r$, and $M$ is coprime to $2A$.

\sgap

The first step is to prove that the $e := [\SL_2(\ZZ/l\ZZ):p_l(\pm\Gamma_1)]$ as in diagram (\ref{cd_schmithusen}) is small.

\sgap

The decomposition of $l$ gives
$$\SL_2(\ZZ/l\ZZ) = \SL_2(\ZZ/2^k\ZZ)\times\SL_2(\ZZ/A\ZZ)\times\SL_2(\ZZ/M\ZZ)$$
Note that $\Gamma_1$ contains $\spmatrix{1}{2}{0}{1}$ and $\spmatrix{1}{0}{n}{1}$, so since $(M,2n) = 1$, the image $p_l(\Gamma_1)$ contains matrices in $\SL_2(\ZZ/l\ZZ)$ congruent to $\spmatrix{1}{1}{0}{1}$ and $\spmatrix{1}{0}{1}{1}\mod M$. Thus, $p_l(\Gamma_1)\supseteq I\times I\times\SL_2(\ZZ/M\ZZ)$.

\sgap

Similarly, $\Gamma_2$ contains $\spmatrix{1}{2}{0}{1}$ and $\spmatrix{1}{0}{n-1}{1}$. Since $n-1$ is coprime to $n$, $p_l(\Gamma_2)\supset I\times \SL_2(\ZZ/A\ZZ)\times I$.

\sgap

Since $I\times I\times\SL_2(\ZZ/M\ZZ)$ and $I\times\SL_2(\ZZ/A\ZZ)\times I$ are normal subgroups of $\SL_2(\ZZ/l\ZZ)$, and since $\Gamma_1,\Gamma_2$ are conjugate, we see that $p_l(\Gamma_1)$ and $p_l(\Gamma_2)$ must each contain the product $I\times \SL_2(\ZZ/A\ZZ)\times\SL_2(\ZZ/M\ZZ)$.

\sgap

Depending on whether $n$ is even or odd, one of $p_l(\Gamma_1)$ or $p_l(\Gamma_2)$ must contain
$$\ol{\Gamma(2,1)} := \langle \spmatrix{1}{2}{0}{1}\text{ mod } 2^k,\spmatrix{1}{0}{1}{1}\text{ mod } 2^k\rangle\times I\times I\cong \langle\spmatrix{1}{2}{0}{1}\text{ mod } 2^k,\spmatrix{1}{0}{1}{1}\text{ mod } 2^k\rangle\subset\SL_2(\ZZ/2^k\ZZ)$$
We may suppose without loss of generality that $n$ is odd, then from this we see that
$$[\SL_2(\ZZ/2^{k}\ZZ):\ol{\Gamma(2,1)}] \ge [\SL_2(\ZZ/l\ZZ):p_l(\Gamma_1)] \ge [\SL_2(\ZZ/l\ZZ):p_l(\pm\Gamma_1)] =: e$$
Thus, we'd like to show that $[\SL_2(\ZZ/2^{k}\ZZ):\ol{\Gamma(2,1)}]$ isn't too large.

\sgap

Note that $\ol{\Gamma(2,1)}$ is the image of $\Gamma(2,1) := \langle\spmatrix{1}{2}{0}{1},\spmatrix{1}{0}{1}{1}\rangle\subset\SL_2(\ZZ)$ under $p_l$. A finite computation shows $[\SL_2(\ZZ):\Gamma(2,1)] = 3$, and so we find that
$$[\SL_2(\ZZ/2^{k}\ZZ):\ol{\Gamma(2,1)}]\le[\SL_2(\ZZ),\Gamma(2,1)] = 3$$
so we get $e \le [\SL_2(\ZZ/2^{k}\ZZ):\ol{\Gamma(2,1)}] \le 3$. The last step is to show that the index $d = [\SL_2(\ZZ):\pm\Gamma_1]$ of diagram (\ref{cd_schmithusen}) is $\ge 4$.

\sgap

To do this, we may consider the four explicit homomorphisms $\psi_i : F_2\twoheadrightarrow S_n$
$$\psi_1 = \varphi_1 : \left\{\begin{array}{rcl}
x & \mapsto & (12) \\
y & \mapsto & (123\cdots n) \\
\end{array}\right.\qquad
\psi_2 = \varphi_1\circ\gamma_E\circ\gamma_T : \left\{\begin{array}{rcl}
x & \mapsto & (n\cdots 321) \\
y & \mapsto & (n\cdots 32)
\end{array}\right.$$
$$
\psi_3 = \varphi_1\circ\gamma_T\circ\gamma_E : \left\{\begin{array}{rcl}
x & \mapsto & (n\cdots 32) \\
y & \mapsto & (12) 
\end{array}\right.\qquad 
\psi_4 = \varphi_1\circ\gamma_E : \left\{\begin{array}{rcl}
x & \mapsto & (n\cdots 321) \\
y & \mapsto & (12) 
\end{array}\right.$$



Note that for a surjection $\psi : F_2\twoheadrightarrow S_n$, $\pm\Gamma_{[\psi]}$ is the stabilizer of the set $\{[\psi],[\psi\circ\gamma_{-I}]\}$ (ie, the equivalence class of $[\psi]$ ``up to $\pm I$''). For any such class $\{[\psi],[\psi\circ\gamma_{-I}]\}$, the orders of the images of the generators $x,y$ of $F_2$ are invariants of this class. Thus, $\psi_1,\psi_2,\psi_3,\psi_4$ are all inequivalent up to $\pm\Gamma$, but lie in the same $\SL_2(\ZZ)$-orbit, so $d \ge 4$ as desired.

\end{proof}

A similar argument considering each residue class of $n\mod 6$ separately yields:

\begin{thm}\label{thm_A_n_noncongruence} Let $A_n$ be the alternating group on $n$ elements, where $n\ge 5$. Let $\varphi : F_2\rightarrow A_n$ be the surjection given by
$$\begin{array}{rcl}
x & \mapsto & (123) \\
y & \mapsto & (123\cdots n)
\end{array}\qquad\text{if $n$ is odd, or} \qquad\begin{array}{rcl}
x & \mapsto & (123) \\
y & \mapsto & (23\cdots n)
\end{array}\qquad\text{if $n$ is even}$$
Then $\Gamma_{[\varphi]}$ is noncongruence.
\end{thm}

\begin{remark} The proof for $S_n$ was only complicated by the nontriviality of $\gcd(|\varphi_i(x)|,|\varphi_i(y)|)$ for either $i = 1$ or $2$, which required the computation of the index of $\ol{\Gamma(2,1)}$. In particular, we find
\end{remark}

\begin{thm}\label{thm_simple_nc_criteria} Let $G$ be a finite group generated by $a,b$ such that the orders $|a|,|b|,|ab|$ are pairwise coprime, then for surjection $\varphi : F_2\rightarrow G$ given by sending $x,y$ to $a,b$, we have $\Gamma_{[\varphi]}$ is \emph{totally noncongruence} (ie, it is noncongruence, and $e = 1$ in diagram (\ref{cd_schmithusen})).
\end{thm}

\begin{cor}\label{cor_PSL2p_noncongruence} For $p\ge 5$, $\PSL_2(\mathbb{F}_p)$ is noncongruence.
\end{cor}

\begin{proof} Using the standard generators represented by $V := \spmatrix{1}{1}{-1}{0}, E := \spmatrix{0}{1}{-1}{0}$ of order 2,3, their product is by $\spmatrix{1}{1}{0}{1}$ which has order $p$. The result follows from Corollary \ref{thm_simple_nc_criteria}.
\end{proof}

The proof of Theorem \ref{thm_S_n_is_NC} also shows:
\begin{cor}\label{cor_level_divides_N} For $\varphi : F_2\twoheadrightarrow G$, the level $l(\Gamma_{[\varphi]})$ divides the exponent $e(G)$ of $G$.
\end{cor}
\begin{proof} This follows from the fact that if $\varphi(x)$ has order $n$, then $n$ divides $e(G)$, and $\gamma_T^n$ stabilizes $\varphi$. Thus, the the smallest integer $k$ such that $\gamma_T^k$ stabilizes $[\varphi]$, which is the cusp width ``attached to $[\varphi]$'', must also divide $n$.
\end{proof}

\sgap

It is a consequence of Thompson's results on $N$-groups that the hypothesis of Theorem \ref{thm_simple_nc_criteria} holds for all \emph{minimal nonabelian finite simple groups} (c.f. \cite{Tmp68} Cor. 3). The Thompson result implies:

\begin{cor}[Thompson]\label{cor_MFSNG} Let $G$ be a minimal nonabelian finite simple group (ie, one for which all proper subgroups are solvable). These are precisely the groups:
\begin{itemize}
\item[1.] $\PSL_2(\FF_{2^p})$, $p$ any prime.
\item[2.] $\PSL_2(\FF_{3^p})$, $p$ any odd prime.
\item[3.] $\PSL_2(\FF_p)$, $p > 3$ a prime with $p^2+1\equiv 0\mod 5$.
\item[4.] $\Sz(2^p)$, $p$ any odd prime. 
\item[5.] $\PSL_3(\FF_3)$.
\end{itemize}
Then $G$ is noncongruence.
\end{cor}
\begin{proof} This follows from \cite{Tmp68} Cor. 3, which states that any finite group is solvable if and only if it does not contain three elements $a,b,c$ of pairwise coprime order with $abc = 1$.

\end{proof}

Since the property of being noncongruence is stable under group extensions, we get:

\begin{cor}\label{cor_noncongruence_summary} Let $G$ be an extension of $S_n$ ($n\ge 4$), $A_n$ ($n\ge 5$), $\PSL_2(\FF_p)$ ($p\ge 5$), or a minimal nonabelian finite simple group. Then $G$ is noncongruence (c.f. \ref{def_purity}).
\end{cor}

A number of papers on Hurwitz groups also prove the property of Theorem \ref{thm_simple_nc_criteria} for a wide class of other nonabelian simple groups. It is expected that every nonabelian finite simple group should satisfy this property. Note that results of this form do not prove that the groups in question are \emph{purely noncongruence}. However, having checked all surjections from $F_2$ onto all finite 2-generated groups of size $\le 255$ and all simple groups of size $\le 29120$, all examples show that if $G$ is noncongruence, then it is purely noncongruence. Accordingly, we make the following conjecture:

\begin{conj}\label{conj_nc_implies_purely_nc} Let $G$ be a finite 2-generated group. Then $G$ is noncongruence if and only if it is purely noncongruence.
\end{conj}

\section{Applications to the arithmetic of noncongruence modular forms}\label{section_UBD} In this section we use our moduli interpretations to give a description of the ``bad primes'' for the unbounded denominators conjecture (Theorem \ref{thm_bad_primes}), and to interpret the conjecture in geometric terms (\S\ref{ss_geometric_UBD}).
\subsection{The Tate curve}
Here we recall some basic properties of the Tate curve, which will play a central role in our discussion. We like to think of the Tate curve as the base change of the universal elliptic curve over $\hH$ with coordinate $\tau$ to an infinitesimal neighborhood of $i\infty$, with coordinate $q = e^{2\pi i \tau}$. Thus, modular functions on modular curves are then realized on the Tate curve as their $q$-expansions (c.f. Proposition \ref{prop_level_structures_are_qexps}). Our main references for this are \cite{SilAT}, \S V and the appendix of \cite{K72}.



\gap

The exponential map $x\mapsto e^x$ gives an isomorphism of elliptic curves $\CC/\langle 2\pi i,2\pi i \tau\rangle\rightiso\CC^\times/q^\ZZ$. If $z$ is a coordinate on $\CC$, then $t := e^{2\pi iz}$ is the corresponding coordinate on $\CC^\times$, and the above isomorphism sends $2\pi idz$ to $dt/t$. By standard calculations, the curve $\CC/\langle 2\pi i, 2\pi i\tau\rangle$ with differential $2\pi idz$ is given by the plane cubic 
$$Y^2 = 4X^3 - \frac{E_4}{12}X + \frac{E_6}{216},\qquad\text{with differential }dX/Y$$
where $X = \wp(2\pi i z,\langle 2\pi i,2\pi i \tau\rangle), Y = \wp'(2\pi iz,\langle 2\pi i,2\pi i\tau\rangle)$, and 
\begin{eqnarray*}
E_4 & = & 1 + 240\sum_{n\ge 1}\sigma_3(n)q^n \\
E_6 & = & 1 - 504\sum_{n\ge 1}\sigma_5(n)q^n
\end{eqnarray*}
where $\sigma_k(n) := \sum_{d\mid n, d\ge 1} d^k$. The equation above defines an elliptic curve over $\ZZ[\frac{1}{6}]\ps{q}$. By applying the change of variables
$$X = x + \frac{1}{12},\qquad Y = x + 2y$$
we may rewrite the equation as
$$y^2 + xy = x^3 + B(q)x + C(q)$$
where
$$\begin{array}{rclcl}
B(q) & = & -5\left(\frac{E_4-1}{240}\right) & = & -5\sum_{n\ge 1}\sigma_3(n)q^n \\
C(q) & = & \frac{-5\left(\frac{E_4-1}{240}\right) - 7\left(\frac{E_6-1}{-504}\right)}{12} & = & \sum_{n\ge 1}\left(\frac{-5\sigma_3(n) - 7\sigma_5(n)}{12}\right) q^n
\end{array}$$

This last equation defines an elliptic curve over $Z\ps{q}$ with $j$-invariant
$$j(q) = \frac{1}{q} + 744 + 196884q + 21493760q^2\cdots,$$
discriminant
$$\Delta(q) = q\prod_{n\ge 1}(1-q^n)^{24},$$
and whose restriction to $\ZZ[\frac{1}{6}]\ps{q}$ is the above curve. The nowhere vanishing differential $dx/(x + 2y)$ restricts to give $dX/Y$ over $\ZZ[\frac{1}{6}]\ps{q}$.

\begin{defn} We define the Tate curve $\Tate(q)$ to be the curve given by $y^2 + xy = x^3 + B(q)x + C(q)$ over $\ZZ\ps{q}$, with canonical differential $\omega_\can = dx/(x+2y)$. It is a generalized elliptic curve over $\ZZ\ps{q}$, which is smooth over $\ZZ\ls{q}$ and degenerates to a nodal curve at $q = 0$.
\end{defn}

There are explicit expressions of $x,y$ as functions of $t = e^{2\pi iz}$
\begin{eqnarray}\label{eq_explicit_equations}
x(t) & = & \sum_{j\in\ZZ}\frac{q^jt}{(1-q^jt)^2} - 2\sum_{j\ge 1}\frac{q^j}{1-q^j} \\
y(t) & = & \sum_{j\in\ZZ}\frac{(q^jt)^2}{(1-q^jt)^3}  + \sum_{j\ge 1}\frac{q^j}{1-q^j}\nonumber
\end{eqnarray}

From this, we deduce:

\begin{cor}\label{cor_tate_level_N} Every $n$ torsion point on $\Tate(q) = \CC^\times/q^\ZZ$, namely the points given by $t = (\zeta_n)^iq^{j/n}$ for $0\le i,j\le n-1$, is defined over $\ZZ\ps{q^{1/n}}\otimes_\ZZ \ZZ[1/n,\zeta_n]$.

\sgap

In particular, for any $n\ge 1$, all $\Gamma(n)$-structures on $\Tate(q)$ are defined over $\ZZ\ps{q^{1/n}}\otimes_\ZZ\ZZ[1/n,\zeta_n]$.
\end{cor}

Note that all coefficients of power series in $\ZZ\ps{q^{1/N}}\otimes_\ZZ\ZZ[1/N,\zeta_N]$ have \emph{bounded denominators}.

\subsection{The Unbounded Denominators Conjecture}

It is known that any $q$-expansion of a modular form for a congruence subgroup of $\SL_2(\ZZ)$ with algebraic Fourier coefficients has bounded denominators (ie, an integer multiple of that $q$-expansion has algebraic integer coefficients). The same is not true in general for noncongruence modular forms.

\begin{defn} For a ring $R$, let $R\ls{q} := R\ps{q}[q^{-1}]$, and $R\ls{q^{1/\infty}} := \varinjlim_n R\ls{q^{1/n}}$. Here we fix once and for all a compatible system of $n$th roots $q^{1/n}$, and whenever it makes sense, we define $q^{1/n} := e^{2\pi i\tau/n}$, where $\tau$ is a coordinate on $\hH$.

\sgap

For a finite index $\Gamma\le\SL_2(\ZZ)$, a cusp $c$ of $\Gamma$ of width $d$, and a meromorphic modular function (ie, a form of weight 0) $f$ for $\Gamma$, we say that a $q$-expansion at $c$ is the expansion of $f$ as a Laurent series in $e^{2\pi i \gamma(\tau)/d}$, where $\gamma\in\SL_2(\ZZ)$ and $\gamma(c) = i\infty$. Let $\mu := [\PSL_2(\ZZ):\pm\Gamma]$, $\{c_1,\ldots,c_k\}$ the cusps of $\hH/\Gamma$ of widths $\{d_1,\ldots,d_k\}$, then $\mu = \sum_{i=1}^k d_i$, and $f$ will have at most $d_i$ distinct $q$-expansions at each cusp $c_i$, corresponding to the $d$ uniformizers $\zeta_d^je^{2\pi i \gamma(c)/d}$ where $\zeta_d = e^{2\pi i/d}$ and $0\le j\le d-1$, and at most $\mu$ distinct $q$-expansions in total, with equality precisely when $f$ is not invariant under any strictly larger subgroup of $\SL_2(\ZZ)$.

\sgap

For a finite index $\Gamma\le\SL_2(\ZZ)$ such that the cusp $i\infty$ has width $d$, and a subring $R\subset\CC$, let $M_k(\Gamma,R)$ be the set of meromorphic modular forms of weight $k$ for $\Gamma$, which are holomorphic on $\hH$ and whose $q$-expansion in $e^{2\pi i\tau/d}$ has coefficients in $R$. By definition the \emph{Fourier coefficients} of a form in $M_k(\Gamma,R)$ are the coefficients of this expansion. Thus, $\Spec M_0(\Gamma,\CC)$ is a $\CC$-model of $\hH/\Gamma$. Note that $M_0(\SL_2(\ZZ),R) = R[j]$ for any subring $R\subset\CC$, where $j$ is the classical $j$-invariant with $q$-expansion

$$j(q) = \frac{1}{q} + 744 + 196884q + 21493760q^2 + \cdots$$

\sgap

Let $M_*(\Gamma,R) := \bigoplus_{k\in\ZZ} M_k(\Gamma,R)$. We will say that an element $f\in M_*(\Gamma,R)$ is \emph{primitive}, or \emph{primitive for $\Gamma$} if $f\notin M_*(\Gamma',R)$ for any strictly larger subgroup $\Gamma'\supsetneq\Gamma$. Note that since $M_{2k}(\Gamma,R) = M_{2k}(\pm\Gamma,R)$, there exist primitive elements of even weight for $\Gamma$ if and only if $-I\in\Gamma$, and in this case a primitive modular function (i.e., form of weight 0) is just a primitive element for the function field extension $\Qbar(Y(\Gamma))/\Qbar(j)$.


\sgap

We say that an element $f(q)\in \Qbar\ls{q^{1/\infty}}$ has \emph{bounded denominators} if there exists an integer $n$ such that $n\cdot f(q)\in\Zbar\ls{q^{1/\infty}}$, where $\Zbar$ is the ring of algebraic integers in $\Qbar$. If $f\in M_*(\Gamma,\Qbar)$ and $c$ is a cusp of $\hH/\Gamma$, then we will say that $f$ has \emph{bounded denominators at $c$} if it has a $q$-expansion at $c$ with bounded denominators. 

\sgap

For $\Gamma\le\SL_2(\ZZ)$, let $\Gamma^c$ be its \emph{congruence closure} --- that is, the intersection of all congruence subgroups containing $\Gamma$.
\end{defn}

Let $A\subset\ZZ$ be any subset. Let UBD for $\Gamma$ of weight in $A$, or alternatively $\UBD(\Gamma,A)$, refer to the following property:

\begin{quote}
$\UBD(\Gamma,A)$: If $\Gamma$ is noncongruence, then all primitive elements $f\in\oplus_{k\in A}M_k(\Gamma,\Qbar)$ have unbounded denominators at all cusps.
\end{quote}

For some $k\in\ZZ$, we'll abuse notation and define $\UBD(\Gamma,k) := \UBD(\Gamma,\{k\})$. Note that the statement $\UBD(\Gamma,k)$ is vacuously true if $k$ is even and $-I\notin\Gamma$, since in this case there are no primitive elements in $M_k(\Gamma,\Qbar)$. One should also be wary of the case where $\Gamma$ is noncongruence, but $\pm\Gamma$ is congruence (c.f. Remark \ref{remark_pcong_does_not_imply_cong}), in which case all modular forms of even weight for $\Gamma$ are also forms for $\pm\Gamma$, and hence have bounded denominators, even though $\Gamma$ is noncongruence. Lastly, we will see in Corollary \ref{cor_primitive_bounded_generates_bounded} that if a primitive element in $M_0(\Gamma,\Qbar)$ has bounded denominators, then the same is true of all elements of $M_0(\Gamma,\Qbar)$.

\sgap

\begin{conj}[Unbounded Denominators Conjecture, c.f. \cite{LL11,KL07,Bir94}]\label{conj_UBD}
$$\text{Property $\UBD(\Gamma,\ZZ)$ holds for every finite index subgroup $\Gamma\le\SL_2(\ZZ)$.}$$
In other words,
$$\text{Every genuinely noncongruence modular form has unbounded denominators at all cusps.}$$
\end{conj}



Before moving on to our main discussion, we make the following important reductions.

\begin{prop}\label{prop_reduce_to_weight_0} We have the following:
\begin{itemize} 
\item[(1)] For any finite index $\Gamma\le\SL_2(\ZZ)$, $\UBD(\Gamma,0)$ is equivalent to $\UBD(\Gamma,2\ZZ)$.
\item[(2)] If $\UBD(\Gamma,2\ZZ)$ holds for all finite index $\Gamma\le\SL_2(\ZZ)$, then $\UBD(\Gamma,\ZZ)$ also holds for all finite index $\Gamma\le\SL_2(\ZZ)$.
\end{itemize}
\end{prop}
\begin{proof} Part (1) is essentially Proposition 5 of \cite{KL07}, but we will give a quick proof. Suppose $\UBD(\Gamma,2\ZZ)$ fails --- ie, $\Gamma$ is noncongruence and there is a primitive element $f\in M_k(\Gamma,\Qbar)$ of even weight $k$ with bounded denominators. By multiplying it by some suitable power of the modular discriminant $\Delta$, we may assume $k\ge 0$. Then for a large enough $n\ge 0$ there exists $f^c\in M_{12n-k}(\Gamma^c,\Qbar)$, which must have bounded denominators (here we use the fact that $12n-k$ is even). Then since the $q$-expansion of $\Delta$ lies in $\ZZ\ps{q}^\times$, the quotient $ff^c/\Delta^n$ has bounded denominators and is a primitive element of $M_0(\Gamma,\Qbar)$.

\sgap

For (2), suppose $\UBD(\Gamma,\ZZ)$ fails for some $\Gamma\le\SL_2(\ZZ)$, then $\Gamma$ is noncongruence, and there exists a $f\in M_k(\Gamma,\Qbar)$ of weight $k$ with bounded denominators. If $f$ is even weight, then $\UBD(\Gamma,2\ZZ)$ also fails, so we're done. If $f$ has odd weight, then we may choose some odd weight congruence modular form $g$, then $fg$ is noncongruence and even weight with bounded denominators, so if $fg$ is primitive for $\Gamma'$, then $\UBD(\Gamma',2\ZZ)$ fails.
\end{proof}

Thus, from now on we will focus on the UBD property for modular functions.

\subsection{Level structures and $q$-expansions} We recall some facts about $\mM(G)$ (c.f. \S\ref{ss_basic_properties}), where $G$ is a finite 2-generated group of order $N$.

\sgap

Fix an elliptic curve $E_0$ over $\Qbar$, let $[\varphi] : \pi_1(E_0)\twoheadrightarrow G$ represent $G$-structure on $E_0$. Then $[\varphi]$ represents a geometric point of $\mM(G)$ (resp. $M(G)$). For a scheme $S$ over $\ZZ[1/N]$, let $\mM(G)_S([\varphi])$ (resp. $M(G)_S([\varphi])$) denote the connected component of $\mM(G)_S$ (resp. $M(G)_S$) containing $[\varphi]$. Since $\mM(G)$ and hence $M(G)$ are defined over $\QQ$, there is an action of $G_\QQ := \Gal(\Qbar/\QQ)$ on the components of $M(G)_{\Qbar}$. The component $Y([\varphi])_{\Qbar}$ of $M(G)_{\Qbar}$ containing $[\varphi]$ is defined over the fixed field $K$ of $\Stab_{G_\QQ}\big(M(G)_{\Qbar}([\varphi])\big)$, which is a number field.

\sgap

Since $\mM(G)$ was defined over $\ZZ[1/N]$, the connected component $\yY([\varphi])_{\ZZ_K[1/N]}$ of $\mM(G)_{\ZZ_K[1/N]}$ containing $[\varphi]$ is geometrically connected and its coarse moduli scheme $Y([\varphi])_{\ZZ_K[1/N]}$ is a model of $\hH/\Gamma_{[\varphi]}$ as a Noetherian affine scheme smooth with relative dimension 1 over $\ZZ_K[1/N]$.

\sgap

Suppose $[\varphi],[\varphi']$ are distinct $G$-structures on $E_0/\Qbar$ corresponding to two geometric points on $\mM(G)$. If their stabilizers $\Gamma_{[\varphi]},\Gamma_{[\varphi']}$ are conjugate in $\SL_2(\ZZ)$, then corresponding geometric components $Y([\varphi])_{\Qbar}$ and $Y([\varphi'])_{\Qbar}$ are isomorphic over $\Qbar$. However, even if they are both defined over a number field $K$, they may not a priori be isomorphic over $K$. Furthermore, given a subgroup $\Gamma\le\SL_2(\ZZ)$, there are in general many ways of realizing $\hH/\Gamma$ as a moduli space of elliptic curves with level structures, and the associated moduli problems arising this way may not be isomorphic (c.f. Remark \ref{remark_same_Gamma_different_stack}). This motivates the following definition.

\begin{defn}\label{defn_gpk} For a finite index $\Gamma\le\SL_2(\ZZ)$, we define a $(G,[\varphi],K)$-interpretation (informally, a \emph{moduli interpretation}) of $\Gamma$ to be the following data:
\begin{itemize}
\item[(1)] A finite group $G$ of order $N$.
\item[(2)] A $G$-structure $[\varphi]$ on some elliptic curve $E_0/\Qbar$ represented by $\varphi : F_2\twoheadrightarrow G$, such that the stabilizer $\Gamma_{[\varphi]}$ is normal in $\Gamma$.
\item[(3)] A number field $K$ such that the component $\mM(G)_K([\varphi])$, is geometrically connected, and such that the quotient $\Big(\mM(G)_{\Qbar}([\varphi])\Big)/(\Gamma/\Gamma_{[\varphi]})$ is defined over $K$.
\item[(4)] The coarse moduli scheme $Y(\Gamma)_{\ZZ_K[1/N]}$ of $\yY(\Gamma)_{\ZZ_K[1/N]} := \Big(\mM(G)_{\ZZ_K}([\varphi])\Big)/(\Gamma/\Gamma_{[\varphi]})$
\end{itemize}
\end{defn}
\begin{remark*} Firstly, the existence of such a $G$ and $\varphi$ is guaranteed by the \emph{congruence subgroup property} of $\Out(F_2)$ (c.f. Corollary \ref{cor_CSP}). Second, note that $Y(\Gamma)_{\ZZ_K[1/N]}$ is determined by the data of (1),(2),(3). We list it here to officially highlight our notation for a moduli-theoretic model of $\hH/\Gamma$. 
\end{remark*}

Thus, for any finite index $\Gamma\le\SL_2(\ZZ)$, upon choosing a $(G,[\varphi],K)$-interpretation, for any elliptic curve $E$ over a $\ZZ_K[1/N]$-scheme $S$, we may speak of $\Gamma$-structures on $E/S$. In this case, a $\Gamma$-structure will be called \emph{noncongruence} (resp. congruence, fine), if $\Gamma$ is noncongruence (resp. congruence, torsion-free). Thus, in general a $\Gamma$-structure is an equivalence class of $G$-structures.



\sgap

All of our subsequent developments are based on the following key fact.

\begin{lemma}\label{lemma_qexp}
Let $\Gamma\le\SL_2(\ZZ)$ be finite index. Any homomorphism
$$M_0(\Gamma,\CC)\longrightarrow\CC\ls{q^{1/\infty}}$$
sending $j\mapsto j(q)$ is the homomorphism defined by taking $q$-expansions at a cusp $c$. If the cusp $c$ has width $w$, then the image lands in $\CC\ls{q^{1/w}}$.
\end{lemma}

\begin{proof} Note that $\Spec M_0(\Gamma,\CC)\cong\hH/\Gamma$. Let $\mu := \deg(\hH/\Gamma\rightarrow Y(1)_\CC)$. Let $c_1,\ldots,c_k$ be the cusps of $\hH/\Gamma$ with widths $w_1,\ldots,w_k$, then $\mu = \sum_{i=1}^k w_i$. Let $w := \lcm(w_1,\ldots,w_k)$, then taking $q$-expansions at each cusp gives us $k$ homomorphisms
$$M_0(\Gamma,\CC)\longrightarrow\CC\ls{q^{1/w}}\subset\CC\ls{q^{1/\infty}}$$
sending $j\in M_0(\Gamma,\CC)$ to $j(q)$. For each cusp $c_i$, sending $q^{1/w_i}\mapsto\zeta_{w_i}q^{1/w_i}$ gives $w_i$ distinct homomorphisms $M_0(\Gamma,\CC)\rightarrow\CC\ls{q^{1/\infty}}$. Thus, the $w_i$ different $q$-expansions at each cusp $c_i$ give a total of $\sum_{i=1}^k w_i = \mu$ distinct homomorphisms $M_0(\Gamma,\CC)\rightarrow\CC\ls{q^{1/\infty}}$ sending $j\mapsto j(q)$, and hence $\mu$ distinct lifts

$$\xymatrix{
 & \hH/\Gamma \ar[d]^\mu \\
 \Spec\CC\ls{q^{1/\infty}}\ar[r]_{\qquad j}\ar@{-->}[ru]^{\text{$q$-exp}} & Y(1)_\CC
}$$



Since the map $\hH/\Gamma\rightarrow Y(1)_\CC$ is finite generically \'{e}tale of degree $\mu$, there are precisely $\mu$ maps $\Spec\CC\ls{q^{1/\infty}}\rightarrow \hH/\Gamma$ over $Y(1)_\CC$, which are all accounted for by taking $q$-expansions at some cusp. Thus, any homomorphism $M_0(\Gamma,\CC)\rightarrow \CC\ls{q^{1/\infty}}$ sending $j\mapsto j(q)$ must be one of these.



\end{proof}

\begin{prop}\label{prop_level_structures_are_qexps} Let $\Gamma\le\SL_2(\ZZ)$ be finite index. Fix a $(G,[\varphi],K)$-interpretation of $\Gamma$, from which we get a model $Y(\Gamma)_{\ZZ_K[1/N]}$ of $\hH/\Gamma$ smooth over $\ZZ_K[1/N]$ which is a coarse moduli scheme for $\Gamma$-structures. Let $R\subset\CC$ be a subring containing $\ZZ_K[1/N]$, and $\Omega\subset\CC\ls{q^{1/\infty}}$ be a subring containing $R\ls{q}$, \emph{such that there exists a $\Gamma$-structure $\alpha$ on $\Tate(q)/\Omega$}. Then $\alpha$ determines a map
$$\alpha_* : \Spec \Omega\longrightarrow Y(\Gamma)_R$$
and this map is precisely $q$-expansion at some cusp in the sense of Lemma \ref{lemma_qexp}.
\end{prop}
\begin{proof} Since the $j$-invariant of $\Tate(q)$ is $j(q)$, $\alpha_*$ sends $j\mapsto j(q)$. We have a commutative diagram
$$\xymatrix{
\Spec\CC\ls{q^{1/\infty}}\ar[d]\ar@{-->}[r]^{\qquad (\alpha_\CC)_*} & Y(\Gamma)_\CC\ar[d]\ar[r] & \Spec\CC\ar[d] \\
\Spec \Omega\ar[r]^{\alpha_*} & Y(\Gamma)_{R}\ar[r] & \Spec R
}$$
where the rightmost square is cartesian and the dashed map $(\alpha_\CC)_*$ is induced by the natural map $\Spec\CC\ls{q^{1/\infty}}\rightarrow\Spec\CC$. Since $\alpha_*$ sends $j\mapsto j(q)$, so must $(\alpha_\CC)_*$, but by Lemma \ref{lemma_qexp}, $(\alpha_\CC)_*$ must be $q$-expansion. Since $Y(\Gamma)_R$ is flat over $R$, the vertical maps induce injections on global sections, and we conclude that $\alpha_*$ is also just $q$-expansion.
\end{proof}

We have the following nice corollary.
\begin{cor}[``Canonical Models'']\label{cor_nice} In the notation of Proposition \ref{prop_level_structures_are_qexps}, if $R = k$ is a subfield of $\CC$ such that for $\Omega := k\ls{q^{1/\infty}}$, there exists a $\Gamma$-structure on $\Tate(q)/\Omega$,
then after possibly replacing $\Gamma$ by a conjugate, we have $Y(\Gamma)_k = \Spec M_0(\Gamma,k)$.

\sgap

In particular, $Y(\Gamma)_{\Qbar} = \Spec M_0(\Gamma,\Qbar)$, and for a large enough number field $K$, we have $Y(\Gamma)_K = \Spec M_0(\Gamma,K)$.
\end{cor}
\begin{proof} Recall that $M_0(\Gamma,k)$ was defined to be the modular functions for $\Gamma$ whose $q$-expansions at $i\infty$ have coefficients in $k$. Thus, we will replace $\Gamma$ by a conjugate such that the map $\alpha_*$ of Proposition \ref{prop_level_structures_are_qexps} corresponds to $q$-expansion at $i\infty$. The map $(\alpha_\CC)_*$ in the proof of Proposition \ref{prop_level_structures_are_qexps} certainly induces an injection on global sections, and since the vertical maps are as well, we conclude that $\alpha_*$ also induces an injection on global sections.

\sgap

Suppose $Y(\Gamma)_k = \Spec B$, then $B$ sits as an $k$-subalgebra of $\Omega = k\ls{q^{1/\infty}}$. Note that $Y(\Gamma)_\CC = \Spec M_0(\Gamma,\CC)$. 
We have a diagram
$$\xymatrix{
\CC\ls{q^{1/\infty}} & M_0(\Gamma,\CC)\ar[l]_{\qquad(\alpha_\CC)^*\quad} \\
k\ls{q^{1/\infty}}\ar[u] & B\ar[l]_{\quad\alpha^*}\ar[u]
}$$
where all maps are inclusions. We would like to show that $B = \underbrace{M_0(\Gamma,\CC)\cap k\ls{q^{1/\infty}}}_{M_0(\Gamma,k)}$. Consider the natural map
$$M_0(\Gamma,k)\otimes_k\CC\stackrel{f}{\longrightarrow}M_0(\Gamma,\CC)$$
By coarse base change, the injection $B\hookrightarrow M_0(\Gamma,\CC)$ induces an isomorphism $B\otimes_k\CC\rightiso M_0(\Gamma,\CC)$. Thus, since $B\subset M_0(\Gamma,k)$, we see that $f$ is surjective.

\sgap

On the other hand, $f$ is a restriction of the visibly injective map $k\ls{q^{1/\infty}}\otimes_k\CC\hookrightarrow\CC\ls{q^{1/\infty}}$, and thus must be injective, hence an isomorphism. We may consider the sequence
\begin{eqnarray}\label{fppf_descent}
0\rightarrow B\rightarrow M_0(\Gamma,k)\rightarrow 0
\end{eqnarray}
By the above discussion, tensoring with $\CC$ yields an exact sequence
$$0\rightarrow B\otimes_k\CC\rightiso M_0(\Gamma,k)\otimes_k\CC\rightarrow 0$$
but since $\CC$ is faithfully flat over $k$, the original sequence (\ref{fppf_descent}) must have been exact as well, hence $B\cong M_0(\Gamma,k)$.
\end{proof}

\begin{remark} Note that while the hypothesis of the corollary depends on a choice of a moduli interpretation of $\Gamma$, the conclusion does not. Thus, while a given $\hH/\Gamma$ may have multiple moduli interpretations, all moduli-theoretic models of $\hH/\Gamma$ are ``commensurate'' with the model given by $\Spec M_0(\Gamma,K)$ for a suitable $K$ as described in the corollary.

\end{remark}

\subsection{The bad primes for UBD} In this section we give a quick application of our moduli interpretations to show that for a finite index $\Gamma\le\SL_2(\ZZ)$, elements of $M_*(\Gamma,\Qbar)$ can only have unbounded denominators at primes dividing $N := |G|$. Specifically,

\begin{thm}\label{thm_bad_primes} Let $\Gamma\le\SL_2(\ZZ)$ be finite index. For any $(G,[\varphi],K)$-interpretation of $\Gamma$, if $p$ is a prime not dividing $N$, then all elements of $M_*(\Gamma,\Qbar)$ have bounded denominators at $p$.
\end{thm}

\begin{remark*} This explains the bad primes first noticed by Atkin and Swinnerton-Dyer in \cite{ASD71}, and the ``$M$'' appearing in \cite{Sch85}, \S5. We note that by possibly using another $(G',[\varphi'],K')$-interpretation for $\Gamma$, we can deduce that the bad primes must all divide $\gcd(|G|,|G'|)$. However, since the level $l(\Gamma)$ divides $|G|$ (c.f. Corollary \ref{cor_level_divides_N}), we cannot rule out the primes dividing $l(\Gamma)$ this way, nor should we expect to if $\UBD$ is to hold.
\end{remark*}

Before we give the proof, we recall Abhyankar's lemma.

\begin{prop}[Abhyankar's Lemma]\label{lemma_abhyankar} Let $A$ be a regular local ring with quotient field $K$, $L$ a finite Galois field extension of $K$ with Galois group $G$, and $B$ the integral closure of $A$ in $L$. Let $x$ be a regular parameter, i.e. an element of a regular system of parameters $x_1,\ldots,x_{n-1},x$ of $A$, and suppose that the principal ideal $A\cdot x$ is the only prime ideal of height 1 in $A$ possibly ramified in $B$. We consider a geometric point
$$\alpha : \Spec\Omega\longrightarrow\Spec B$$
that is localized at a prime ideal over $A\cdot x$. Suppose the order $e := |G_\alpha|$ of the ramification group $G_\alpha := \{\sigma\in G : \sigma\circ\alpha = \alpha\}$ is not divisible by the characteristic of $\Omega$. Then $B$ is regular. More precisely: Let $\mf{m}$ be a maximal ideal of $B$ and let $\tilde{A}$ resp. $\tilde{B}_\mf{m}$ be strict Henselizations of $A$ and $B_\mf{m}$. Then
$$\tilde{B}_\mf{m} = \tilde{A}[\sqrt[e]{x}].$$
The elements $x_1,\ldots,x_{n-1},\sqrt[e]{x}$ form a regular system of parameters for $\tilde{B}_\mf{m}$. The ramification degree $e$ is relatively prime to the characteristic of the residue field of $A$.
\end{prop}
\begin{proof} See \cite{FK13}, appendix \S A I.11.
\end{proof}

\begin{cor}\label{cor_abhyankar} Let $R$ be a Henselian discrete valuation ring of mixed characteristic $(0,p)$ residue field $k$, and uniformizer $\pi$. Then for $R'$ finite \'{e}tale over $R$, and $e\ge 1$ coprime to $p$, $R'\ls{q^{1/e}}$ is finite \'{e}tale over $R\ls{q}$, and moreover every connected finite \'{e}tale extension of $R\ls{q}$ is dominated by some $R'\ls{q^{1/e}}$.
\end{cor}
\begin{proof} Since $R'/R$ is finite, we have $R\ps{q}\otimes_R R' = R'\ps{q}$ so $R'\ps{q}$ is \'{e}tale over $R\ps{q}$. Since $R'\ps{q^{1/e}}$ over $R'\ps{q}$ is ramified only over $(q)$, upon inverting $q$ we see that $R'\ls{q^{1/e}}\rightarrow R'\ls{q}\rightarrow R\ls{q}$ is \'{e}tale.

\sgap

The ring $R\ps{q}$ is a regular Henselian (c.f. \cite{San68}) Noetherian local ring with maximal ideal $(\pi,q)$ and residue field $k$. The residue field of the height 1 prime $(q)$ is characteristic 0, so for any connected finite extension $A$ of $R\ps{q}$ \'{e}tale away from $(q)$, any ramification above $(q)$ is necessarily tame. Since $R\ps{q}$ is Henselian, $A$ is necessarily local, so by Abhyankar's lemma, we have $A^\sh = R\ps{q}^\sh[q^{1/e}]$ with $e$ coprime to $p$. Further, the finite \'{e}tale extensions of $R\ps{q}$ are in functorial bijection with finite separable extensions of its residue field $k$, which in turn correspond to finite \'{e}tale extensions $R'$ of $R$. Since $R^\sh$ is the colimit of all such extensions and $A$ is finite over $R\ps{q}$, there exists an $R'$ finite \'{e}tale over $R$ with $R'\ps{q}[q^{1/e}]$ dominating $A$. Inverting $q$ then shows that $R'\ls{q^{1/e}}$ dominates $A[q^{-1}]$.
\end{proof}

\sgap

We now prove Theorem \ref{thm_bad_primes}.
\begin{proof} The $(G,[\varphi],K)$-interpretation gives us a geometrically connected stack $\yY(\Gamma)$ over $\ZZ_K[1/N]$, finite \'{e}tale over $\mM(1)$, whose coarse moduli scheme is $Y(\Gamma)_{\ZZ_K[1/N]}$. Let $\mf{p}$ be a prime ideal of $\ZZ_K$ not dividing $N$, then let $\ZZ_\mf{p}$ be the localization of $\ZZ_K$ at $\mf{p}$, and $\ZZ_\mf{p}^\h$ its Henselization. Fix an embedding of $\ZZ_\mf{p}^\h$ into $\Qbar$. The stack $\yY(\Gamma)_{\ZZ_\mf{p}^\h}$ is finite \'{e}tale over $\mM(1)_{\ZZ_\mf{p}^\h}$. The Tate curve $\Tate(q)$ over $\ZZ_\mf{p}^\h\ls{q}$ determines a map $\ZZ_\mf{p}^\h\ls{q}\rightarrow \mM(1)_{\ZZ_\mf{p}^\h}$, and the pullback $\yY(\Gamma)_{\ZZ_\mf{p}^\h}\times_{\mM(1)_{\ZZ_\mf{p}^\h}} \ZZ_\mf{p}^\h\ls{q}$ is precisely the scheme of $\Gamma$-structures on $\Tate(q)/\ZZ_\mf{p}^\h\ls{q}$, and thus is trivialized over a finite \'{e}tale extension $\Omega$ of $\ZZ_\mf{p}^\h\ls{q}$. By Lemma \ref{cor_abhyankar}, we may take $\Omega = \ZZ_\mf{q}^\h\ls{q^{1/e}}$, where $\mf{q}$ is a prime in some finite extension of $K$ lying above $\mf{p}$. This means that all $\Gamma$-structures on $\Tate(q)$ are defined over $\ZZ_\mf{q}^\h\ls{q^{1/e}}$, so by Proposition \ref{prop_level_structures_are_qexps} we find that all $q$-expansions of global sections of $Y(\Gamma)_{\ZZ_\mf{q}^\h}$ lie in $\ZZ_\mf{q}^\h\ls{q^{1/e}}$.

\sgap

Let $B$ denote the ring of global sections of $Y(\Gamma)_{\ZZ_\mf{q}^\h}$, then by Corollary \ref{cor_nice}, $Y(\Gamma)_{\Qbar} = \Spec M_0(\Gamma,\Qbar)$, so we have 
$$M_0(\Gamma,\Qbar) = B\otimes_{\ZZ_\mf{q}^\h}\Qbar\subset \ZZ_\mf{q}^\h\ls{q^{1/e}}\otimes_{\ZZ_\mf{q}^\h}\Qbar$$
where the last ring consists only of Laurent series in $q^{1/e}$ with bounded denominators at $p$. This proves the result for modular functions. The full result then follows from Proposition \ref{prop_reduce_to_weight_0}.
\end{proof}

\subsection{The geometric Unbounded Denominators Conjecture}\label{ss_geometric_UBD} In this section we give geometric interpretations of the Unbounded Denominators Conjecture.

\sgap

\textbf{Notation.} For a finite index subgroup $\Gamma\le\SL_2(\ZZ)$, we will make extensive use of the notion of a ``$\Gamma$-structure'', with the understanding that the precise definition depends on a choice of a $(G,[\varphi],K)$-interpretation of $\Gamma$ (c.f. Definition \ref{defn_gpk}). However, if $\Gamma$ is definitively a congruence subgroup (for example, if $\Gamma = \Gamma^c,\Gamma_1(n),\Gamma(n)$), then we will defer to the classical moduli interpretations in terms of torsion points.

\sgap

For an algebraic extension $L/\QQ$, let $\ZZ_L$ be its ring of integers. We define
$$B(L,q) := \varinjlim_{K}\ZZ_K\ls{q^{1/\infty}}\otimes_{\ZZ_K}K$$
where the limit ranges over all finite extensions $K$ of $\QQ$ contained in $L$. Note that for a number field $K$, $B(K,q) = \ZZ_K\ls{q^{1/\infty}}\otimes_{\ZZ_K}K$ and is precisely the subring of $K\ls{q^{1/\infty}}$ consisting of $q$-series with bounded denominators. By Proposition \ref{prop_level_structures_are_qexps}, we find that for any $\Gamma\le_f\SL_2(\ZZ)$, there is a number field $K$ for which the $q$-expansions of elements in $M_0(\Gamma,\Qbar)$ lie in $K\ls{q^{1/\infty}}$. Thus, an element in $f\in M_0(\Gamma,\Qbar)$ has bounded denominators at some cusp $c$ if and only if its $q$-expansion at $c$ lies in $B(K,q)$.

\sgap

We may think of Proposition \ref{prop_level_structures_are_qexps} as saying that if there exists noncongruence $\Gamma$-structures on $\Tate(q)$ over $B(\Qbar,q)$, then $\UBD$ must be false. We now work towards proving a converse.

\sgap

Note that since each $B(K,q)$ is a Dedekind domain, we have the following.

\begin{lemma}\label{lemma_integrally_closed_bounded_denominators} $B(\Qbar,q)$ is an integrally closed domain.
\end{lemma}

This immediately implies the following two useful facts.

\begin{cor}\label{cor_bounded_generates_bounded} Let $\Gamma_1,\Gamma_2$ be finite index subgroups of $\SL_2(\ZZ)$ such that one of them contains $-I$. Then $M_0(\Gamma_1\cap\Gamma_2,\Qbar)$ has bounded denominators at some cusp $c$ if and only if $M_0(\Gamma_1,\Qbar),M_0(\Gamma_2,\Qbar)$ both have bounded denominators at $c$.
\end{cor}
\begin{proof} 
The assumption that one of $\Gamma_1,\Gamma_2$ contains $-I$ allows us to identify $\Qbar(Y(\Gamma_1\cap\Gamma_2))$ with the function field of an irreducible component $Y$ of the fiber product
$$Y(\Gamma_1)_{\Qbar}\times_{Y(1)_{\Qbar}}Y(\Gamma_2)_{\Qbar}$$
(In general $\Qbar(Y(\Gamma_1\cap\Gamma_2))$ may be a quadratic extension of the function field of $Y$) It's clear that the global sections of $Y$ have bounded denominators at $c$, though in general $Y$ may not be smooth. Nonetheless, by \ref{cor_nice}, we may identify $M_0(\Gamma_1\cap\Gamma_2,\Qbar)$ with the integral closure of $\Qbar[j]$ inside $\Qbar(Y(\Gamma_1\cap\Gamma_2))$, and the result follows by noting that $B(\Qbar,q)$ is integrally closed.
\end{proof}

\begin{cor}\label{cor_primitive_bounded_generates_bounded} Suppose there is a primitive $f\in M_0(\Gamma,\Qbar)$. If $f$ has bounded denominators at a cusp $c$, then every modular function in $M_0(\Gamma,\Qbar)$ has bounded denominators at $c$.
\end{cor}
\begin{proof} Let $\Gamma_0$ be a normal subgroup of $\SL_2(\ZZ)$ contained in $\Gamma$. Then the function field $\Qbar(Y(\Gamma_0))$ is Galois over $\Qbar(j)$ with Galois group $\PSL_2(\ZZ)/\pm\Gamma_0$. By the Galois correspondence there is a bijection between its intermediate fields and subgroups $\pm\Gamma'\le\PSL_2(\ZZ)$ containing $\pm\Gamma_0$, which are all accounted for by the function fields $\Qbar(Y(\Gamma'))$. Certainly $\Qbar(j,f)\subset\Qbar(Y(\Gamma))$, and by our assumption on $f$, $\Qbar(j,f)$ is not contained in a smaller field, so $\Qbar(j,f) = \Qbar(Y(\Gamma))$. This implies that $\Qbar(Y(\Gamma)) = \Frac\;\Qbar[j][f]$, so that $M_0(\Gamma,\Qbar)$ is the integral closure of $\Qbar[j][f]$. The result follows from the fact that $B(\Qbar,q)$ is integrally closed.
\end{proof}



While a $\Gamma$-structure $\alpha$ on $\Tate(q)/B(\Qbar,q)$ certainly induces a map $\Spec B(\Qbar,q)\rightarrow Y(\Gamma)$ sending $j$ to $j(q)$, because of the existence of twists of $\Tate(q)$, it may not be true that every map $\Spec B(\Qbar,q)\rightarrow Y(\Gamma)$ sending $j$ to $j(q)$ comes from a level structure on $\Tate(q)/B(\Qbar,q)$. However, just from degree considerations, we have

\begin{lemma}\label{lemma_structures_on_twists} Let $\Gamma$ be torsion-free, and suppose all $\Gamma$-structures are defined over $\Tate(q)/B(\Qbar,q)$, then any morphism $\Spec B(\Qbar,q)\rightarrow Y(\Gamma)_{\Qbar}$ sending $j$ to $j(q)$ coincides with the map induced by an isomorphism class of pairs of the form $(\Tate(q)/B(\Qbar,q),\alpha)$ where $\alpha$ is a $\Gamma$-structure on $\Tate(q)$ over $B(\Qbar,q)$.
\end{lemma}

\begin{proof} The argument is identical to that of Lemma \ref{lemma_qexp}. It follows from the fact that $Y(\Gamma)_{\Qbar}$ is generically \'{e}tale over $Y(1)_{\Qbar}$, and under our hypotheses, the $\Gamma$-structures $\alpha$ on $\Tate(q)/B(\Qbar,q)$ exhaust all possible maps $\Spec B(\Qbar,q)\rightarrow Y(\Gamma)_{\Qbar}$ over $Y(1)_{\Qbar}$.
\end{proof}

\begin{cor}\label{cor_congruence_closure_torsion_free} Let $\Gamma$ be noncongruence such that $\Gamma^c$ is torsion-free. Then if there is a cusp at which all modular functions in $M_0(\Gamma,\Qbar)$ have bounded denominators, then there must exist a $\Gamma$-structure on $\Tate(q)/B(\Qbar,q)$.
\end{cor}
\begin{proof} 
As a torsion-free congruence subgroup, $\Gamma^c$ must have level $N\ge 3$, and be a quotient of $Y(n)$ ($n\ge 3$) and hence all $\Gamma^c$ structures must be defined over $\Tate(q)/B(\Qbar,q)$ (c.f. Cor \ref{cor_tate_level_N}).

\sgap

We have a finite \'{e}tale cover $Y(\Gamma)_{\Qbar}\rightarrow Y(\Gamma^c)_{\Qbar}$. Taking $q$-expansions at some cusp of $Y(\Gamma)_{\Qbar}$ gives a map $\Spec B(\Qbar,q)\stackrel{u}{\longrightarrow}Y(\Gamma)_{\Qbar}$ which sends $j\mapsto j(q)$. Since the composition
$$\Spec B(\Qbar,q)\stackrel{u}{\longrightarrow}Y(\Gamma)_{\Qbar}\longrightarrow Y(\Gamma^c)_{\Qbar}$$
must also send $j\mapsto j(q)$, then by Lemma \ref{lemma_structures_on_twists}, and the fact that all $\Gamma^c$-structures on $\Tate(q)$ are defined over $B(\Qbar,q)$, the composition corresponds to a $\Gamma^c$-structure on $\Tate(q)/B(\Qbar,q)$, so the morphism $u$ must correspond to a $\Gamma$-structure on $\Tate(q)/B(\Qbar,q)$.
\end{proof}


\sgap

\begin{thm}[Main Equivalence]\label{thm_level_structures_and_UBD} Let $\Gamma\le\SL_2(\ZZ)$. The following are equivalent:
\begin{itemize}
\item[(a)] If there exists a $\Gamma$-structure on $\Tate(q)$ over $B(\Qbar,q)$, then for some cusp $c$, all modular functions in $M_0(\Gamma,\Qbar)$ have bounded denominators at $c$. 
\item[(b)] Conversely, if at some cusp $c$, all modular functions in $M_0(\Gamma,\Qbar)$ have bounded denominators at $c$, then there exists a $\pm\Gamma$-structure on $\Tate(q)/B(\Qbar,q)$.
\end{itemize}
In particular, (a) and (b) are equivalent if $-I\in\Gamma$, and in this case, by \ref{prop_reduce_to_weight_0} and \ref{cor_primitive_bounded_generates_bounded}, we find that $\UBD(\Gamma,2\ZZ)$ is equivalent to the nonexistence of noncongruence $(\Gamma = \pm\Gamma)$-structures on $\Tate(q)$ over $B(\Qbar,q)$.
\end{thm}

\begin{proof}  We fix a $(G,[\varphi],\Qbar)$-interpretation of $\Gamma$. Suppose there exists a $\Gamma$-structure $\alpha$ defined on $\Tate(q)/B(\Qbar,q)$, then this corresponds to a morphism $\alpha_* : \Spec B(\Qbar,q)\rightarrow Y(\Gamma)_{\Qbar}$ sending $j\mapsto j(q)$, 
which at the level of rings is
$$M_0(\Gamma,\Qbar)\stackrel{\alpha^*}{\longrightarrow}B(\Qbar,q)$$

sending $j\mapsto j(q)$. By Prop \ref{prop_level_structures_are_qexps}, this map is actually just a $q$-expansion at some cusp $c$, so this shows that all modular functions for $\Gamma$ have bounded denominators at $c$.

\gap

Conversely, suppose at some cusp $c$, all modular functions in $M_0(\Gamma,\Qbar)$ has bounded denominators, then the same is true of all modular functions in $M_0(\pm\Gamma,\Qbar)$. 
Let $\Gamma'$ be any torsion-free congruence subgroup, then we know that $M_0(\Gamma',\Qbar)$ also has bounded denominators at $c$, and hence by Lemma \ref{cor_bounded_generates_bounded}, $M_0(\pm\Gamma\cap\Gamma',\Qbar)$ has bounded denominators. We may consider the product
$$\yY(\pm\Gamma)\times_{\mM(1)}\yY(\Gamma')\qquad\text{as stacks over $\Qbar$}$$
Let $\yY(\pm\Gamma\cap\Gamma')_{\Qbar}$ be a connected component of the product, then since $\pm\Gamma\cap\Gamma'$ is torsion-free, it is represented by the modular curve $Y(\pm\Gamma\cap\Gamma')_{\Qbar} = \Spec M_0(\pm\Gamma\cap\Gamma',\Qbar)$, which is a fine moduli scheme for elliptic curves equipped with a $\pm\Gamma\cap\Gamma'$-structure, where a $\pm\Gamma\cap\Gamma'$-structure is equivalent to the data of both a $\pm\Gamma$-structure and a $\Gamma'$-structure (see the proof of Corollary \ref{cor_representable_cofinal} if $\Gamma' = \Gamma_1(p)$). By \ref{cor_congruence_closure_torsion_free}, we find that $\Tate(q)/B(\Qbar,q)$ admits a $\pm\Gamma\cap\Gamma'$-structure, and hence also a $\pm\Gamma$-structure.
\end{proof}

\begin{remark} Note that by Remark \ref{remark_pcong_does_not_imply_cong}, there exist noncongruence subgroups $\Gamma$ for which $\pm\Gamma$ is congruence. This is why the condition $-I\in\Gamma$ is necessary to produce an equivalence between $\UBD(\Gamma,2\ZZ)$ and the nonexistence of noncongruence $\Gamma$-structures on $\Tate(q)/B(\Qbar,q)$. Roughly speaking, the existence of $\Gamma$ structures on $\Tate(q)/B(\Qbar,q)$ is governed by the property ``all modular functions for $\Gamma$ has bounded denominators'', and since $M_0(\Gamma,\Qbar) = M_0(\pm\Gamma,\Qbar)$, this property cannot distinguish between $\Gamma$ and $\pm\Gamma$.

\end{remark}




\begin{cor} If the unbounded denominators conjecture is true, then there does not exist any $G$-Galois cover of $\Tate(q)/B(\Qbar,q)$ unramified away from the identity, where $G$ is any nonabelian simple group of order $\le 29120$ or an extension of such a group.
\end{cor}
\begin{proof} Computational data show that all nonabelian simple groups of order $\le 29120$ are purely noncongruence in the sense of Definition \ref{def_purity}. Thus, if $G$-structures exist, then by Theorem \ref{thm_level_structures_and_UBD}, there exist genuinely noncongruence modular functions with bounded denominators.
\end{proof}

If we assume Conjecture \ref{conj_nonsolvable_implies_noncongruence}, which says that every nonsolvable group $G$ is purely noncongruence, then the following would be a consequence of UBD:
\begin{conj}[Geometric UBD, Version A]\label{conj_geometric_UBD_A}
$$\text{Every Galois cover of $\Tate(q)/B(\Qbar,q)$ unramified away from $O$ is solvable.}$$
\end{conj}
Essentially, this says that nonsolvable Galois covers of $\Tate(q)$ must have defining coefficients which are $q$-series with unbounded denominators. Conversely, Theorem \ref{thm_level_structures_and_UBD} tells us that if every Galois cover of $\Tate(q)/B(\Qbar,q)$ is solvable, then for any nonsolvable $G$-structure $[\varphi]$, no modular form primitive for $\Gamma_{[\varphi]}$ can have bounded denominators.

\sgap

\sgap

To get a geometric statement equivalent to UBD which requires no conjectures, we may consider another viewpoint. In light of Theorem \ref{thm_level_structures_and_UBD}, we see that UBD is equivalent to the nonexistence of noncongruence level structures on $\Tate(q)/B(\Qbar,q)$. We can express this as follows.

\gap

Let $t_* : \Spec B(\Qbar,q)\rightarrow\mM(1)_{\Qbar}$ be the morphism corresponding to the Tate curve $\Tate(q)$ over $B(\Qbar,q)$. Note that a $\Gamma$-structure is defined over $\Tate(q)/B(\Qbar,q)$ precisely if
$$T := t_*(\pi_1(\Spec B(\Qbar,q))) \subseteq\gamma\ol{\Gamma}\gamma^{-1}$$
for some $\gamma\in\SL_2(\ZZ)$, where $\ol{\Gamma}$ denotes the closure of $\Gamma$ inside $\widehat{\SL_2(\ZZ)}$. On the other hand, we know that every $\Gamma(n)$-structure is defined on $\Tate(q)/B(\Qbar,q)$, so we must have

\begin{equation}\label{eq_UBD_B}
T \subseteq\bigcap_{n\ge 1}\ol{\Gamma(n)}
\end{equation}
This last group is precisely the \emph{congruence kernel} :
$$ \bigcap_{n\ge 1}\ol{\Gamma(n)} = \ker\left(\widehat{\SL_2(\ZZ)}\longrightarrow\SL_2(\widehat{\ZZ})\right)$$
\begin{remark*} It is a result of Mel'nikov \cite{Mel76} that this kernel is isomorphic to $\widehat{F_\omega}$, the free profinite group on countably many generators.
\end{remark*}

Note that we have an exact sequence:

\begin{equation}\label{eq_congruence_kernel}
1\longrightarrow\bigcap_{n\ge 1}\ol{\Gamma(n)}\longrightarrow \underbrace{\widehat{\SL_2(\ZZ)}}_{\pi_1(\mM(1)_{\Qbar})}\longrightarrow\SL_2(\widehat{\ZZ})\longrightarrow 1
\end{equation}

\sgap

Returning to our discussion of UBD, note that as the continuous image of one profinite group inside another, $T$ must be closed in $\widehat{\SL_2(\ZZ)}$. Since closed subgroups of profinite groups are intersections of open subgroups, the inclusion of (\ref{eq_UBD_B}) is strict if and only if there exists an open subgroup $\ol{\Gamma}$ of $\widehat{\SL_2(\ZZ)}$ containing $T$, but not containing $\bigcap_{n\ge 1}\ol{\Gamma(n)}$, but such a subgroup must correspond to a noncongruence subgroup $\Gamma$ of $\SL_2(\ZZ)$. Thus, the above discussion shows that a noncongruence $\Gamma$-structure on $\Tate(q)/B(\Qbar,q)$ exists if and only if $T = \ker\left(\widehat{\SL_2(\ZZ)}\longrightarrow\SL_2(\widehat{\ZZ})\right)$. From this, we get another interpretation of UBD:

\begin{conj}[Geometric UBD, Version B]\label{conj_geometric_UBD_B} In the above notation, the following sequence is exact
\begin{equation*}
\pi_1(\Spec B(\Qbar,q))\stackrel{t_*}{\longrightarrow}\widehat{\SL_2(\ZZ)}\longrightarrow\SL_2(\widehat{\ZZ})\longrightarrow 1
\end{equation*}
\end{conj}

As promised, we have:

\begin{thm}\label{thm_UBDB=UBD} Conjecture \ref{conj_geometric_UBD_B} is equivalent to the UBD Conjecture (\ref{conj_UBD}).
\end{thm}
\begin{proof} The discussion above shows that the exactness of the sequence in Conjecture \ref{conj_geometric_UBD_B} is equivalent to the nonexistence of noncongruence $\Gamma$-structures on $\Tate(q)$ over $B(\Qbar,q)$. This is equivalent to the UBD conjecture by Theorem \ref{thm_level_structures_and_UBD} and Proposition \ref{prop_reduce_to_weight_0}.
\end{proof}


\sgap

We end with a final remark. The Inverse Galois Problem for a field $K$ asks if every finite group can be realized as a Galois group over $K$. This would certainly be the case if $\Gal(\ol{K}/K)$ had $\widehat{F_\omega}$ as a quotient. In light of the above discussion and the result of Mel'nikov \cite{Mel76}, we find that the UBD conjecture would imply the Inverse Galois Problem for $\Frac\;B(\Qbar,q)$, but by a beautiful result of Harbater \cite{Har84}, the Inverse Galois Problem actually holds for such fields.


\gap

\begin{appendix}
\section{Proof of Theorem \ref{thm_outer_representation}}\label{proof_thm_outer_representation} 


\sgap

Let $\eE^\circ$ be the universal elliptic curve over $\mM(1)_{\Qbar}$ with the identity section removed, and let $\eE_{x_0}^\circ$ be a geometric fiber above some geometric point $x_0\in\mM(1)_{\Qbar}$ corresponding to an elliptic curve $E_0/\Qbar$, then $\eE_{x_0} = E_0$, and there is an exact sequence
$$1\rightarrow\pi_1(E_0^\circ)\rightarrow\pi_1(\eE^\circ)\rightarrow\pi_1(\mM(1)_{\Qbar})\rightarrow 1,$$
from which we get an outer representation
\begin{equation*}
\widehat{\SL_2(\ZZ)}\cong\pi_1(\mM(1)_{\Qbar})\stackrel{\rho_{\Qbar}}{\longrightarrow}\Out(\pi_1(E_0^\circ)) \cong \Out(\widehat{F_2})
\end{equation*}
On the other hand, there is a classical exact sequence
$$1\rightarrow\Inn(F_2)\rightarrow\Aut(F_2)\rightarrow\GL_2(\ZZ)\rightarrow 1$$
identifying $\GL_2(\ZZ)$ with $\Out(F_2)$. Then Theorem \ref{thm_outer_representation} states 

\begin{thm*}[Theorem \ref{thm_outer_representation}]  The representation $\rho_{\Qbar}$ is induced by the natural (outer) action of $\SL_2(\ZZ)\subset\Out(F_2)$ on $F_2$ sitting as a discrete dense subgroup of $\widehat{F_2}$.
\end{thm*}

Our strategy is first to consider the ``universal elliptic curve'' $\EE$ over the punctured $j$-line $B := \AA^1_\CC\setminus\{0,1728\}$. Let $j_0\in B$, then we can compute the monodromy action of $\pi_1(B,j_0)$ on $\pi_1(\EE^\circ_{j_0})$ by computing its homological invariant, and using results of Kodaira to see that $\pi_1(B)$ acts ``through $\SL_2(\ZZ)$''. Then, we show that the map $B\stackrel{\EE/B}{\longrightarrow}\mM(1)_{\Qbar}$ induces a surjection on (\'{e}tale) fundamental groups $\pi_1(B_{\Qbar})\twoheadrightarrow\pi_1(\mM(1)_{\Qbar})$, which proves the theorem. 

\subsubsection*{Topological monodromy}

Let $\ol{S}$ be a compact Riemann surface and $S := \ol{S} - \{a_1,\ldots,a_r\}$ where $a_i\in\ol{S}$, and $E/S$ be a complex analytic elliptic fibration\footnote{by this we mean that $E\rightarrow S$ is a surjective holomorphic map with every fiber a smooth connected curve of genus 1} with holomorphic ``identity'' section $e : S\rightarrow E$, and another section $g : S\rightarrow E$. Any such elliptic fibration with section $e$ is projective, hence algebraic. Let $E^\circ := E - e(S)$. Let $s\in S$ be a point. We have the split homotopy exact sequence of topological fundamental groups

$$\xymatrix{
1\ar[r] & \pi_1^\tp(E_s^\circ,g(s))\ar[r] &  \pi_1^\tp(E^\circ,g(s))\ar[r] & \pi_1^\tp(S,s)\ar[r]\ar@/^-1pc/[l]_{g_*} & 1
}$$

Let $\gamma$ be a loop in $S$ based at $s$, and let $\alpha\in\pi_1^\tp(E_{s}^\circ,g(s))$, then we may ``transport'' $\alpha$ along $g\circ\gamma$ in a unique way up to homotopy, such that at $\gamma(1)$ we have the loop $\alpha^\gamma := g(\gamma)\alpha g(\gamma)^{-1}\in\pi_1^\tp(E^\circ_{s},g(s))$. Thus, we see that this monodromy action is precisely the action of $\pi_1^\tp(S,s)$ on $\pi_1^\tp(E^\circ_s,g(s))$ by conjugation inside $\pi_1^\tp(E^\circ,g(s))$ via $g_*$.

\gap

In the same way we have a monodromy action of $\pi_1^\tp(S,s)$ acting on $H^1(E^\circ_s,\ZZ)$, which is compatible with the natural abelianization map $F_2\cong\pi_1^\tp(E^\circ_s,g(s))\rightiso H^1(E^\circ_s,\ZZ)\cong\ZZ^2$ sending homotopy classes of loops to their homology classes. The inclusion $E^\circ_s\hookrightarrow E_s$ induces an isomorphism on homology. Thus, if we fix a basis $\alpha,\beta$ for $\pi_1^\tp(E^\circ,g(s))$, whose homology classes give a positively oriented basis for $H^1(E^\circ_s,\ZZ)$ and $H^1(E_s,\ZZ)$, we may identify $\pi_1^\tp(E^\circ_s,g(s))$ with the free group $F_2$ on generators $\alpha,\beta$, and the homology groups with $\ZZ^2$ with generators the homology classes of $\alpha,\beta$. Let $\rho : \pi_1^\tp(S,s)\rightarrow\Aut\big(\pi_1^\tp(E^\circ_s,g(s))\big)\cong \Aut(F_2)$ be the representation induced by $g_*$. We have a commutative diagram

\begin{equation}\label{eq_actions}
\xymatrix{
\pi_1^\tp(S,s)\ar[r]^{\rho} & \Aut(F_2)\ar@{->>}[d]\ar[r]^{\ab\qquad} & \Aut(H^1(E^\circ_s,\ZZ))\ar[d]^\sim\ar[r]^\sim & \Aut(H^1(E_s,\ZZ))\ar[dl]^\sim \\
 & \Out(F_2)\ar[r]^\sim & \GL_2(\ZZ)
}\end{equation}

The composition of the three maps on the top row is classically referred to as the \emph{homological invariant} of the elliptic fibration $E/S$, which is orientation preserving and hence the image in $\Aut(H^1(E_s,\ZZ))\cong\GL_2(\ZZ)$ is contained in $\SL_2(\ZZ)$. Thus, we may think of the homological invariant as a map
\begin{equation}\label{eq_homological_invariant}
h : \pi_1^\tp(S,s)\stackrel{\rho}{\longrightarrow}\Aut^+(F_2)\stackrel{u}{\longrightarrow}\SL_2(\ZZ)
\end{equation}
where $\Aut^+(F_2)$ is the preimage of $\SL_2(\ZZ)\subset\Out(F_2)$, and $u$ is just ``quotient-by-$\Inn(F_2)$''. This $h$ is only defined up to conjugation in $\SL_2(\ZZ)$ (since it depends on the choice of positively oriented basis for $H^1(E_s,\ZZ)$). In other words, the homological invariant is just the outer action associated to the monodromy action of $\pi_1^\tp(S,s)$ on $\pi_1(E^\circ_s,g(s))$ via $g_*$.

\sgap



Let $\hH^\circ$ denote the upper half plane with all elliptic points (ie, $\PSL_2(\ZZ)$-orbits of $i$ and $e^{2\pi i/3}$) removed. The $j$-invariant map $\hH^\circ\rightarrow\AA^1_\CC\setminus\{0,1728\}$ is then a Galois unramified cover with Galois group $\PSL_2(\ZZ)$. Let $B := \AA^1_\CC\setminus\{0,1728\}$, then we have a natural surjection $\pi_1^\tp(B)\twoheadrightarrow\PSL_2(\ZZ)$. For any elliptic fibration $E/S$ with no fibers of $j$-invariant 0 or 1728, the $j$-invariant gives a map $j : S\rightarrow B$ (sometimes called the \emph{functional invariant} of $E/S$), and hence a map
\begin{equation}\label{eq_functional_invariant}
\pi_1^\tp(S)\stackrel{j_*}{\longrightarrow}\pi_1^\tp(B)\twoheadrightarrow\PSL_2(\ZZ)
\end{equation}

\begin{thm}\label{thm_h_belongs_to_j} (Kodaira) The homological invariant ``$h$'' (\ref{eq_homological_invariant}) ``\emph{belongs}'' to $j_*$. Ie, fixing a base point $s\in S$, the following diagram commutes up to conjugacy in $\SL_2(\ZZ)$.
$$\xymatrix{
 & & \SL_2(\ZZ)\ar[d] \\
\pi_1^\tp(S,s)\ar[rru]^{h = u\circ\rho}\ar[r]_{j_*} & \pi_1^\tp(B,j(s))\ar@{->>}[r] & \PSL_2(\ZZ) 
}$$
\end{thm}
\begin{proof} See Theorem 7.2 of \cite{KoCAS2} (also the discussion before Definition 8.1). 
\end{proof}

As an application of Theorem \ref{thm_h_belongs_to_j} where $S = B$, we get

\begin{prop}\label{prop_homological_invariant_on_universal_elliptic_curve} Let $\EE$ be the elliptic curve over $B := \AA^1_\CC\setminus\{0,1728\}$ with coordinate $j$ given by
\begin{equation}\label{eq_UEC}
y^2 + xy = x^3 - \frac{36}{j-1728}x - \frac{1}{j-1728}
\end{equation}
Fix a base point $j_0\in B$, and let $\EE^\circ := \EE - O$, then $\EE^\circ$ admits a section $g$. The homological invariant $h : \pi_1(B,j_0)\longrightarrow \SL_2(\ZZ)$ is surjective and agrees with the outer representation of $\pi_1(B,j_0)$ on $\pi_1(\EE^\circ_{j_0},g(j_0))$ obtained from the homotopy exact sequence.
\end{prop}
\begin{proof} For the section, we may set $g = \left(-\frac{1}{36},y_0\right)$, where $y_0$ is a root of $y^2 - \frac{1}{36}y + \frac{1}{36}$. Using standard formulas we may calculate that the $j$-invariant of the fiber $\EE_j$ over any $j\in B$ is $j$. Thus, in the situation of Theorem \ref{thm_h_belongs_to_j}, the map $\pi_1^\tp(B,j_0)\longrightarrow\PSL_2(\ZZ)$ is surjective. 
Then, the commutativity of the diagram in Theorem \ref{thm_h_belongs_to_j} tells us that the homological invariant $h : \pi_1(B)\longrightarrow\SL_2(\ZZ)$ must have image a subgroup of $\SL_2(\ZZ)$ which surjects onto $\PSL_2(\ZZ)$. Thus, the image must contain either $\spmatrix{0}{1}{-1}{0}$ or $\spmatrix{0}{-1}{1}{0}$. Whichever it contains, the image must contain $-I$, but a subgroup of $\SL_2(\ZZ)$ containing $-I$ that surjects onto $\PSL_2(\ZZ)$ must be all of $\SL_2(\ZZ)$.

\end{proof}



\subsubsection*{Algebraic monodromy}

If we view $\EE/B$ as an algebraic curve over the Riemann surface $B$, the homotopy exact sequence of \'{e}tale fundamental groups becomes
$$\xymatrix{
1\ar[r] & \widehat{\pi_1^\tp(\EE^\circ_{j_0},g(j_0))}\ar[r] & \widehat{\pi_1^\tp(\EE^\circ,g(j_0))}\ar[r] & \widehat{\pi_1^\tp(B,j_0)}\ar[r]\ar@/^1pc/[l]^{g_*} & 1
}$$
where the maps are just the profinite completions of the maps in the corresponding exact sequence of topological fundamental groups. Since everything is defined over $\Qbar$, the exact sequence above holds upon descending to $\Qbar$. From now on everything will be considered over $\Qbar$, so $\hH/\Gamma$ will refer to a $\Qbar$-model of the Riemann surface $\hH/\Gamma$. Note that because it admits a map to the $j$-line ramified only above $j = 0,1728$ and $i\infty$, such a $\Qbar$-model exists by Belyi's theorem.

\sgap

Proposition \ref{prop_homological_invariant_on_universal_elliptic_curve} showed us that the outer monodromy action of $\pi_1(B,j_0)$ on $\pi_1(\EE^\circ_{j_0},g(j_0))$ is through $\SL_2(\ZZ)$ acting on a basis for $\pi_1^\tp(\EE^\circ_{j_0},g(j_0))\subset \widehat{\pi_1^\tp(\EE^\circ_{j_0},g(j_0))}$. Thus, to prove Theorem \ref{thm_outer_representation}, it will suffice to show that the map $(f_{\EE})_* : \pi_1(B)\rightarrow\pi_1(\mM(1)_{\Qbar})$ induced by $\EE/B$ is surjective. 



\begin{prop}\label{prop_monodromy} Let $f_\EE : B\rightarrow \mM(1)$ be the morphism corresponding to the elliptic curve $\EE/B$ of Proposition \ref{prop_homological_invariant_on_universal_elliptic_curve}. Then $(f_\EE)_* : \pi_1(B)\rightarrow\pi_1(\mM(1))$ is surjective.
\end{prop}
\begin{proof} In general, given a map of schemes/stacks $a : X\rightarrow Y$, the induced map $a_*$ of \'{e}tale fundamental groups is surjective if and only if for every connected finite \'{e}tale cover $C/Y$, the pullback $a^*C$ is also connected. Let $\mM\rightarrow \mM(1)$ be a connected finite \'{e}tale cover, then for some $n\ge 3$, let $\mM'$ be a connected component of the stack $\mM\times_{\mM(1)}\mM(n)$, so that $\mM'$ is a connected scheme finite \'{e}tale over both $\mM$ and $\mM(n)$, so $\mM'$ is smooth, hence irreducible. Since we are working over $\Qbar$, by Proposition \ref{prop_basic_results_cms}(4), $f_\EE$ is, at the level of topological spaces, an inclusion of an open subset, and hence the underlying topological space of $f_\EE^*\mM'$ can be identified with an open subset of the topological space of $\mM'$, so $f_\EE^*\mM'$ is also irreducible, hence connected, which implies the connectedness of $f_\EE^*\mM$.
\end{proof}

\section{Tables of noncongruence modular curves}\label{appendix_tables}

In \S\ref{section_examples}, we discussed a number of examples of noncongruence modular curves which appeared as components of $\mM(G)_{\Qbar}$ for various $G$. In this appendix we list some additional examples. For a description of what the data in each column represents, see \S\ref{section_first_examples}.

\subsection{Noncongruence components of $\mM(G)_{\Qbar}$ for groups of order $\le 255$}\label{section_NC_1-255}

Using code written in GAP\footnote{\url{www.gap-system.org}}, we have computed data for every component of $\mM(G)_{\Qbar}$ as $G$ ranges over all nonabelian 2-generated finite groups of order $\le 255$. Of the 2036 such groups, 218 were purely noncongruence (c.f. Definition \ref{def_purity}), and the other 1818 were purely congruence. The table in this section lists data for the components of $\mM(G)_{\Qbar}$ for the 218 purely noncongruence groups.

\sgap

Compared to the tables in \S\ref{section_examples}, there are two additional columns, ``$i$'' and ``SL''. Here, $i$ is the index of the group $G$ in GAP's Small Groups Library. The data in the columns ``Size'' and $i$ together completely determine the isomorphism class of $G$. If $G$ has order $n$, then $G$ can be accessed in GAP using the command
$$\texttt{SmallGroup(n,i)};$$

The data here is ordered first by Size, then by $i$, then by genus, and then $d$.

\sgap

The field ``SL'' stands for ``Solvable Length'' of $G$. If $G' =: G^{(1)}$ denotes the commutator subgroup of $G$, then the solvable length of $G$ is the least element of the set $\{n\in\NN : G^{(n)} = 1\}$. If $G$ is solvable with solvable length $n$, then the derived series
$$G\rhd G' \rhd G'' \rhd \cdots \rhd G^{(n)} = 1$$
is the shortest descending normal series with all successive quotients abelian. If $G$ is not solvable, then its solvable length is $\infty$. Thus, we may think of the solvable length of $G$ as a measure of how nonabelian $G$ is.

\sgap

The column ``$G$'' is a formatted version of GAP's \texttt{StructureDescription(G)}. Here, $C_n, A_n, S_n, D_n, Q_n$ and $QD_n$ refer to the cyclic, alternating, symmetric, dihedral, quaternion, and quasidihedral group of order $n$. The conjunctions ``$\times$'', ``$\rtimes$'', ``$\cdot$'' denote ``direct product'', ``semidirect product'', and ``non-split extension''. In the latter two cases, the normal subgroup is always on the left side of the conjunction. In some cases, for two \texttt{StructureDescription}'s $A,B$, the syntax ``$A = B$'' denotes two equivalent ways of describing the group (e.g., the third row of the table shows ``$C_{2}\cdot S_{4}=\SL_{2}(\FF_{3})\cdot C_{2}$''). This is for convenience only, and by no means represents \emph{all} the ways of describing the group. Also note that nonisomorphic groups may have the same \texttt{StructureDescription}.

\begin{remark} Note that every group in the list has derived length $\ge 3$. Furthermore, each of the 1818 purely congruence $G$ had solvable length $\le 3$. If we take the solvable length to be a measure of nonabelian-ness, then this supports the notion that being noncongruence is connected to being ``sufficiently nonabelian''.

\end{remark}

\begin{center}\scriptsize
\begin{longtable}{m{0.15cm}m{0.25cm}m{4.66cm}m{0.1cm}m{0.06cm}m{0.23cm}m{0.1cm}m{0.1cm}m{0.1cm}m{2.9cm}m{0.21cm}m{0.32cm}m{0.28cm}m{0.15cm}m{0.15cm}}
\caption[Components of $\mM(G)_{\Qbar}$ for Noncongruence Groups $G$ of Order $\le 255$]{Noncongruence Groups of Order $\le 255$} \label{table_NC_1-255} \\


\text{Size} & $\;i$ & $G$ & SL & $m$ & $\;d$ & $c_4$ & $c_6$ & $c_{-1}$ & $\;$\text{cusp widths} & $\!\!\!\!$\text{genus} & \text{c/nc} & \text{c/f} & $e$ & $g$ \\
\hline
\endfirsthead

\text{Size} & $\;i$ & $G$ & SL & $m$ & $\;d$ & $c_4$ & $c_6$ & $c_{-1}$ & $\;$\text{cusp widths} & $\!\!\!\!$\text{genus} & \text{c/nc} & \text{c/f} & $e$ & $g$ \\
\hline \\ \multicolumn{13}{c}{...Continued from previous page} \\  \\

\endhead

\\ \multicolumn{13}{c}{Continued on next page...} \\
\endfoot

\hline \hline
\endlastfoot

\noalign{\vskip 1.3mm} 24 & 3 & $\SL_{2}(\FF_{3})$ & 3 & 1 & 32 & 0 & 1 & 0 & $3^{2} 4^{1} 6^{1} $ & 0 & \text{ncng} & \text{crse} & 4 & 10\\
\noalign{\vskip 1.3mm} 24 & 12 & $S_{4}$ & 3 & 1 & 9 & 1 & 0 & 1 & $2^{1} 3^{1} 4^{1} $ & 0 & \text{ncng} & \text{crse} & 3 & 9\\
\noalign{\vskip 1.3mm} 48 & 28 & $C_{2}\cdot S_{4}=\SL_{2}(\FF_{3})\cdot C_{2}$ & 4 & 1 & 36 & 2 & 0 & 1 & $3^{2} 4^{2} 6^{1} 8^{2} $ & 0 & \text{ncng} & \text{crse} & 6 & 21\\
\noalign{\vskip 1.3mm} 48 & 29 & $\GL_{2}(\FF_{3})$ & 4 & 1 & 36 & 0 & 0 & 0 & $2^{2} 3^{2} 8^{1} $ & 0 & \text{ncng} & \text{fine} & 6 & 21\\
\noalign{\vskip 1.3mm} 48 & 30 & $A_{4}\rtimes C_{4}$ & 3 & 1 & 36 & 0 & 0 & 0 & $3^{2} 4^{3} $ & 0 & \text{ncng} & \text{fine} & 3 & 17\\
\noalign{\vskip 1.3mm} 48 & 32 & $C_{2}\times \SL_{2}(\FF_{3})$ & 3 & 1 & 96 & 0 & 0 & 0 & $3^{2} 4^{3} 6^{5} $ & 0 & \text{ncng} & \text{fine} & 4 & 19\\
\noalign{\vskip 1.3mm} 48 & 33 & $\SL_{2}(\FF_{3})\rtimes C_{2}$ & 3 & 1 & 96 & 0 & 0 & 0 & $2^{4} 3^{2} 4^{1} 6^{1} 12^{2} $ & 0 & \text{ncng} & \text{fine} & 4 & 19\\
\noalign{\vskip 1.3mm} 48 & 48 & $C_{2}\times S_{4}$ & 3 & 1 & 18 & 0 & 0 & 1 & $2^{2} 4^{2} 6^{1} $ & 0 & \text{ncng} & \text{crse} & 3 & 17\\
\end{longtable}
\end{center}

\subsection{Components of $\mM(G)_{\Qbar}$ for nonabelian simple groups of order $\le 29120$}\label{section_23_FSGs}
The data in this section is sorted first by Size, then by the genus of the modular curve, and then by $d$. As expected, all components are noncongruence. 
\begin{center}\footnotesize
\begin{longtable}{llllrlllllllll}
\caption[Components of $\mM(G)_{\Qbar}$ for Nonabelian Simple Groups $G$ of Order $\le 29120$]{First 23 Nonabelian Finite Simple Groups} \label{table_23_FSGs} \\


\text{Size} & $G$ & $m$ & $d$ & $c_4$ & $c_6$ & $c_{-1}$ & \text{cusp widths} & \text{genus} & \text{c/nc} & \text{c/f} & $e$ & $g$ \\
\hline
\endfirsthead

\text{Size} & $G$ & $m$ & $d$ & $c_4$ & $c_6$ & $c_{-1}$ & \text{cusp widths} & \text{genus} & \text{c/nc} & \text{c/f} & $e$ & $g$ \\
\hline \\ \multicolumn{13}{c}{...continued from previous page.} \\  \\

\endhead

\\ \multicolumn{13}{c}{Continued on next page...} \\
\endfoot

\hline \hline
\endlastfoot
60 & $A_5$ & 2 & 10 & 0 & 1 & 1 & $2^{1}3^{1}5^{1}$ & 0 & \text{ncng} & \text{crse} & 5 & 25\\
60 & $A_5$ & 1 & 18 & 0 & 0 & 1 & $2^{1}3^{2}5^{2}$ & 0 & \text{ncng} & \text{crse} & 3 & 21\\
\\
168 & $\PSL_2(\FF_{7})$ & 2 & 7 & 1 & 1 & 1 & $3^{1}4^{1}$ & 0 & \text{ncng} & \text{crse} & 7 & 73\\
168 & $\PSL_2(\FF_{7})$ & 1 & 32 & 0 & 1 & 0 & $2^{1}3^{1}4^{1}7^{1}$ & 0 & \text{ncng} & \text{crse} & 4 & 64\\
168 & $\PSL_2(\FF_{7})$ & 1 & 32 & 0 & 1 & 0 & $2^{1}3^{1}4^{1}7^{1}$ & 0 & \text{ncng} & \text{crse} & 4 & 64\\
168 & $\PSL_2(\FF_{7})$ & 1 & 36 & 0 & 0 & 0 & $1^{1}3^{2}4^{1}7^{1}$ & 0 & \text{ncng} & \text{fine} & 3 & 57\\
\\
\end{longtable}
\end{center}

\end{appendix}

\bibliography{references}

\end{document}

%% file: mathbase_paper.tex
\input{mathbase.tex}

\theoremstyle{definition}\newtheorem{defn}{Definition}[subsection]
\theoremstyle{remark}
\theoremstyle{remark}
\theoremstyle{remark}
\theoremstyle{remark}\newtheorem{remark}[defn]{Remark}
\theoremstyle{remark}\newtheorem*{remark*}{Remark}
\theoremstyle{remark}
\theoremstyle{remark}\newtheorem{pt}[defn]{}
\theoremstyle{remark}\newtheorem{example}[defn]{Example}

\theoremstyle{plain}\newtheorem{prop}[defn]{Proposition}
\theoremstyle{plain}\newtheorem{thm}[defn]{Theorem}
\theoremstyle{plain}\newtheorem*{thm*}{Theorem}
\theoremstyle{plain}\newtheorem{lemma}[defn]{Lemma}
\theoremstyle{plain}\newtheorem{cor}[defn]{Corollary}
\theoremstyle{plain}\newtheorem{conj}[defn]{Conjecture}
\theoremstyle{plain}\newtheorem*{conj*}{Conjecture}
\theoremstyle{plain}\newtheorem*{prop*}{Proposition}

%% file: mathbase.tex
\usepackage{comment}
\usepackage{amsfonts}
\usepackage{amssymb}
\usepackage{amsthm}
\usepackage{amsmath}

\usepackage{cite}
\usepackage{amsrefs}

\usepackage{longtable}

\usepackage[all,cmtip]{xy}

\usepackage{fancyhdr}

\usepackage{graphicx}
\usepackage{hyperref}

\usepackage{tikz-cd}

\usepackage{aurical}

\usepackage{stmaryrd}

\usepackage[T1]{fontenc}

\usepackage[outline]{contour}
\usepackage{xcolor}

\usepackage{array}

\renewcommand{\AA}{\mathbb{A}}

\newcommand{\et}{{\acute{e}t}}

\newcommand{\UBD}{\text{UBD}}

\newcommand{\CC}{\mathbb{C}}

\newcommand{\EE}{\mathbb{E}} 
\newcommand{\FF}{\mathbb{F}}

\newcommand{\LL}{\mathbb{L}}
\newcommand{\NN}{\mathbb{N}}
\newcommand{\PP}{\mathbb{P}}
\newcommand{\QQ}{\mathbb{Q}}

\newcommand{\XX}{\mathbb{X}}
\newcommand{\ZZ}{\mathbb{Z}}

\newcommand{\Qbar}{\overline{\mathbb{Q}}}
\newcommand{\Zbar}{\overline{\mathbb{Z}}}

\newcommand{\bB}{\mathcal{B}}
\newcommand{\cC}{\mathcal{C}}

\newcommand{\eE}{\mathcal{E}}
\newcommand{\fF}{\mathcal{F}}
\newcommand{\gG}{\mathcal{G}}
\newcommand{\hH}{\mathcal{H}}

\newcommand{\mM}{\mathcal{M}}
\newcommand{\nN}{\mathcal{N}}

\newcommand{\pP}{\mathcal{P}}

\newcommand{\xX}{\mathcal{X}}
\newcommand{\yY}{\mathcal{Y}}

\newcommand{\Zhat}{\widehat{\mathbb{Z}}}
\newcommand{\Fhat}{\widehat{F_2}}

\newcommand{\G}{\textbf{G}}

\newcommand{\mf}[1]{\mathfrak{#1}} 

\newcommand{\fs}{\!\!\fatslash}

\newcommand{\lcm}{\textrm{lcm}}
\newcommand{\Aut}{\textrm{Aut}}

\newcommand{\Inn}{\textrm{Inn}}
\newcommand{\Out}{\text{Out}}

\newcommand {\spmatrix}[4]{\left[\begin{smallmatrix}#1&#2\\#3&#4\end{smallmatrix}\right]}

\newcommand{\ch}{\textrm{char}}

\newcommand{\Frac}{\textrm{Frac}}

\newcommand{\Gal}{\textrm{Gal}}
\newcommand{\id}{\textrm{id}}
\newcommand{\Hom}{\textrm{Hom}}
\newcommand{\hHom}{\mathcal{H}om}

\newcommand{\Surj}{\text{Surj}}

\newcommand{\GL}{\textrm{GL}}
\newcommand{\SL}{\textrm{SL}}
\newcommand{\PSL}{\textrm{PSL}}

\newcommand{\PSU}{\text{PSU}}

\newcommand{\Sz}{\text{Sz}}

\newcommand{\Stab}{\textrm{Stab}}

\newcommand{\ab}{\text{ab}}

\newcommand{\sep}{\textrm{sep}}

\newcommand{\Inv}{\text{Inv}}

\newcommand{\rightiso}{\stackrel{\sim}{\longrightarrow}}

\newcommand{\ol}[1]{\overline{#1}}

\newcommand{\an}{\textrm{an}}

\newcommand{\h}{\text{h}} 
\newcommand{\sh}{\text{sh}} 

\newcommand{\Spec}{\textrm{Spec }}

\newcommand{\prim}{\textrm{prim }}

\newcommand{\ext}{\text{ext}}
\newcommand{\sur}{\text{sur}}

\newcommand{\surext}{\text{sur-ext}}

\newcommand{\ls}[1]{(\!(#1)\!)}
\newcommand{\ps}[1]{[\![#1]\!]}

\newcommand{\tp}{\text{top}}

\newcommand{\can}{\text{can}}
\newcommand{\Tate}{\text{Tate}}

\newcommand{\Sch}{\underline{\textbf{Sch}}}
\newcommand{\Sets}{\underline{\textbf{Sets}}}
\newcommand{\FiniteSets}{\underline{\textbf{FiniteSets}}}
\newcommand{\ProfiniteSets}{\underline{\textbf{ProfiniteSets}}}

\newcommand{\Top}{\underline{\textbf{Top}}}

\newcommand{\Pro}{\text{Pro-}}
\newcommand{\FEt}{\text{FEt}}
\newcommand{\PFEt}{\text{PFEt}}
\newcommand{\Ob}{\text{Ob}}

\newcommand{\gap}{\vspace{0.4cm}}
\newcommand{\sgap}{\vspace{0.2cm}}